\documentclass[10pt,reqno]{amsart}
\usepackage{amsmath}
\usepackage{amsfonts}
\usepackage{amsthm}
\usepackage{kantlipsum}
\usepackage{enumitem}
\usepackage{caption}
\usepackage{breqn}
\usepackage{amssymb}
\usepackage{tikz}
\usepackage{mathtools}
\usetikzlibrary{external}
\usetikzlibrary{patterns}
\usetikzlibrary{positioning}
\usetikzlibrary{calc,shapes.geometric}
\usepackage{caption}
\usepackage{multicol}
\usepackage[utf8]{inputenc}
\usepackage{csquotes}
\usepackage{subcaption}
\usepackage{tkz-euclide}
\usepackage{amssymb}
\usepackage{pgf,tikz,pgfplots}
\pgfplotsset{compat=1.14}
\usepackage{pgf,tikz,pgfplots}
\usepackage{mathrsfs}
\usetikzlibrary{arrows}
\usepackage{mathrsfs}
\usetikzlibrary{arrows}
\usepackage{capt-of}
\usetikzlibrary{decorations.pathreplacing}
\usepackage{minipage-marginpar}
\usetikzlibrary{cd}
\usetikzlibrary{positioning}
\usetikzlibrary{calc}
\usepackage{algorithmic}
\usepackage{float}
\usetikzlibrary{patterns}
\calclayout

\usepackage[]{hyperref}
\hypersetup{
  colorlinks   = true,          
  urlcolor     = blue,          
  linkcolor    = red,          
  citecolor   = orange             
}

\theoremstyle{plain}
\newtheorem{theorem}{Theorem}
\newtheorem{corollary}{Corollary}
\newtheorem{proposition}{Proposition}

\newtheorem{lemma}{Lemma}

\theoremstyle{remark}
\newtheorem{remark}{Remark}

\theoremstyle{definition}

\newtheorem{definition}{Definition}

\newtheorem{exmp}{Example}

\numberwithin{equation}{subsection}

\usepackage[utf8]{inputenc}

\title{Point-Line Geometry in the Tropical Plane}
\author{ayush kumar tewari}

\address{Technische Universität Berlin,
  Chair of Discrete Mathematics/Geometry}
\email{tewari@math.tu-berlin.de}

\makeatletter
\@namedef{subjclassname@1991}{Mathematics Subject Classification (2020)}
\makeatother

\subjclass{14T15, 52C35}

\keywords{stable tropical lines, semiuniform faces, tropical near pencil}

\begin{document}
\definecolor{wrwrwr}{rgb}{0.3803921568627451,0.3803921568627451,0.3803921568627451}
\definecolor{rvwvcq}{rgb}{0.08235294117647059,0.396078431372549,0.7529411764705882}

\begin{abstract}
We study the classical result by Bruijn and Erd\H os regarding the bound on the number of lines determined by a $n$-point configuration in the plane, and in the light of the recently proven Tropical Sylvester-Gallai theorem, come up with a tropical version of the above-mentioned result. In this work, we introduce stable tropical lines, which help in answering questions pertaining to incidence geometry in the tropical plane. Projective duality in the tropical plane helps in translating the question for stable lines to stable intersections that have been previously studied in depth. Invoking duality between Newton subdivisions and line arrangements, we are able to classify stable intersections with shapes of cells in subdivisions, and this ultimately helps us in coming up with a bound. In this process, we also encounter various unique properties of linear Newton subdivisions which are dual to tropical line arrangements.
\end{abstract}

\maketitle

\section{Introduction}

Point-line geometry has been studied for a long time, and it mainly deals with the question of \textbf{incidence}, i.e.\ when a point meets a line. There are many classical results established about the incidence of points and lines in projective and affine planes like the Sylvester-Gallai theorem, de-Bruijn Erd\H os theorem, Szemeredi-Trotter theorem, Beck's theorem etc. In recent times, there has been a lot of development in generalising these classical results, like \cite{JH17} surveys the work done on generalizations of de-Bruijn Erd\H os theorem. Likewise in \cite{FZ18}, ordinary lines in three space are studied and results regarding spanned lines and Langer's inequality are established. In a recent study in \cite{JMRS20}, tropical lines present in a fixed plane are also studied.

Since tropical geometry provides a piecewise linear model of point line geometry, many incidence geometric results have also been proved in it. In \cite{DR17} a tropical version of Sylvester-Gallai theorem and Motzkin-Rabin theorem is established along with the universality theorem. In \cite{LT11} the term \textbf{geometric construction} is coined ,  in order to identify all the types of classical incidence geometric results which can have a tropical analogue. Even in \cite{BS05} and \cite{LT05} a tropical version of Pappus theorem is discussed along with classical point-line configurations. Another aspect is the relation to oriented matroids, and as mentioned in \cite{DR17}, it is elaborated in \cite{AF09}, in the context of hyperplane arrangements and how they correspond to tropical oriented matroids and how these matroids encode incidence information about point-line structures in the tropical plane. The fact that the tropical plane allows tropical duality, felicitates much of the above mentioned results. 

In this article, we start with some basic notions of point line geometry and specifically the point-line geometry in the tropical plane. Subsequently, using the results obtained in \cite{DR17} and by introducing the notion of stable tropical lines we state a tropical counterpart to de-Bruijn-Erd\H os theorem. We also establish the equivalence between a much general notion of stability for curves, in \cite{LT11}, and the stable lines that we define in our work. We find that tropicalization of generic lifts of points determines the stable tropical line passing through them. We establish the duality between stable lines and stable intersections and provide a full classification of the faces that they correspond to in the dual Newton subdivision. With this setup, we prove the following tropical analogue of de-Bruijn-Erd\H os theorem, 

\begin{theorem}[Tropical de-Bruijn-Erd\H os Theorem]\label{thm:erdos}
Let $\mathcal{S}$ denote a set of points in the tropical plane. Let $v$ $(v \geq 4)$ denote the number of points in $\mathcal{S}$, and let $b$ denote the number of stable tropical lines determined by these points. Then,

\begin{enumerate}
    \item $b \geq v - 3$
    \item if $b = v - 3$, then $\mathcal{S}$ forms a \textbf{tropical near-pencil}.
\end{enumerate}
\end{theorem}

The definitions and the results required to state and prove the above result are elaborated in the latter part of the article.

\textbf{Acknowledgements} - I am sincerely thankful to Hannah Markwig who had regular discussions with me and went through earlier drafts of this work and gave concrete suggestions which immensely helped in this piece of work. I would also like to thank Michael Joswig, Marta Panizzut, Dhruv Ranganathan, Yue Ren for fruitful conversations and guidance during the time I was working on this problem. This research is supported by the Deutsche Forschungsgemeinschaft (SFB-TRR 195 \enquote{Symbolic Tools in Mathematics and their Application}). I would also like to thank the Mittag-Leffler Institute which hosted me for the semester program \enquote{Tropical Geometry, Amoebas and Polytopes} where a significant part of the work done on this article was carried out.

\section{Classical Incidence Geometry}

In classical incidence geometry a \textbf{linear space} is defined in the following manner \cite{PR83},

\begin{definition}\label{thm: tropical-LS}
A \emph{finite linear space} is a pair $(X,\mathcal{B})$, where $X$ is a finite set and $\mathcal{B}$ is a set of proper subsets of $X$, such that 
\begin{enumerate}
    \item every unordered pair of elements of $X$ occur in a unique $B$ $ \in \mathcal{B}.$
    \item Every $B$ $\in \mathcal{B}$ has cardinality at least two.
\end{enumerate}
\end{definition}

Essentially, a linear space is a point-line incidence structure, in which any two points lie on a unique line.

\begin{exmp}
Consider $L = (X, \mathcal{B})$, where $X$ is the set of points in the Euclidean plane and $\mathcal{B}$ is the set of lines determined by $X$  .
\end{exmp}

Another important definition about lines is,

\begin{definition}
A line which passes through exactly two points is called an \emph{ordinary} line.
\end{definition}

Erd\H os and de-Bruijn, came up with a theorem about point-line arrangements in a linear space \cite{PB48}, which is established in the following manner \cite{LB97},

\begin{theorem}[de-Bruijn-Erd\H os Theorem]\label{thm: tropical-BE}
Let $S = (X,\mathcal{B})$ be a linear space. Let $v$ denote the number of points in $S (= |X|$), and $b$ denote the number of lines determined by these points $(= |\mathcal{B}|$, $b > 1)$. Then
\begin{enumerate}
    \item $b \geq v$,
    \item if $b = v$, any two lines have a point in common. In case (2), either one line has $v - 1$ points and all others have two points, or every line has $k + 1$ points and every point is on $k + 1$ lines, $k \geq 2$. 
\end{enumerate}
\end{theorem}
 
For a more general treatment and recent developments, one can read \cite{JH17}, where enumerative results like the above have been discussed in a more  general setting of geometric lattices.

The above result is clearly a very general statement, and in the case for points and lines in the Euclidean plane, the bound on the number of lines is attained when points are in a near -pencil configuration and the proof follows by induction, by invoking the following result 

\begin{theorem}[Sylvester-Gallai Theorem]
Given a finite collection of points in the Euclidean plane, such that not all of them lie on one line, then there exists a line which passes through exactly two of the points.
\end{theorem}

\section{A Brief Introduction to Tropical Geometry}

Tropical geometry can be defined as the study of  geometry over the tropical semiring $\mathbb{T} = (\mathbb{R} \cup \{-\infty\}$, max, +). A \textbf{tropical polynomial} $p(x_{1}, \hdots , x_{n})$ is defined as a linear combination of tropical monomials with operations as the tropical addition and tropical multiplication.

\begin{center}
$p(x_{1}, \hdots , x_{n}) = a \otimes {x_{1}}^{i_{1}}{x_{2}}^{i_{2}} \hdots {x_{n}}^{i_{n}} \oplus b \otimes {x_{1}}^{j_{1}}{x_{2}}^{j_{2}} \hdots {x_{n}}^{j_{n}} \oplus \hdots $
\end{center}

With the above definitions, we see that a tropical polynomial is a function $p : \mathbb{R}^{n} \longrightarrow \mathbb{R}$ given by maximum of a finite set of linear functions.

\begin{definition}
The \emph{hypersurface} $V(p)$ of $p$ is is the set of all points $w$ $\in \mathbb{R}^{n}$ at which the maximum is attained at least twice. Equivalently, a point $w$ $\in \mathbb{R}^{n}$ lies in $V(p)$ if and only if $p$ is not linear at $w$.
\end{definition}

The tropical polynomial defining a tropical line is given as

\begin{center}
$p(x,y) =  a \otimes x \oplus b \otimes y \oplus c,$ where $a,b,c \in \mathbb{R}$,
\end{center}

and the corresponding hypersurface is the corner locus defined by the above polynomial,

\begin{figure}[H]
    \centering
    \includegraphics[scale=0.7]{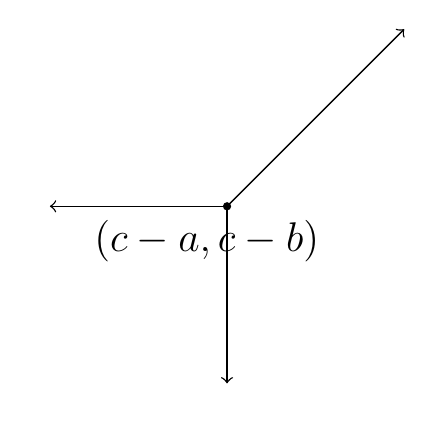}
    \caption{A tropical line}
    \label{fig:tropical_line}
\end{figure}

which is a collection of three half rays emanating from the point $(c - a, c - b)$ in the primitive directions of $(-1,0),(0,-1)$ and $(1,1)$ (Refer \cite{DS15}).

Now we look at the intersections of lines in the tropical plane. As is evident from the setup, tropical lines can intersect over a half ray. However, two tropical lines have a unique \emph{stable intersection}, where a stable intersection is the limit of points of intersection of nearby lines which have a unique point of intersection, within a suitable $\epsilon$, with the limit being taken as $\epsilon$ tends to 0 \cite{DS15}. We refer the reader to \cite{DS15} for further details about stable intersections in full generality. We also define the two types of stable intersections which we encounter in the case of line arrangements,

\begin{definition}
A stable intersection in a tropical line arrangement is called \emph{stable intersection of first kind} if  no vertex of any line from the line arrangement is present at the point of intersection.
\end{definition}

\begin{definition}
A stable intersection in a tropical line arrangement is called \emph{stable intersection of second kind} if the vertex of a line from the line arrangement is present at the point of intersection.
\end{definition}

An important observation is the \textbf{projective duality} which exists in the tropical plane \cite{DR17}, which means that given a set of points $\mathcal{P}$, there exists a incidence preserving map $\phi$ which maps $\mathcal{P}$ to its dual set of tropical lines $\mathcal{L}$, where for each point $P \in \mathcal{P}$, $\phi(P) = l$ with $-P$ as the vertex of the line $l \in \mathcal{L}$.

The \emph{support} of a tropical polynomial is the collection of the exponents of the monomials which have a finite coefficient. The convex hull of the exponents in the support of a tropical polynomial defines a \textbf{Newton polytope}. A subdivision of a set of points in $\mathbb{R}^{2}$, is a polytopal complex which covers the convex hull of the set of points and uses a subset of the point set as vertices. If such a subdivision of points is induced by a weight vector $c$, then it is called a \textbf{regular subdivision}. There exists a duality between a tropical curve $T$, defined by a tropical polynomial $p$, and the subdivision of the Newton polygon corresponding to $p$, induced by the coefficients of the tropical polynomial $p$. For further details about the description of this duality, the reader can refer to \cite[Chapter~3]{DS15} and \cite[Proposition 2.5]{GM15}. 

For a comprehensive study in a general setting, we analyze the underlying field $K$. A \textbf{valuation} on $K$ is a map $val : K \rightarrow \mathbb{R} \cup \{\infty\}$ such that it follows the following three axioms \cite{DS15}
\begin{enumerate}
    \item $val(a) = \infty$ if and only if $a=0$;
    \item $val(ab) = val(a) + val(b)$;
    \item $val(a + b) \geq min\{val(a),val(b)\}$ for all $a, b \in$ K.
\end{enumerate}

An important example of a field with a non-trivial valuation is the field of \textbf{Puiseux series} over a arbitrary field $k$, represented as $K = k\{\{t\}\}$. The elements in this field are formal power series 
\[ k(t) = k_{1}t^{a_{1}} + k_{2}t^{a_{2}} + k_{3}t^{a_{3}}  \hdots ,\]

where each $k_{i} \> \in \>  k, \> \forall \> i$ and $a_{1} < a_{2} < a_{3} < .... $ are rational numbers with a common denominator. This field has a natural valuation $val : k\{\{t\}\} \rightarrow \mathbb{R}$ given by taking a nonzero element $k(t) \in k\{\{t\}\}^{*}$, (where $k\{\{t\}\}^{*}$ represents the non zero element in the field $k\{\{t\}\}$) and mapping it to the lowest exponent $a_{1}$ in the series expansion of $k(t)$ \cite{DS15}.

It is an important observation that the valuation on the field of Puiseux series mimics the operations of a tropical semiring in essence and for further discussions one can think of the underlying field for the computations to be a Puiseux series with non-trivial valuation. So points which are considered in the plane, would have lifts residing in corresponding field of Puiseux series and the map which maps these lifts back to the points is the \emph{tropicalization} map. For a polynomial $f = \sum _{u \in \mathbb{N}^{n+1}} c_{u}x^{u}$, where the coefficients are from the field with a non-trivial valuation, the tropicalization of $f$ can be defined as  \cite{DS15}

\[ \text{trop}(f)(w) = \text{max}\{ -val(c_{u}) + w \cdot u : u  \in \mathbb{N}^{n+1} \> \text{and} \> c_{u} \ne 0 \}    \]

We refer the reader to \cite{DS15} for further details about this map. 

A \textbf{tropical line arrangement} is a finite collection of distinct tropical lines in $\mathbb{R}^{2}$. 

\begin{definition}\label{def:near_pencil_line}
A tropical line arrangement $\mathcal{L}$ is said to be a \textbf{tropical near-pencil arrangement} if in the dual Newton subdivision, for all triangular faces present in the subdivision; at least one of the edges of the triangular face lies on the boundary of the Newton polygon.
\end{definition}

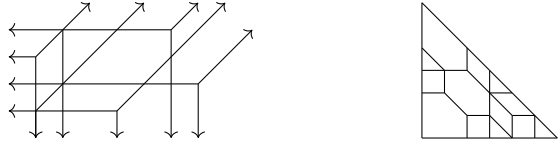
\begin{figure}
\centering
\begin{tikzpicture}[scale=0.18]
\draw[->][] (0,0) -- (2,2);
\draw[->][] (0,0) -- (0,-8);
\draw[->][] (0,0) -- (-4,0);
\draw[->][] (-2,-2) -- (-4,-2);
\draw[][] (-2,-2) -- (0,0);
\draw[->][] (-2,-2) -- (-2,-8);
\draw[->][] (-2,-6) -- (6,2);
\draw[][] (-2,-6) -- (-2,-8);
\draw[->][] (-2,-6) -- (-4,-6);
\draw[][] (-2,-6) -- (4,-6);
\draw[->][] (4,-6) -- (4,-8);
\draw[->][] (4,-6) -- (12,2);
\draw[->][] (10,-4) -- (10,-8);
\draw[->][] (10,-4) -- (-4,-4);
\draw[->][] (10,-4) -- (14,0);
\draw[->][] (8,0) -- (8,-8);
\draw[][] (8,0) -- (0,0);
\draw[->][] (8,0) -- (10,2);
\end{tikzpicture}
\hspace{2cm}
\begin{tikzpicture}[scale=0.3]
\draw[][] (0,0) -- (0,-2);
\draw[][] (0,-2) -- (1,-3);
\draw[][] (1,-3) -- (2,-3);
\draw[][] (0,-2) -- (0,-3);
\draw[][] (2,-3) -- (2,-2);
\draw[][] (2,-2) -- (0,0);
\draw[][] (1,-3) -- (0,-3);
\draw[][] (0,-3) -- (0,-4);
\draw[][] (1,-3) -- (1,-4);
\draw[][] (0,-4) -- (1,-4);
\draw[][] (1,-4) -- (2,-5);
\draw[][] (2,-5) -- (2,-6);
\draw[][] (2,-6) -- (0,-6);
\draw[][] (0,-4) -- (0,-6);
\draw[][] (2,-3) -- (3,-4);
\draw[][] (3,-4) -- (3,-5);
\draw[][] (3,-5) -- (2,-5);
\draw[][] (3,-5) -- (3,-6);
\draw[][] (2,-6) -- (3,-6);
\draw[][] (2,-2) -- (3,-3);
\draw[][] (3,-3) -- (3,-4);
\draw[][] (3,-4) -- (4,-4);
\draw[][] (3,-3) -- (4,-4);
\draw[][] (3,-4) -- (4,-5);
\draw[][] (4,-5) -- (4,-6);
\draw[][] (4,-6) -- (3,-5);
\draw[][] (3,-6) -- (4,-6);
\draw[][] (4,-5) -- (5,-5);
\draw[][] (4,-4) -- (5,-5);
\draw[][] (4,-6) -- (5,-6);
\draw[][] (5,-5) -- (5,-6);
\draw[][] (5,-5) -- (6,-6);
\draw[][] (5,-6) -- (6,-6);
\end{tikzpicture}
\caption{An example of a tropical near pencil arrangement; a tropical near pencil arrangement (left), and the corresponding dual subdivision (right)}
\label{fig:tropical_near_pencil_arrangement}
\end{figure}

\begin{definition}\label{def:near_pencil_point}
A set of points $\mathcal{N}$ in the tropical plane, is said to form a \textbf{tropical near-pencil} if the dual tropical line arrangement is a tropical near pencil arrangement.
\end{definition}

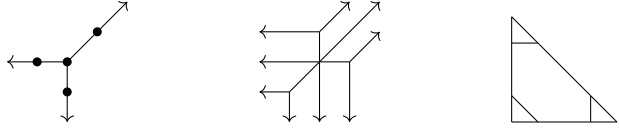
\begin{figure}
\centering
\begin{tikzpicture}[scale=0.20]
\draw[->][] (0,0) -- (-4,0);
\draw[->][] (0,0) -- (0,-4);
\draw[->][] (0,0) -- (4,4);
\fill[black]  (0,0) circle (.3cm);
\fill[black]  (0,-2) circle (.3cm);
\fill[black]  (-2,0) circle (.3cm);
\fill[black]  (2,2) circle (.3cm);
\end{tikzpicture}
\hspace{1.5cm}
\begin{tikzpicture}[scale=0.2]
\draw[][] (0,0) -- (0,2);
\draw[->][] (0,0) -- (-4,0);
\draw[->][] (0,2) -- (-4,2);
\draw[->][] (0,2) -- (2,4);
\draw[->][] (0,0) -- (4,4);
\draw[->][] (0,0) -- (0,-4);
\draw[][] (0,0) -- (-2,-2);
\draw[->][] (-2,-2) -- (-4,-2);
\draw[->][] (-2,-2) -- (-2,-4);
\draw[][] (0,0) -- (2,0);
\draw[->][] (2,0) -- (4,2);
\draw[->][] (2,0) -- (2,-4);
\end{tikzpicture}
\hspace{1.5cm}
\begin{tikzpicture}[scale=0.35]
\draw[][] (0,0) -- (0,-1);
\draw[][] (0,-1) -- (1,-1);
\draw[][] (0,0) -- (1,-1);
\draw[][] (0,-3) -- (0,-4);
\draw[][] (0,-4) -- (1,-4);
\draw[][] (0,-3) -- (1,-4);
\draw[][] (1,-4) -- (3,-4);
\draw[][] (3,-4) -- (4,-4);
\draw[][] (3,-3) -- (4,-4);
\draw[][] (3,-3) -- (1,-1);
\draw[][] (0,-1) -- (0,-3);
\draw[][] (3,-3) -- (3,-4);
\end{tikzpicture}
\caption{An example of a tropical near pencil; a point set with the stable tropical line (left), the dual tropical near pencil line arrangement (center), and the dual subdivision for the line arrangement (right)}
\label{fig:tropical_near_pencil}
\end{figure}

For a tropical line arrangement with lines $l_{1},\hdots,l_{n}$ with corresponding tropical polynomials being $f_{1},\hdots,f_{n}$ the tropical line arrangement, as a union of tropical hypersurfaces, is defined by the polynomial

\[ f = f_{1} \odot f_{2} \hdots  \odot f_{n} \]

The dual Newton subdivision corresponding to the tropical line arrangement is the Newton subdivision dual to the tropical hypersurface defined by the tropical polynomial $f$ (cf. \cite{MJ00}). We realize that stable intersections of first kind correspond to parallelograms and hexagons in the dual Newton subdivision and stable intersections of second kind correspond to irregular cells with four, five or six edges in the dual Newton subdivision.

For an elaborate description of dual Newton subdivisions, corresponding to tropical line arrangements, the reader is advised to refer to \cite[Section 2.3]{DR17}.

\section{Tropical Incidence Geometry}
The behaviour of point-line structures in the tropical plane is distinct from the Euclidean case, specifically with the appearance of \textbf{coaxial} points.

\begin{definition}
Two points are said to be coaxial if they lie on the same axis of a tropical line containing them \cite{DR17}.
\end{definition}

\begin{figure}[H]
    \centering
    \includegraphics[scale=0.65]{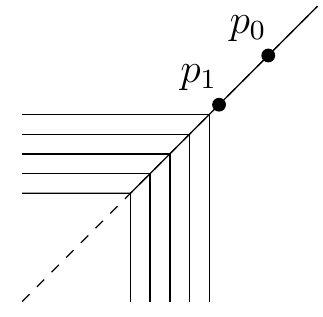}
    \caption{The infinite number of lines passing through the coaxial points $p_{0}$ and $p_{1}$}
    \label{fig:coaxial_points}
\end{figure}

\begin{definition}
Two lines are said to be coaxial if their vertices are coaxial.
\end{definition}

A recent result \cite{DR17} proves the tropical version of the Sylvester Gallai Theorem,  

\begin{theorem}[Tropical Sylvester-Gallai]
Any set of four or more points in the tropical plane determines at least one ordinary tropical line.
\end{theorem}

An important observation is that if we consider a point set with no two points being coaxial, then there is a unique line passing through any two points , and therefore the point-line incidence structure in this case forms a \emph{linear space}. Hence, we can invoke the classical de-Bruijn-Erd\H os theorem to conclude that such a set of $n$ points determines at least $n$ lines.

With the existence of a Tropical Sylvester-Gallai theorem, it is quite natural to explore the possibility of a tropical version of the de-Bruijn-Erd\H os theorem, i.e., a lower bound on the number of tropical lines determined by a $n$ point set in the tropical plane. However, the number of lines determined by coaxial points are infinite in this setting. For the question of counting lines to be well posed, we would like to be in a scenario where a finite set of points determines a finite set of lines. Hence, rather than counting the number of lines as shown in the figure above, we count a special class of lines, namely \emph{stable tropical} lines.

\begin{definition}\label{thm: tropical-stable}
Consider $(L, p_{1}, \hdots, p_{n}), (n \geq 2)$ where $L$ is a tropical line with the points $(p_{1}, \hdots, p_{n})$ on the line $L$, then $(L, p_{1}, \hdots, p_{n})$, is called \emph{stable} if
\begin{enumerate}
\item either $L$ is the unique line passing through the $p_{i}$'s,  or
\item one of the points  $p_{1}, \hdots, p_{n}$ is the vertex of $L$.
\end{enumerate}
\end{definition}

\begin{figure}[H]
    \centering
    \includegraphics[scale=0.5]{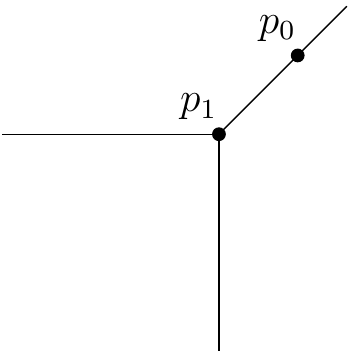}
    \caption{A stable tropical line $(L,p_{0},p_{1})$}
    \label{fig:stable_line}
\end{figure}

Now we show that this restriction on the counting of lines, turns out to be quite general as these stable lines turn out to be the tropicalization of the line passing through generic lifts of the points.

\begin{proposition}
Given two coaxial points $p_{1} = (-u,-v)$, $p_{2}=(-u',-v') \in {\mathbb{K}}^{2}$, pick lifts $P_{1} = (a_{1}t^{u}+ \hdots$ , $b_{1}t^{v} + \hdots )$ and $P_{2} = (a_{2}t^{u'} + \hdots$ , $b_{2}t^{v} + \hdots )$ over $\mathbb{K}\{\{t\}\}$. If $b_{1} \neq b_{2}$, then $trop({P_{1}P_{2}})$ is the stable tropical line through $p_{1}$ and $p_{2}$.
\end{proposition}
\begin{proof}

Since we assume that the two points, $p_{1}$ and $p_{2}$ are coaxial, we take $v = v'$ which would imply that the two points are coaxial in the $(-1,0)$ primitive direction.

An equation of a line in the plane is $ax + by = c$. So if the lifts $P_{1}$ and $P_{2}$ lie on this line, then they satisfy this equation

\begin{equation}
a(a_{1}t^{u}+ \hdots ) + b(b_{1}t^{v} + \hdots ) = c
\end{equation}
\begin{equation}
a(a_{2}t^{u'}+ \hdots ) + b(b_{2}t^{v} + \hdots ) = c
\end{equation}
Without loss of generality we assume $u >u'$ and $a = 1$. So subtracting the two equations gives us

\
\begin{center}

$-a_{2}t^{u'} + O(t^{u'})$ + $b((b_{1} - b_{2})t^{v} + O(t^{v})) = 0$

\

$\implies b = \frac{a_{2}t^{u'} + O(t^{u'})}{(b_{1} - b_{2})t^{v} + O(t^{v})}  $

\

$\implies c = a_{2}t^{u'} + \hdots + \frac{(a_{2}t^{u} + O(t^{u'})) \cdot (b_{2}t^{v} + \hdots )}{(b_{1} - b_{2})(t^{v}) + \hdots )} $

\
\end{center}
Therefore $val(c) = -u'$ and $val(b) = -u'-v$, and we get the following Newton polytope and the tropicalization, 

\begin{figure}[H]
    \centering
    \begin{subfigure}{.5\textwidth}
     \centering
    \includegraphics[scale=0.45]{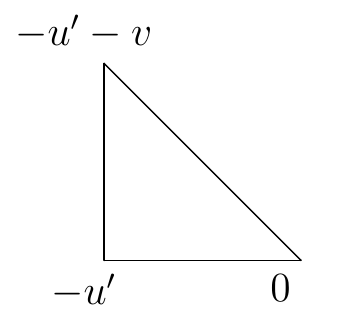}
    \label{fig:trianglestable}
    \end{subfigure}%
     \begin{subfigure}{.5\textwidth}
      \centering
     \includegraphics[scale=0.45]{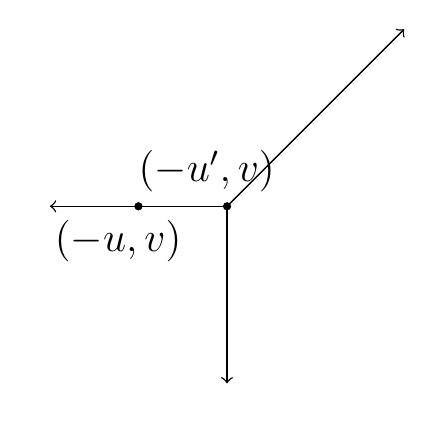}
    \label{fig:stableline}
    \end{subfigure}
    \caption{Newton polytope and tropicalization for $trop(P_{1}P_{2})$}
\end{figure}


which is a stable tropical line passing through p$_{1}=(-u',-v)$ and p$_{2}=(-u,-v)$.

The result for two points being coaxial in the other two primitive directions also follows with a similar computation.
\end{proof}

Alternatively, in \cite{LT11} in Section 2.2 a notion of a stable curve though a set of $n$ points is introduced. The definition of a stable curve in \cite{LT11} is as follows

\begin{definition}
The stable curve of support $I$ passing through $\{q_{1}, \hdots ,q_{\delta-1}\}$
is the curve defined by the polynomial $f$ = “$\sum_{i \in I}$ $a_{i}x^{i_{1}} y^{i_{2}}$”, where the coordinates $a_{i}$ of $f$ are the stable solutions to the linear system imposed by passing through the points $q_{j}$ .
\end{definition}
where for a curve $H$ given by a polynomial $f$, the support is the set of tuples of $ i \in \mathbb{Z}^{n} $ such that $a_{i}$ appears in $f$, $\delta(I)$ denotes the number of elements in $I$ and the stable solution for a set of tropical linear forms is the common solution for all the linear forms, which is also stable under small perturbations of the coefficients of the linear forms \cite{LT05}\cite{BS05}.

So let us consider the above case for tropical lines and try to see the equivalent definitions of stable lines through two points according to \cite{LT11}.

The linear form that represents a tropical line in the tropical plane is given by 
\begin{equation}\label{thm: tropical-form}
a\odot x \oplus b \odot y \oplus c
\end{equation}
So the support in this case is a set of 3-tuples of $\mathbb{Z}^{3}$ and $\delta(I) = 3$. We take two arbitrary points in the (-1,0) direction of a tropical line $P_{1} = (-u,v)$ and $P_{2} = (-u',v)$, where $u$ and $u'$ both are positive and $u'$ $\le u$. Now let us compute the stable line passing through $P_{1}$ and $P_{2}$ in the setup of \cite{LT11}.

The tropical linear system obtained by plugging in the points in ~\ref{thm: tropical-form} is as follows

\[
\left\{ 
\begin{array}{c}
a\odot (-u) \oplus b \odot v \oplus c  = 0\\ 
a\odot (-u') \oplus b \odot v \oplus c  = 0\\ 
\end{array}
\right. 
\]

Now the stable solution of the above tropical linear system provides the coefficients for the linear form which defines the stable line passing through the two given points. The corresponding coefficient matrix is given as 
\begin{center}
$C = \begin{bmatrix}
  -u & v & 0\\
  -u'& v & 0 
\end{bmatrix}$.
\end{center}

With the help of explicit computations for calculating stable solutions of tropical linear systems elaborated in \cite{LT05} and \cite{BS05}, in the case above, we find that the stable solution is given by ($|O^{1}|_{t} :  |O^{2}|_{t} : |O^{3}|_{t}$) = $(v : -u' : -u' + v)$ and hence the linear form representing the stable line through P$_{1}$ and P$_{2}$ is given as 
\begin{equation}
v\odot x \oplus -u' \odot y \oplus -u'+ v    
\end{equation}

This is a tropical line with vertex $(\alpha,\beta)$ satisfying 

$\alpha + v = -u' + v \implies \alpha = -u'$ and $\beta - u' = -u' + v \implies \beta = v$.
 
\begin{figure}[H]
    \centering
    \includegraphics[scale=0.6]{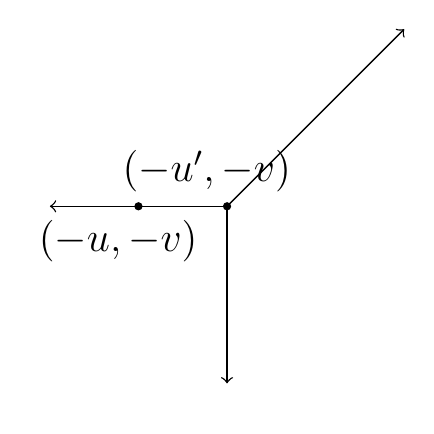}
    \caption{Stable line passing through two given points}
    \label{fig:stableline2}
\end{figure}

The computation for a two point configuration in the other two primitive directions also follows in the same manner.

So as is evident from the above discussion, taking two points on any one of the rays of a tropical line, we see that the definition of a stable line in \cite{LT11} coincides with ~\ref{thm: tropical-stable}.

An important observation here is that the Sylvester-Gallai Theorem fails if we restrict ourselves to stable tropical lines. The following figure shows explicit examples of sets of points in the tropical plane  with $n = 4$ and $5$ points such that these point sets do not determine an ordinary stable tropical line.

\begin{figure}[H]
    \centering
    \includegraphics[scale=0.7]{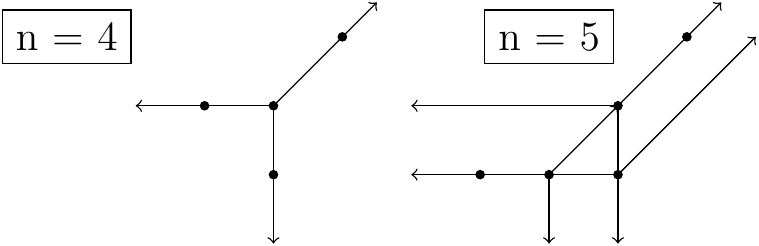}
    \caption{Point sets which do not determine an ordinary stable tropical line}
    \label{fig:my_label}
\end{figure}

\begin{proposition}\label{thm: duality}
Given a $n$-point set $P$ in the tropical plane, the number of stable lines determined by $P$ is equal to the number of stable intersections obtained in the corresponding dual line arrangement.
\end{proposition}

\begin{figure}[H]
\begin{tikzpicture}[scale=0.5]
   \tkzInit[xmax=4,ymax=4,xmin=-4,ymin=-4]
   \tkzAxeXY
   \draw[thick,latex-latex][<-] (-4,2) -- (2,2) node[anchor=south west] {}; 
   \draw[thick,latex-latex][->] (2,2) -- (2,-4) node[anchor=south west] {};
   \draw[thick,latex-latex][->] (2,2) -- (4,4) node[anchor=south west] {}; 
   \draw[dashed,latex-latex][<-] (-3,-4) -- (-3,-3) node[anchor=south west] {}; 
   \draw[dashed,latex-latex][->] (-3,-3) -- (-4,-3) node[anchor=south west] {}; 
   \draw[dashed,latex-latex][->] (-3,-3) -- (2,2) node[anchor=south west] {}; 
   \draw[dashed,latex-latex][->] (-1,-2) -- (-4,-2) node[anchor=south west] {};
   \draw[dashed,latex-latex][->] (-1,-2) -- (-1,-4) node[anchor=south west] {};
   \draw[dashed,latex-latex][->] (-1,-2) -- (4,3) node[anchor=south west] {};
   \draw[red,thick,dashed] (-2,-2) circle (0.2cm);
   \fill[black]  (1,2) circle (.1cm)  node[anchor=south west] {};
   \fill[black]  (3,3) circle (.1cm)  node[anchor=south west] {};
   \fill[black]  (-1,-2) circle (.07cm)  node[anchor=south west] {};
   \fill[black]  (-3,-3) circle (.07cm)  node[anchor=south west] {};
  \end{tikzpicture}
  \hspace{2cm}
  \begin{tikzpicture}[scale=0.5]
   \tkzInit[xmax=4,ymax=4,xmin=-4,ymin=-4]
   \tkzAxeXY
   \draw[thick][<-] (1.5,2) -- (4,2);
   \draw[thick,latex-latex][->] (3,2) -- (3,0) node[anchor=south west] {}; 
   \draw[thick,latex-latex][->] (3,2) -- (4,3) node[anchor=south west] {}; 
   \draw[thick,latex-latex][->] (4,2) -- (4,0) node[anchor=south west] {}; 
   \draw[thick,latex-latex][->] (4,2) -- (5,3) node[anchor=south west] {};
   \draw[dashed,latex-latex][->] (-3,-2) -- (-3,-4) node[anchor=south west] {}; 
   \draw[dashed,latex-latex][->] (-3,-2) -- (-2,-1) node[anchor=south west] {}; 
   \draw[dashed,latex-latex][->] (-4,-2) -- (-4,-4) node[anchor=south west] {}; 
   \draw[dashed,latex-latex][->] (-4,-2) -- (-3,-1) node[anchor=south west] {}; 
   \draw[dashed,latex-latex][<-] (-4.5,-2) -- (-2,-2) node[anchor=south west] {};
   \draw[dashed,latex-latex][->] (-2,-2) -- (-1,-1) node[anchor=south west] {};
   \draw[dashed,latex-latex][->] (-2,-2) -- (-2,-4) node[anchor=south west] {}; 
   \draw[red,thick,dashed] (-3,-2) circle (0.2cm);
   \draw[red,thick,dashed] (-4,-2) circle (0.2cm);
   \fill[black]  (2,2) circle (.1cm)  node[anchor=south west] {};
   \fill[black]  (3,2) circle (.1cm)  node[anchor=south west] {};
   \fill[black]  (4,2) circle (.1cm)  node[anchor=south west] {};
  \end{tikzpicture}
  \caption{Duality between stable lines and stable intersections}
  \label{fig:duality_of_lines}
\end{figure}
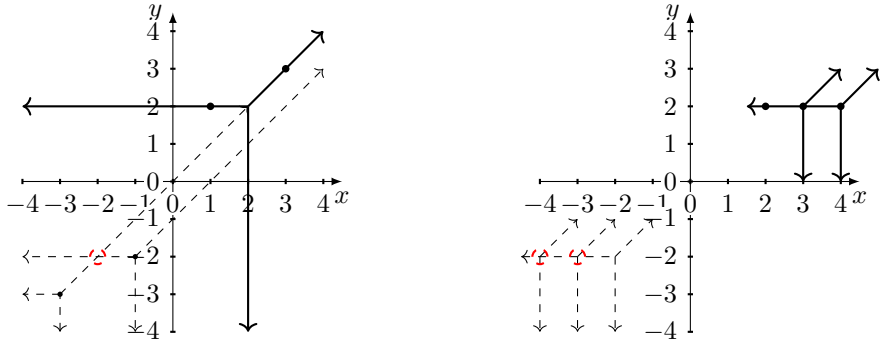

\begin{proof}
Consider an arbitrary stable tropical line $(L,p_{1}, p_{2} \hdots p_{n})$, then by definition either the points $p_{1}, p_{2} \hdots p_{n}$ uniquely determine $L$ or one of the points amongst the $p_{i}'s$ is the vertex of the line $L$. We first consider the case when the points $p_{1}, p_{2} \hdots p_{n}$ determine the line uniquely, and in this case there must be at least two non coaxial points present on the line $L$, and we realize  that under duality, the reflection of the vertex of the line $L$ with respect to the origin corresponds to a unique stable intersection obtained in the dual line arrangement, illustrated in the Figure \ref{fig:duality_of_lines}. This implies a one to one correspondence between stable lines determined by such points and the stable intersections obtained in the dual line arrangement.

Also, if one of the points amongst the $p_{i}'s$ is the vertex of the stable tropical line $L$, then we again oberve that the reflection of the the vertex of the line $L$ with respect to the origin, corresponds to a unique stable intersection in the dual line arrangement, illustrated in the Figure \ref{fig:duality_of_lines}. Hence, we see a one to one correspondence between stable tropical lines and the number of stable intersections in the dual line arrangement.


We realize that this duality between stable intersections and stable lines is a bit stronger; if the stable line is the unique line passing through the points on it, then the vertex of the line corresponds to a stable intersection of first kind and if the stable line has one of the points as a vertex, then the vertex corresponds to a stable intersection of second kind.

\end{proof}

This result illustrates the fact that stable tropical lines are in fact dual to stable intersections of tropical lines.

The above result leads on to the following corollary.

\begin{corollary}
For a given tropical line arrangement $\mathcal{L}$ in the tropical plane, the number of stable intersections equals the number of non-triangular faces in the dual Newton subdivision corresponding to the tropical line arrangement. 
\end{corollary}

\begin{proof}
Since all stable intersections are obtained as intersections of two or more rays, each point of intersection has at least four or more rays emanating from it in the primitive directions. This corresponds through duality to faces with at least four edges or more and the only other faces which contribute in the dual Newton subdivision are triangular faces which are not stable intersections. Hence, the number of stable intersections in the line arrangement is equal to the number of non-triangular faces in the dual Newton subdivision.
\end{proof}

\begin{figure}[H]
    \centering
    \includegraphics[scale=0.7]{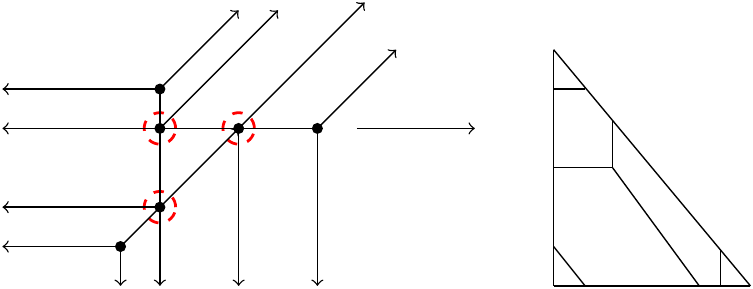}
    \caption{An example demonstrating duality between stable intersections and cells in the Newton subdivision}
    \label{fig:dualityofcells}
\end{figure}

With this duality established, let us look at the total number of faces, which we denote as $t$, present in a dual Newton subdivision of a tropical line arrangement of $n$ tropical lines, where $n$ remains fixed for our discussion. Firstly, there is a trivial lower bound of $n$ on $t$, since the $n$ vertices of the tropical lines contribute at least $n$ faces in the corresponding Newton subdivision. Also $t$ is bounded above by the number $\binom{n}{2} + n$, which is the number of faces when any two lines in the line arrangement intersect transversally at a unique point \cite{DR17}. Therefore, $t$ satisfies the following inequality 

\[  n \leq t \leq \binom{n}{2} + n \]

We recall that stable intersection of first kind correspond to parallelograms and hexagons in the dual Newton subdivision and stable intersections of second kind correspond to irregular cells with four, five or six edges in the dual Newton subdivision. A common description of the all the faces appearing in a dual Newton subdivision is described in the Figure \ref{fig:cell_shape} also present in \cite{DR17},

\begin{figure}[H]
\begin{center}
\begin{tikzpicture}[scale=0.5]
\draw[] (0,0) -- (2,0);
\draw[] (0,0) -- (0,-1.5);
\draw[] (0,-1.5) -- (1.5,-3);
\draw[] (2,0) -- (4,-2);
\draw[] (4,-2) -- (4,-3);
\draw[] (1.5,-3) -- (4,-3);
\fill[black] (1,0) circle (.0000001cm) node[align=left,   above]{$w_{3}\quad$};
\fill[black] (-0.8,-1.1) circle (.0000001cm) node[align=left,   above]{$w_{2} + c\quad$};
\fill[black] (0.55,-2.75) circle (.0000001cm) node[align=left,   above]{$w_{1}\quad$};
\fill[black] (2.9,-3.85) circle (.0000001cm) node[align=left,   above]{$w_{3} + c\quad$};
\fill[black] (4.85,-2.8) circle (.0000001cm) node[align=left,   above]{$w_{2}\quad$};
\fill[black] (4.25,-1.2) circle (.0000001cm) node[align=left,   above]{$w_{1} + c\quad$};
\end{tikzpicture}
\caption{ A cell in the Newton Subdivision, which is dual to a tropical line arrangement \cite{DR17}}
\label{fig:cell_shape}
\end{center}
\end{figure}
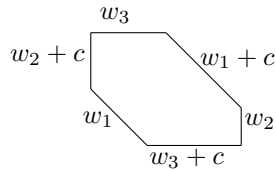

where $w_{1},w_{2}$ and $w_{3}$ are the number of lines, which are coaxial in the three primitive directions, and $c$ represents the number of lines centered at the point dual to the face in the tropical line arrangement. A Newton subdivision with faces of the shape described in Figure \ref{fig:cell_shape}, is called a \emph{linear Newton subdivision} and if the only faces occurring in a linear Newton subdivision are triangles, parallelograms and hexagons, then such a subdivision is called a \emph{semiuniform subdivision} \cite{DR17}.

We refer to faces in the shapes of parallelograms and hexagons as \textbf{semiuniform} faces and faces dual to stable intersections of second kind as \textbf{non-uniform} faces.

Figure~\ref{fig:typesoffaces} shows all the possible shapes of cells present in the dual Newton subdivision of a tropical line arrangement; in the figure for all semiuniform faces, for each edge length parameter we consider $w_{i} = 1$ and for all non-uniform faces we take $w_{i} = c = 1$. For higher values of $w_{i}'s$ and $c$ the shapes remain the same however the edge lengths corresponding to each parameter get elongated according to the values described in Figure \ref{fig:cell_shape}.

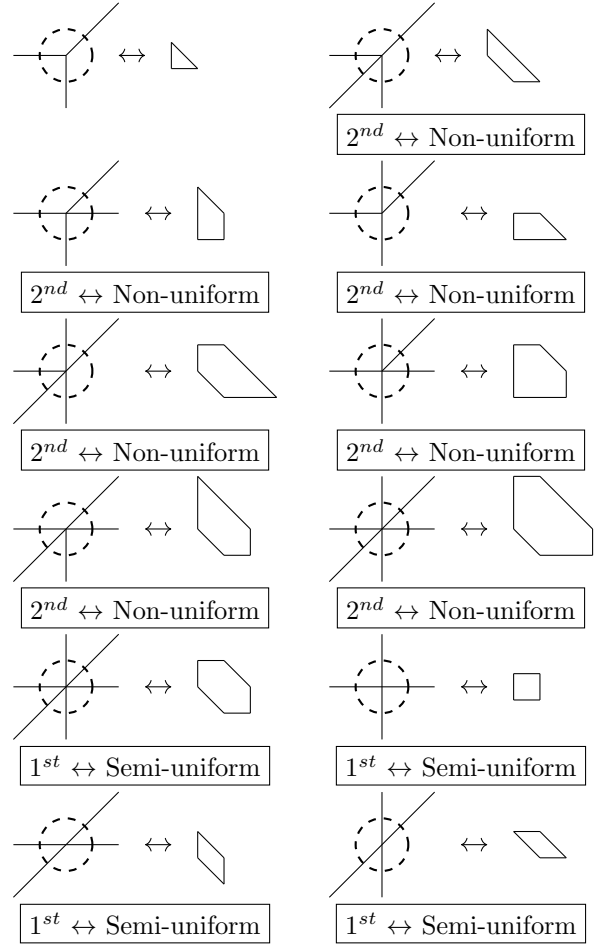
\begin{figure}
\begin{center}
\begin{tikzpicture}[scale=0.35]
\draw[] (0,0) -- (-2,0);
\draw[] (0,0) -- (0,-2);
\draw[] (0,0) -- (2,2);
\draw[thick,dashed] (0,0) circle (1cm);
\draw[<->] (2,0) -- (3,0);
\draw[] (4,0.5) -- (4,-0.5);
\draw[] (4,-0.5) -- (5,-0.5);
\draw[] (5,-0.5) -- (4,0.5);
\draw[] (12,0) -- (10,0);
\draw[] (12,0) -- (12,-2);
\draw[] (12,0) -- (14,2);
\draw[] (12,0) -- (10,-2);
\draw[<->] (14,0) -- (15,0);
\draw[] (16,0) -- (16,1);
\draw[] (16,0) -- (17,-1);
\draw[] (17,-1) -- (18,-1);
\draw[] (16,1) -- (18,-1);
\draw[thick,dashed] (12,0) circle (1cm);
\node[draw] at (15,-3) {$2^{nd}$ $\leftrightarrow$ Non-uniform};
\draw[] (0,-6) -- (-2,-6);
\draw[] (0,-6) -- (0,-8);
\draw[] (0,-6) -- (2,-4);
\draw[] (0,-6) -- (2,-6);
\draw[<->] (3,-6) -- (4,-6);
\draw[] (5,-7) -- (5,-5);
\draw[] (6,-6) -- (5,-5);
\draw[] (6,-6) -- (6,-7);
\draw[] (5,-7) -- (6,-7);
\draw[thick,dashed] (0,-6) circle (1cm);
\node[draw] at (3,-9) {$2^{nd}$ $\leftrightarrow$ Non-uniform};
\draw[] (12,-6) -- (10,-6);
\draw[] (12,-6) -- (12,-8);
\draw[] (12,-6) -- (14,-4);
\draw[] (12,-6) -- (12,-4);
\draw[<->] (15,-6) -- (16,-6);
\draw[] (17,-6) -- (17,-7);
\draw[] (17,-7) -- (19,-7);
\draw[] (17,-6) -- (18,-6);
\draw[] (18,-6) -- (19,-7);
\draw[thick,dashed] (12,-6) circle (1cm);
\node[draw] at (15,-9) {$2^{nd}$ $\leftrightarrow$ Non-uniform};
\draw[] (0,-12) -- (-2,-12);
\draw[] (0,-12) -- (0,-14);
\draw[] (0,-12) -- (2,-10);
\draw[] (0,-12) -- (0,-10);
\draw[] (0,-12) -- (-2,-14);
\draw[<->] (3,-12) -- (4,-12);
\draw[thick,dashed] (0,-12) circle (1cm);
\draw[] (5,-12) -- (5,-11);
\draw[] (5,-11) -- (6,-11);
\draw[] (6,-11) -- (8,-13);
\draw[] (8,-13) -- (6,-13);
\draw[] (6,-13) -- (5,-12);
\node[draw] at (3,-15) {$2^{nd}$ $\leftrightarrow$ Non-uniform};
\draw[] (12,-12) -- (10,-12);
\draw[] (12,-12) -- (12,-14);
\draw[] (12,-12) -- (14,-10);
\draw[] (12,-12) -- (12,-10);
\draw[] (12,-12) -- (14,-12);
\draw[<->] (15,-12) -- (16,-12);
\draw[thick,dashed] (12,-12) circle (1cm);
\draw[] (17,-11) -- (17,-13);
\draw[] (17,-13) -- (19,-13);
\draw[] (19,-13) -- (19,-12);
\draw[] (19,-12) -- (18,-11);
\draw[] (18,-11) -- (17,-11);
\node[draw] at (15,-15) {$2^{nd}$ $\leftrightarrow$ Non-uniform};
\draw[] (0,-18) -- (-2,-18);
\draw[] (0,-18) -- (0,-20);
\draw[] (0,-18) -- (2,-16);
\draw[] (0,-18) -- (2,-18);
\draw[] (0,-18) -- (-2,-20);
\draw[<->] (3,-18) -- (4,-18);
\draw[thick,dashed] (0,-18) circle (1cm);
\draw[] (5,-18) -- (5,-16);
\draw[] (5,-18) -- (6,-19);
\draw[] (6,-19) -- (7,-19);
\draw[] (7,-18) -- (7,-19);
\draw[] (7,-18) -- (5,-16);
\node[draw] at (3,-21) {$2^{nd}$ $\leftrightarrow$ Non-uniform};
\draw[] (12,-18) -- (10,-18);
\draw[] (12,-18) -- (12,-20);
\draw[] (12,-18) -- (14,-16);
\draw[] (12,-18) -- (12,-16);
\draw[] (12,-18) -- (14,-18);
\draw[] (12,-18) -- (10,-20);
\draw[<->] (15,-18) -- (16,-18);
\draw[thick,dashed] (12,-18) circle (1cm);
\draw[] (17,-18) -- (17,-16);
\draw[] (17,-16) -- (18,-16);
\draw[] (18,-16) -- (20,-18);
\draw[] (20,-18) -- (20,-19);
\draw[] (20,-19) -- (18,-19);
\draw[] (18,-19) -- (17,-18);
\node[draw] at (15,-21) {$2^{nd}$ $\leftrightarrow$ Non-uniform};
\draw[] (-2,-24) -- (2,-24);
\draw[] (2,-22) -- (-2,-26);
\draw[] (0,-22) -- (0,-26);
\draw[<->] (3,-24) -- (4,-24);
\draw[thick,dashed] (0,-24) circle (1cm);
\draw[] (5,-24) -- (6,-25);
\draw[] (6,-25) -- (7,-25);
\draw[] (7,-25) -- (7,-24);
\draw[] (7,-24) -- (6,-23);
\draw[] (5,-23) -- (6,-23);
\draw[] (5,-24) -- (5,-23);
\node[draw] at (3,-27) {$1^{st}$ $\leftrightarrow$ Semi-uniform};
\draw[] (10,-24) -- (14,-24);
\draw[] (12,-22) -- (12,-26);
\draw[<->] (15,-24) -- (16,-24);
\draw[thick,dashed] (12,-24) circle (1cm);
\draw[] (17,-24.5) -- (17,-23.5);
\draw[] (17,-23.5) -- (18,-23.5);
\draw[] (18,-23.5) -- (18,-24.5);
\draw[] (17,-24.5) -- (18,-24.5);
\node[draw] at (15,-27) {$1^{st}$ $\leftrightarrow$ Semi-uniform};
\draw[] (-2,-30) -- (2,-30);
\draw[] (-2,-32) -- (2,-28);
\draw[<->] (3,-30) -- (4,-30);
\draw[thick,dashed] (0,-30) circle (1cm);
\draw[] (5,-29.5) -- (5,-30.5);
\draw[] (5,-30.5) -- (6,-31.5);
\draw[] (6,-31.5) -- (6,-30.5);
\draw[] (6,-30.5) -- (5,-29.5);
\node[draw] at (3,-33) {$1^{st}$ $\leftrightarrow$ Semi-uniform};
\draw[] (12,-28) -- (12,-32);
\draw[] (10,-32) -- (14,-28);
\draw[<->] (15,-30) -- (16,-30);
\draw[thick,dashed] (12,-30) circle (1cm);
\draw[] (17,-29.5) -- (18,-29.5);
\draw[] (17,-29.5) -- (18,-30.5);
\draw[] (18,-29.5) -- (19,-30.5);
\draw[] (18,-30.5) -- (19,-30.5);
\node[draw] at (15,-33) {$1^{st}$ $\leftrightarrow$ Semi-uniform};
\end{tikzpicture}
\caption{All possible shapes of faces present in the Newton subdivision of a tropical line arrangement; with the labelings having type of stable intersections on the left along with the type of face that corresponds to it on the right}
\label{fig:typesoffaces}
\end{center}
\end{figure}


We move on to discuss one of the extremal cases for the values of $t$, which is the case when $t = n$.

\begin{lemma}\label{lem:lemma1}
Let $\mathcal{L}$ be a tropical line arrangement of n lines, having exactly n faces in the corresponding dual Newton subdivision, then $\mathcal{L}$ has no stable intersections of first kind in the tropical line arrangement.
\end{lemma}
\begin{proof}
We start with a tropical line arrangement $\mathcal{L}$ of $n$ tropical lines, such that it has exactly $n$ faces. We continue by contradiction, assuming that there does exist a stable intersection of first kind in the line arrangement $\mathcal{L}$. However, since there are at least $n$ faces contributed by the $n$ vertices of the $n$ tropical lines, and the face corresponding to a stable intersection of first kind is not one of them, therefore this would imply that the total number of faces in the dual Newton subdivision corresponding to $\mathcal{L}$ has at least $n + 1$  faces, which is a contradiction to the fact that $\mathcal{L}$ has $n$ faces in the dual Newton subdivision. Hence, the proof.
\end{proof}

We look at an example of a $n$ line arrangement with exactly $n$ faces. 

\begin{center}
\begin{figure}[H]
    \centering
    \includegraphics[scale=0.6]{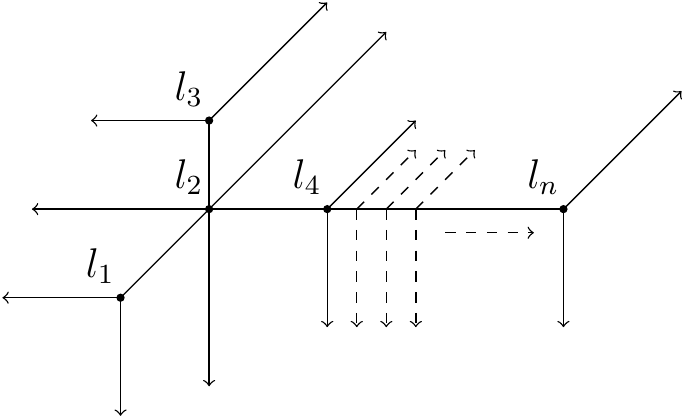}
    \caption{An example of a line arrangement with exactly $n$ faces and three triangular faces}
    \label{fig:extremal}
\end{figure}
\end{center}

The example depicted in Figure \ref{fig:extremal} shows a tropical line arrangement of $n$ tropical lines $\{l_{1},l_{2},l_{3},l_{4}, \hdots,l_{n}\}$, such that the total number of faces in the corresponding Newton subdivision is $n$, and it has exactly three triangular faces located at the corners of the Newton polygon.

We use $({v}{^{max}})_{t}$ to represent the maximum number of triangular faces present in a Newton subdivision corresponding to a tropical line arrangement with total number of faces in the Newton subdivision being equal to $t$.
 
\begin{lemma}\label{lem:lemma2}
Let $\mathcal{L}$ be a tropical line arrangement of $n$ tropical lines such that the dual Newton subdivision $\mathcal{N}$ has exactly $n$ faces, then the maximum number of triangular faces in $\mathcal{N}$ is $3$, i.e., $({v}{^{max}})_{n} = 3$ .
\end{lemma}
\begin{proof}

As can be seen from the example above, there are explicit tropical line arrangements of $n$ tropical lines with $n$ faces in the dual Newton subdivision with exactly 3 triangular faces. 
 
We proceed by contradiction, and assume that $({v}{^{max}})_{n} > 3$. With $({v}{^{max}})_{n} > 3$, we can conclude that there does exist at least one triangular face $T$ in the relative interior of the Newton polygon, i.e., when no edges of $T$ lie on the boundary of the Newton polygon or there is at least one  triangular face $T$ which intersects the boundary of the Newton polygon in exactly one edge. We first consider the case when $T$ is in the relative interior of the Newton polygon. 

Let us consider the three faces $C_{1},C_{2}$ and $C_{3}$ that share an edge with the triangular face $T$ and we consider an example of the local line arrangement around $T$ as depicted in the Figure \ref{fig:case1}.

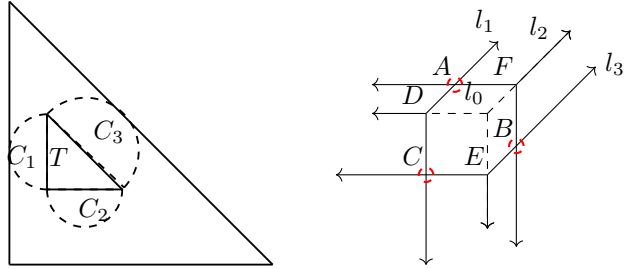
\begin{figure}
\begin{center}
\begin{tikzpicture}[ semic/.style args={#1,#2}{semicircle,minimum width=#1,draw,anchor=arc end,rotate=#2},outer sep=0pt,line width=.7pt]
\draw[] (0,0) -- (3.5,0);
\draw[] (0,0) -- (0,3.5);
\draw[] (3.5,0) -- (0,3.5);
\draw[] (0.5,1) -- (0.5,2);
\draw[] (0.5,1) -- (1.5,1);
\draw[] (0.5,2) -- (1.5,1);
\pgfmathsetmacro{\ATAN}{atan{-.95}} 
\node [dashed][semic={1cm,90}]    at (0.5,1){};
\node [dashed][semic={1cm,180}]   at (1.5,1){};
\node [dashed][semic={1.414cm,\ATAN}] at (0.5,2){};
\node[below left= 0.15mm of {(0.5,1.7)}] {$C_{1}$};
\node[below right= 0.05mm of {(0.8,1)}] {$C_{2}$};
\node[above right= 0.005mm of {(1,1.5)}] {$C_{3}$};
\node[above left= 0.0000000001mm of {(0.9,1.2)}] {$T$};
\end{tikzpicture}
\qquad
\begin{tikzpicture}[scale=0.48]
\draw[][dashed] (1.2,1.2) -- (-0.5,1.2);
\draw[->][dashed] (1.2,1.2) -- (1.2,-2);
\draw[->][dashed] (1.2,1.2) -- (3.5,3.5);
\draw[->][] (-0.5,1.2) -- (-0.5,-3);
\draw[->][] (-0.5,1.2) -- (1.5,3.2);
\draw[->][] (-0.5,1.2) -- (-2,1.2);
\draw[->][] (2,2) -- (3.5,3.5);
\draw[->][] (2,2) -- (2,-2.5);
\draw[->][] (2,2) -- (-2,2);
\draw[->][] (1.2,-0.5) -- (-3,-0.5);
\draw[->][] (1.2,-0.5) -- (1.2,-2);
\draw[->][] (1.2,-0.5) -- (4.2,2.5);
\draw[red,thick,dashed] (0.3,2) circle (0.2cm);
\draw[red,thick,dashed] (2,0.3) circle (0.2cm);
\draw[red,thick,dashed] (-0.5,-0.5) circle (0.2cm);
\fill[black] (0.3,2) circle (.0001cm) node[align=left,   above]{$A\quad$};
\fill[black] (2,0.3) circle (.0001cm) node[align=left,   above]{$B\quad$};
\fill[black] (-0.5,-0.5) circle (.0001cm) node[align=left,   above]{$C\quad$};
\fill[black] (-0.5,1.2) circle (.0001cm) node[align=left,   above]{$D\quad$};
\fill[black] (1.2,-0.5) circle (.0001cm) node[align=left,   above]{$E\quad$};
\fill[black] (2,2) circle (.0001cm) node[align=left,   above]{$F\quad$};
\fill[black] (1.2,1.2) circle (.0001cm) node[align=left,   above]{$l_{0}\quad$};
\fill[black] (1.5,3.2) circle (.0001cm) node[align=left,   above]{$l_{1}\quad$};
\fill[black] (3,3) circle (.0001cm) node[align=left,   above]{$l_{2}\quad$};
\fill[black] (5.1,2) circle (.0001cm) node[align=left,   above]{$l_{3}\quad$};
\end{tikzpicture}
\end{center}
\caption{Positions of cells in the Newton subdivision and the local line arrangement dual to it}
\label{fig:case1}
\end{figure}

In the figure we see that the points $D$, $E$ and $F$ represent the vertices of the tropical lines $l_{1},l_{2}$ and $l_{3}$ which are present at the stable intersections of second kind at these points, dual to the cells $C_{1},C_{2}$ and $C_{3}$ in $\mathcal{N}$. Also $l_{0}$ represents the line dual to the triangular face $T$. Also, by Lemma \ref{lem:lemma1} we know that no stable intersections of first kind are present in the line arrangement.

In the local picture, we obtain three stable intersections of first kind at the points $A, B$ and $C$. Since these points are stable intersections and by Lemma \ref{lem:lemma1} we know we can not have any stable intersections of first kind therefore there must exist a line with its vertex at these points. Let us consider one of these intersections, $A$. The points $D$, $A$ and $F$ are represented as $(x_{1},y_{1}),(x_{2},y_{2})$ and $(x_{3},y_{3})$, then it is easy to see that 

\[ x_{1} < x_{2} < x_{3} \]

This helps to conclude that if there is a tropical line present with vertex at A, then it would either intersect the lines $l_{0}$ and $l_{3}$ at two points, or meet the vertex of the line $l_{0}$. There cannot be a line with vertex at A meeting the vertex of the line $l_{0}$ as that would contradict the fact that the face corresponding to $l_{0}$ is a triangular face $T$ in $\mathcal{N}$. So we continue with the other case when the line has the vertex at A and intersects the lines $l_{0}$ and $l_{3}$ at two points. But there cannot be a tropical line present at A with two points of intersection with the lines $l_{0}$ and $l_{3}$, as that would contradict the fact that the cells $C_{1},C_{2}$ and $C_{3}$ corresponding to the stable intersections at $D$, $E$ and $F$, share an edge with the triangular face $T$. Hence, there cannot be a tropical line with a vertex at $A$, and therefore $A$ has to be a stable intersection of first kind, which contradicts Lemma \ref{lem:lemma1}. The same argument follows for the other two points of intersections, $B$ and $C$. However, this is a contradiction to the Lemma \ref{lem:lemma1}. Another  observation is that for all possibilities of non-uniform  faces (arising from stable intersections of second kind) surrounding $T$, we obtain points of intersections in similar positions as $A$, $B$ and $C$ which establishes the existence of at least three stable intersections of first kind, and hence gives a contradiction. Therefore, this shows that it is not possible to place a triangular face in the relative interior of the Newton polygon.


The other possible case is when the triangular face intersects the boundary of the Newton polygon in exactly one edge. This means that the triangular face is surrounded by other faces of the subdivision from two sides.

Without loss of generality, we take the triangular face to be intersecting with one of the edges of the Newton polygon as depicted in the Figure~\ref{fig:case2} and we look at the local line arrangement around the triangular face $T$. 

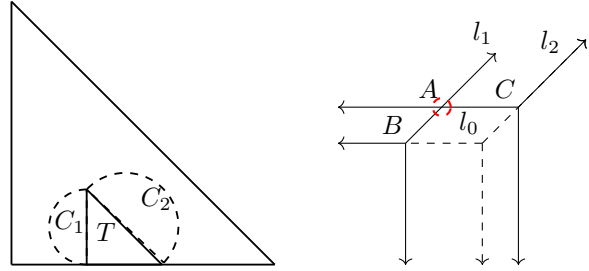
\begin{figure}
\begin{center}
\begin{tikzpicture}[semic/.style args={#1,#2}{semicircle,minimum width=#1,draw,anchor=arc end,rotate=#2},outer sep=0pt,line width=.7pt]
\draw[] (0,0) -- (3.5,0);
\draw[] (0,0) -- (0,3.5);
\draw[] (3.5,0) -- (0,3.5);
\draw[] (1,0) -- (2,0);
\draw[] (1,0) -- (1,1);
\draw[] (1,1) -- (2,0);
\pgfmathsetmacro{\ATAN}{atan{-.95}} 
\node [dashed][semic={1cm,90}]    at (1,0){};
\node [dashed][semic={1.414cm,\ATAN}] at (1,1){};
\node[above right= 0.05mm of {(1.6,0.6)}] {$C_{2}$};
\node[above left= 0.005mm of {(1.1,0.3)}] {$C_{1}$};
\node[above left= 0.0000000001mm of {(1.5,0.2)}] {$T$};
\end{tikzpicture}
\qquad
\begin{tikzpicture}[scale=0.6]
\draw[][dashed] (1.2,1.2) -- (-0.5,1.2);
\draw[->][dashed] (1.2,1.2) -- (1.2,-1.5);
\draw[->][dashed] (1.2,1.2) -- (3.5,3.5);
\draw[->][] (-0.5,1.2) -- (-0.5,-1.5);
\draw[->][] (-0.5,1.2) -- (1.5,3.2);
\draw[->][] (-0.5,1.2) -- (-2,1.2);
\draw[->][] (2,2) -- (3.5,3.5);
\draw[->][] (2,2) -- (2,-1.5);
\draw[->][] (2,2) -- (-2,2);
\draw[red,thick,dashed] (0.3,2) circle (0.2cm);
\fill[black] (0.3,2) circle (.0001cm) node[align=left,   above]{$A\quad$};
\fill[black] (-0.5,1.2) circle (.0001cm) node[align=left,   above]{$B\quad$};
\fill[black] (2,2) circle (.0001cm) node[align=left,   above]{$C\quad$};
\fill[black] (1.2,1.2) circle (.0001cm) node[align=left,   above]{$l_{0}\quad$};
\fill[black] (1.5,3.2) circle (.0001cm) node[align=left,   above]{$l_{1}\quad$};
\fill[black] (3,3) circle (.0001cm) node[align=left,   above]{$l_{2}\quad$};
\end{tikzpicture}
\caption{Positions of cell in the Newton subdivision and the local line arrangement dual to it}
\label{fig:case2}
\end{center}
\end{figure}

We argue in the same way as we did in the previous case, and realize that by Lemma \ref{lem:lemma1}, $C_{1}$ and $C_{2}$ are non-uniform faces. As we see in the figure the points $B$ and $C$ represent the vertices of the tropical lines $l_{1}$ and $l_{2}$ which are present at the stable intersection of second kind at these points, dual to the cells $C_{1}$ and $C_{2}$ in $\mathcal{N}$. Here $l_{0}$ represents the line dual to the triangular face $T$.

In the local picture, we obtain a stable intersection of first kind at the point $A$. If the points $B$, $A$ and $C$ are represented as $(x_{1},y_{1}),(x_{2},y_{2})$ and $(x_{3},y_{3})$, then it is easy to see that 

\[ x_{1} < x_{2} < x_{3} \]

This helps to conclude that if there is a tropical line present with vertex at $A$, then it would either intersect the line $l_{0}$, or meet the vertex of the line $l_{0}$. There cannot be a line with vertex at $A$ meeting the vertex of the line $l_{0}$ as that would contradict the fact that the face corresponding to $l_{0}$ is a triangular face $T$ in $\mathcal{N}$. So we continue with the other case when the line has the vertex at $A$ and intersects the lines $l_{0}$. But there cannot be a tropical line present at this intersection as that would contradict the fact that the cells $C_{1}$ and $C_{2}$  corresponding to the stable  intersections of second kind at $B$ and $C$ share an edge with the triangular face $T$. Hence, there cannot be a tropical line with a vertex at $A$, and therefore $A$ has to be a stable intersection of first kind, which contradicts Lemma \ref{lem:lemma1}. It is easy to verify that this contradiction occurs for all possibilities of  non-uniform faces (arising from stable intersections of second kind), which can be adjacent to $T$.

Therefore, the only places left to place a triangular face in the Newton polygon, are the three corners, and hence the maximum number of triangular faces that can be obtained is three, i.e., $({v}{^{max}})_{n} = 3$.
\end{proof}

With this result we obtain the following corollary,

\begin{corollary}
Let $\mathcal{L}$ be a tropical line arrangement of $n$ lines, such that $t = n$ and let $v$ denote the number of triangular faces present in the dual Newton subdivision $\mathcal{N}$. Then 
\[ n - v \geq n - 3 \]
\end{corollary}

\begin{remark}
An important inference is that for tropical line arrangements of $n$ lines, with $n$ faces in the dual Newton subdivision, they occur in four distinct classes. Each class is represented by the number of triangular faces at the corners, which varies between $0, 1, 2$ and $3$.  
\end{remark}

With this result, we now know the bound on the number of stable intersections of an $n$ line arrangement with exactly $n$ faces in the corresponding dual Newton subdivision. We now move on to the more general situation.

We now define what it means for a semiuniform face to be \emph{determined} by a triangular face $T$.

\begin{definition}
A semiuniform face $S$ in a dual Newton subdivision is said to be \textbf{determined} by a triangular face $T$ if,

\begin{enumerate}
    \item $S$ is adjacent to $T$, i.e., $T$ and $S$ share an edge, or
    \item $S$ is located as the faces $S_{1},S_{2}$ or $S_{3}$ depicted in the Figure \ref{fig:determined}
\end{enumerate}
\end{definition}

 Here the shapes and location of these three semiuniform faces has to be exactly the same as shown in the figure in order for the faces to be determined by the triangular face $T$. We also note that edge lengths  of these faces need not be unit length, and they could be elongated depending on the lattice length parameters $w_{i}$ and $c$ of the adjacent faces to $T$. We also note that a triangular face determines at most six semiuniform faces; at most three adjacent to it and at most three non adjacent to it. 

\begin{center}
\begin{figure}
\begin{tikzpicture}[scale=0.6]
\draw[][] (0,0) -- (0,-1);
\draw[][] (0,-1) -- (1,-1);
\draw[][] (1,-1) -- (0,0);
\draw[] (0,0) -- (-1,1);
\draw[] (-1,1) -- (-1,2);
\draw[] (-1,2) -- (0,1);
\draw[] (0,1) -- (0,0);
\draw[] (0,-1) -- (0,-2);
\draw[] (0,-1) -- (-1,-1);
\draw[] (-1,-1) -- (-1,-2);
\draw[] (0,-2) -- (-1,-2);
\draw[] (1,-1) -- (2,-1);
\draw[] (1,-1) -- (2,-2);
\draw[] (2,-2) -- (3,-2);
\draw[] (2,-1) -- (3,-2);
\node[below left= 0.15mm of {(0.7,-0.3)}] {$T$};
\node[below left= 0.15mm of {(-0.05,1.65)}] {$S_{1}$};
\node[below left= 0.15mm of {(-0.05,-1.15)}] {$S_{2}$};
\node[below left= 0.15mm of {2.6,-1.15)}] {$S_{3}$};
\end{tikzpicture}
\caption{The non-adjacent semiuniform faces determined by a triangular face $T$}
\label{fig:determined}
\end{figure}
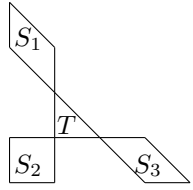
\end{center}

We note that as a consequence of the definition, the determined faces $S_{1},S_{2}$ or $S_{3}$ cannot be hexagonal faces.

With the above definitions, we look at the number of semiuniform faces determined by a triangular face depending on the location of the triangular face in the dual Newton subdivision.

\begin{theorem}\label{thm:triangledetermine}\label{thm:basecase}
Let $\mathcal{L}$ be a tropical line arrangement of $n$ lines and let $\mathcal{N}$ be its dual Newton subdivision. If $T$ is a triangular face in $\mathcal{N}$ (excluding the corners), then \begin{enumerate}
    \item $T$ determines at least three seminuniform faces if $T$ is in the relative interior of the Newton polygon; i.e., when no edges of $T$ lie on the boundary of the Newton polygon.
    \item $T$ determines at least one seminuniform face if $T$ is at the boundary of the Newton polygon; i.e, when one of the edges of $T$ lie on the boundary of the Newton polygon. 
\end{enumerate} 
\end{theorem}

\begin{proof}

We continue with the discussion in the Lemma \ref{lem:lemma2}. As we see in Figure~\ref{fig:case1}, it is shown that a triangular face $T$, which is not adjacent to a semiuniform face, determines at least three semiuniform faces if $T$ is in the interior and at least one semiuniform face if $T$ is located at the boundary. However, semiuniform faces might also occur as faces adjacent to the triangular face. Therefore, when we consider the triangular face $T$ in the interior, then $T$ can be adjacent to either one, two or at most three semiuniform faces. We know that if $T$ is adjacent to semiuniform faces at all edges, then there are at least three semiuniform faces determined by $T$ in the subdivision, trivially. Now we consider the case, when the triangular face is adjacent to two semiuniform faces. In this case, the location of the triangular face, implies existence of at least one non-adjacent semiuniform faces. Similarly, in the case when the triangular face is adjacent to one semiuniform face, at least two non-adjacent semiuniform faces are obtained. Both these cases are illustrated through an example in the Figure~\ref{fig;theorem5}. Hence if a triangular face is in the interior of the Newton polygon, then it implies the existence of at least three semiuniform faces. 

\begin{center}
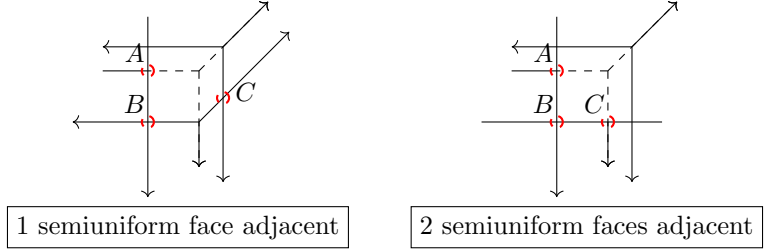
\begin{figure}
\begin{tikzpicture}[scale=0.4]
\draw[][dashed] (1.2,1.2) -- (-0.5,1.2);
\draw[->][dashed] (1.2,1.2) -- (1.2,-2);
\draw[->][dashed] (1.2,1.2) -- (3.5,3.5);
\draw[->][] (-0.5,1.2) -- (-0.5,-3);
\draw[][] (-0.5,1.2) -- (-0.5,3);
\draw[][] (-0.5,1.2) -- (-2,1.2);
\draw[->][] (2,2) -- (3.5,3.5);
\draw[->][] (2,2) -- (2,-2.5);
\draw[->][] (2,2) -- (-2,2);
\draw[->][] (1.2,-0.5) -- (-3,-0.5);
\draw[->][] (1.2,-0.5) -- (1.2,-2);
\draw[->][] (1.2,-0.5) -- (4.2,2.5);
\draw[red,thick,dashed] (-0.5,1.2) circle (0.2cm);
\draw[red,thick,dashed] (2,0.3) circle (0.2cm);
\draw[red,thick,dashed] (-0.5,-0.5) circle (0.2cm);
\fill[black] (3.2,-0.1) circle (.0001cm) node[align=left,   above]{$C\quad$};
\fill[black] (-0.5,-0.5) circle (.0001cm) node[align=left,   above]{$B\quad$};
\fill[black] (-0.5,1.2) circle (.0001cm) node[align=left,   above]{$A\quad$};
\node[draw] at (0.5,-4) {1 semiuniform face adjacent};
\end{tikzpicture}
\qquad
\begin{tikzpicture}[scale=0.4]
\draw[][dashed] (1.2,1.2) -- (-0.5,1.2);
\draw[->][dashed] (1.2,1.2) -- (1.2,-2);
\draw[->][dashed] (1.2,1.2) -- (3.5,3.5);
\draw[->][] (-0.5,1.2) -- (-0.5,-3);
\draw[][] (-0.5,1.2) -- (-0.5,3);
\draw[][] (-0.5,1.2) -- (-2,1.2);
\draw[->][] (2,2) -- (3.5,3.5);
\draw[->][] (2,2) -- (2,-2.5);
\draw[->][] (2,2) -- (-2,2);
\draw[][] (1.2,-0.5) -- (-3,-0.5);
\draw[->][] (1.2,-0.5) -- (1.2,-2);
\draw[][] (1.2,-0.5) -- (3,-0.5);
\draw[red,thick,dashed] (-0.5,1.2) circle (0.2cm);
\draw[red,thick,dashed] (1.2,-0.5) circle (0.2cm);
\draw[red,thick,dashed] (-0.5,-0.5) circle (0.2cm);
\fill[black] (1.2,-0.5) circle (.0001cm) node[align=left,   above]{$C\quad$};
\fill[black] (-0.5,-0.5) circle (.0001cm) node[align=left,   above]{$B\quad$};
\fill[black] (-0.5,1.2) circle (.0001cm) node[align=left,   above]{$A\quad$};
\node[draw] at (0.5,-4) {2 semiuniform faces adjacent};
\end{tikzpicture}
\caption{Examples depicting local line arranagements dual to  a triangular face with 1 or 2 semiuniform faces adjacent to it.}
\label{fig;theorem5}
\end{figure}
\end{center}

Similarly, if we consider the case when the triangular face $T$ is located at the boundary, then if there are semiuniform faces adjacent to $T$ at one or two edges, then there exists at least one semiuniform face in the subdivision, trivially. If $T$ is not adjacent to any semiuniform faces, then we see in Figure~\ref{fig:case2}, that $T$ determines at least one semiuniform face. Hence, we can conclude that if a triangular face is at the boundary then it determines at least one semiuniform face.
\end{proof}

We now move on to count the total number of semiuniform face determined by the triangular faces. Since two or more triangular faces can determine common semiuniform faces, therefore the total count cannot be a direct sum of determined faces of all triangular faces. 

With an abuse of notation we denote $T$ to be a triangular face and  $n(T)$ represent the number of semiuniform faces determined by the triangular face $T$. Hence, $n(T_{1} \cup ... \cup T_{m})$ denotes the total number of semiuniform faces determined by the triangular faces $T_{1},...,T_{m}$.

\begin{theorem}\label{thm:induction}
Let $\mathcal{L}$ be a tropical line arrangement of $n$ lines and $\mathcal{N}$ be its dual Newton subdivision, with $T_{1} \hdots T_{m}$ being the triangular faces in $\mathcal{N}$ (excluding the corners) and $k$ be the number of stable intersections of first kind. Then,

\[ k \geq n(T_{1} \cup T_{2} \hdots \cup T_{m}) \geq m \]
\end{theorem}

\begin{proof}
We proceed with induction on $m$, with the base case being $m=1$. We see that in this case, by Theorem \ref{thm:basecase}, we know that the unique triangle present in the interior of $\mathcal{N}$ determines at least one semiuniform face, therefore 
\[ k \geq n(T) \geq 1 \]

\begin{figure}\centering
\begin{tikzpicture}[scale=0.35]
\draw[] (0,2) -- (2,4);
\draw[<-] (-2,4) -- (2,4);
\draw[] (2,4) -- (4,4);
\draw[->] (0,2) -- (-2,2);
\draw[->] (0,2) -- (0,0);
\draw[->] (2,4) -- (4,6);
\draw[][->] (0.3,2) -- (2.2,3.8);
\draw[dashed] (1.5,3) circle (3cm);
\draw[dashed][->] (-1.5,1.8) -- (-2.5,1.8);
\draw[dashed][->] (-0.2,1) -- (-0.2,0);
\fill[black] (0,2) circle (.0001cm) node[align=left,   above]{$T\quad$};
\fill[black] (2,4) circle (.0001cm) node[align=left,   above]{$P\quad$};
\end{tikzpicture}
\hspace{3cm}
\begin{tikzpicture}[scale=0.35]
\draw[<-] (2,0) -- (2,4);
\draw[<-] (-2,4) -- (2,4);
\draw[] (2,4) -- (4,4);;
\draw[->] (2,4) -- (4,6);
\draw[dashed] (1.5,3) circle (3cm);
\draw[dashed][->] (-1.5,3.8) -- (-2.5,3.8);
\draw[dashed][->] (1.8,1) -- (1.8,0);
\end{tikzpicture}
\caption{An example to illustrate the rearrangement when $T$ is adjacent to semiuniform faces.}
\label{fig;first_kind}
\end{figure}
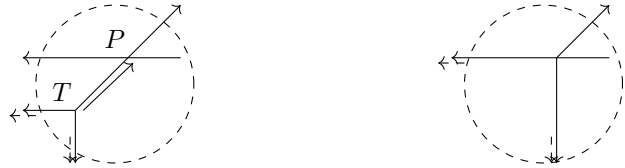

Firstly, we consider a subdivision $\mathcal{N}$ with $m$ triangular faces in the interior. We now show that for any such subdivision $\mathcal{N}$, we can always construct a subdivision $\mathcal{N}'$, such that $\mathcal{N}'$ has exactly $m-1$ triangular faces, via a rearrangement of $\mathcal{L}$ to $\mathcal{L}'$ . We consider a triangular face $T$ in $\mathcal{N}$, dual to $l'$ in $\mathcal{L}$, which we rearrange to obtain a stable intersection in order to construct the subdivision $\mathcal{N}'$. We go through the following cases based on the types of faces adjacent to $T$ in $\mathcal{N}$,

\begin{enumerate}
    \item If $T$ has at least one semiuniform face adjacent to it, dual to a stable intersection of first kind $P$.
    \end{enumerate}

We move the vertex of the line $l'$, dual to $T$, along with coaxial lines towards $P$, such that the vertex of $l'$ is superimposed on the point $P$, illustrated in the Figure \ref{fig;first_kind}. If during the rearrangement, any rays of the lines coaxial to $l'$ meet the vertex of another line, which might result in a reduction in the total number of triangular faces, we can consider a local perturbation of the vertex of such a line, along the half ray, and in this way we can prevent such a situation. In this way we obtain a subdivision $\mathcal{N}'$ with exactly $m-1$ triangular faces, via a local rearrangement. We also notice that the determined semiuniform face dual to the point $P$ in $\mathcal{N}$, ceases to exist in $\mathcal{N'}$, since the vertex of $l'$ gets superimposed on $P$.

\begin{enumerate}[resume]
    \item If $T$ is adjacent only to non-uniform faces, with at least one of the adjacent non-uniform faces being five or six edged. 
\end{enumerate}

If $T$ is adjacent to non-uniform faces in $\mathcal{N}$, then by the definition of determined faces from the Figure \ref{fig:determined}, we realize that $T$ determines uniquely at least one non-adjacent semiuniform face dual to a stable intersection of first kind $P$, in $\mathcal{N}$. We move the vertex of the line $l'$ dual to $T$ (along with any coaxial lines to $l'$ if there exist any), illustrated in Figure \ref{fig;five_edged}, such that it meets the half ray of another line in $\mathcal{L}$ and there is an effective decrease in the number of triangular faces by $1$ (in our example we assume $P_{2}$ to be the face which has to be a five or six edged face). We show the location of lines coaxial to $l'$ (if present) by a dotted arrow along the ray of coaxiality in the rearrangement. If during the rearrangement, any rays of the lines coaxial to $l'$ meet the vertex of another line, which might result in the reduction in the total number of triangular faces, we can consider a local perturbation of the vertex of such a line, along the half ray, and in this way we can prevent such a situation. Hence in this way we construct a subdivision $\mathcal{N}'$  with exactly $m-1$ triangular faces, via a local rearrangement. We also observe that the determined semiuniform face dual to $P$ in $\mathcal{N}$ no more remains a determined semiuniform face in $\mathcal{N'}$, because firstly by the definition of determined faces, the face dual to $P$ cannot be a hexagon. Additionally, out of the four edges of the face dual to $P$, only at two edges can it be adjacent to triangular faces, and we realize that in $\mathcal{N}'$ at both these edges, the face is adjacent to non-triangular faces. Hence, the face dual to $P$ cannot be a determined face by virtue of being adjacent to a triangular face in $\mathcal{N}'$. Also, it can neither be a non-adjacent determined face, since the face dual to $P$ was the unique non-adjacent determined face with respect to $T$, and the triangular face $T$ no longer exists in $\mathcal{N}'$.

\begin{figure}\centering
\begin{tikzpicture}[scale=0.35]
\draw[] (0,2) -- (2,4);
\draw[<-] (2,-2) -- (2,4);
\draw[<-] (-2,4) -- (2,4);
\draw[] (2,4) -- (5,4);
\draw[->] (0,2) -- (-2,2);
\draw[->] (0,2) -- (0,-2);
\draw[->] (2,4) -- (4,6);
\draw[->] (0,0) -- (-2,0);
\draw[->] (0,0) -- (6,6);
\draw[] (0,0) -- (-2,-2);
\draw[][->] (0,2.2) -- (0,4);
\draw[dashed] (1,2.5) circle (3.5cm);
\draw[red,thick,dashed] (2,2) circle (0.2cm);
\fill[black] (0,2) circle (.0001cm) node[align=left,   above]{$T\quad$};
\fill[black] (2,2) circle (.0001cm) node[align=left,   above]{$P\quad$};
\fill[black] (0,0) circle (.0001cm) node[align=left,   above]{$P_{1}\quad$};
\fill[black] (2,4) circle (.0001cm) node[align=left,   above]{$P_{2}\quad$};
\end{tikzpicture}
\hspace{3cm}
\begin{tikzpicture}[scale=0.35]
\draw[<-] (2,-2) -- (2,4);
\draw[<-] (-2,4) -- (2,4);
\draw[] (2,4) -- (5,4);;
\draw[->] (0,4) -- (0,-2);
\draw[->] (2,4) -- (4,6);
\draw[->] (0,4) -- (2,6);
\draw[->] (0,0) -- (-2,0);
\draw[->] (0,0) -- (6,6);
\draw[] (0,0) -- (-2,-2);
\draw[dashed] (1,2.5) circle (3.5cm);
\fill[black] (2,2) circle (.0001cm) node[align=left,   above]{$P\quad$};
\fill[black] (0,0) circle (.0001cm) node[align=left,   above]{$P_{1}\quad$};
\fill[black] (2,4) circle (.0001cm) node[align=left,   above]{$P_{2}\quad$};
\end{tikzpicture}
\caption{An example to illustrate the rearrangement when $T$ is adjacent to five or six edged non-uniform faces.}
\label{fig;five_edged}
\end{figure}
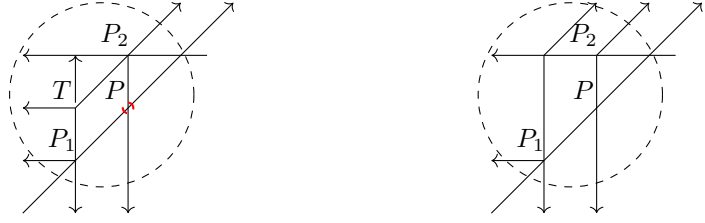

 
 \begin{enumerate}[resume]
    \item If $T$ is adjacent to only four-edged non-uniform faces.
\end{enumerate}

\begin{figure}\centering
\begin{tikzpicture}[scale=0.45]
\draw[] (1.2,1.2) -- (-0.5,1.2);
\draw[->] (1.2,1.2) -- (1.2,-2);
\draw[->] (1.2,1.2) -- (3.5,3.5);
\draw[->][] (-0.5,1.2) -- (-0.5,-3);
\draw[->][] (-0.5,1.2) -- (2,3.7);
\draw[][] (-0.5,1.2) -- (-2,1.2);
\draw[->][] (2,2) -- (3.5,3.5);
\draw[->][] (2,2) -- (2,-2.5);
\draw[->][] (2,2) -- (-2,2);
\draw[->][] (1.2,-0.5) -- (-3,-0.5);
\draw[->][] (1.2,-0.5) -- (1.2,-2);
\draw[->][] (1.2,-0.5) -- (4.2,2.5);
\draw[dashed] (1,1) circle (3cm);
\draw[red,thick,dashed] (2,0.3) circle (0.2cm);
\draw[dashed][->] (-0.6,1) -- (-1.5,1);
\draw[dashed][->] (0.2,2.1) -- (1.2,3.1);
\draw[dashed][->] (0.2,2.1) -- (-1,2.1);
\draw[dashed][->] (-0.7,-0.7) -- (-2,-0.7);
\draw[dashed][->] (-0.7,-0.7) -- (-0.7,-1.7);
\draw[dashed][->] (1.1,-0.7) -- (1.1,-1.7);
\draw[dashed][->] (2.15,0.3) -- (3.15,1.3);
\draw[dashed][->] (2.15,0.3) -- (2.15,-1.5);
\draw[dashed][->] (-0.7,-0.7) -- (-0.7,-1.7);
\draw[dashed][->] (2.2,2) -- (3.2,3);
\fill[black] (2,0.3) circle (.0001cm) node[align=left,   above]{$P\quad$};
\fill[black] (0.3,2) circle (.0001cm) node[align=left,   above]{$Q\quad$};
\fill[black] (-0.5,-0.5) circle (.0001cm) node[align=left,   above]{$R\quad$};
\fill[black] (1.2,1.2) circle (.0001cm) node[align=left,   above]{$T\quad$};
\fill[black] (2,2) circle (.0001cm) node[align=left,   above]{$l_{1}\quad$};
\fill[black] (1.2,-0.5) circle (.0001cm) node[align=left,   above]{$l_{2}\quad$};
\fill[black] (-0.5,1.2) circle (.0001cm) node[align=left,   above]{$l_{3}\quad$};
\end{tikzpicture}
\hspace{3cm}
\begin{tikzpicture}[scale=0.45]
\draw[] (0,2) -- (3,2);
\draw[] (3,2) -- (3,5);
\draw[] (0,5) -- (3,5);
\draw[] (0,2) -- (0,5);
\draw[] (0,2) -- (3,5);
\draw[->] (0,2) -- (-2,2);
\draw[->] (0,2) -- (0,0);
\draw[->] (0,5) -- (-2,5);
\draw[->] (0,5) -- (2,7);
\draw[->] (3,2) -- (3,0);
\draw[->] (3,2) -- (5,4);
\draw[->] (3,5) -- (5,7);
\draw[dashed] (1.5,3) circle (3cm);
\draw[dashed][->] (3.2,5) -- (4.7,6.5);
\draw[dashed][->] (0.2,5) -- (1.7,6.5);
\draw[dashed][->] (-0.2,1.8) -- (-2,1.8);
\draw[dashed][->] (-0.2,1.8) -- (-0.2,0);
\draw[dashed][->] (3.2,2) -- (4.7,3.5);
\draw[dashed][->] (3.2,2) -- (3.2,0);
\fill[black] (3,2) circle (.0001cm) node[align=left,   above]{$l_{2}\quad$};
\fill[black] (3,5) circle (.0001cm) node[align=left,   above]{$l_{1}\quad$};
\fill[black] (0,5) circle (.0001cm) node[align=left,   above]{$l_{3}\quad$};
\end{tikzpicture}
\hspace{4cm}
\begin{tikzpicture}[scale=0.45]
\draw[] (1.2,1.2) -- (-0.5,1.2);
\draw[->] (1.2,1.2) -- (1.2,-2);
\draw[->] (1.2,1.2) -- (3.5,3.5);
\draw[->][] (-0.5,1.2) -- (-2,1.2);
\draw[->][] (2,2) -- (3.5,3.5);
\draw[->][] (2,2) -- (2,-2.5);
\draw[->][] (2,2) -- (-2,2);
\draw[->][] (1.2,-0.5) -- (-3,-0.5);
\draw[->][] (1.2,-0.5) -- (1.2,-2);
\draw[->][] (1.2,-0.5) -- (4.2,2.5);
\draw[dashed] (1,1) circle (3cm);
\draw[red,thick,dashed] (2,0.3) circle (0.2cm);
\fill[black] (2,0.3) circle (.0001cm) node[align=left,   above]{$P\quad$};
\fill[black] (1.2,1.2) circle (.0001cm) node[align=left,   above]{$T\quad$};
\fill[black] (2,2) circle (.0001cm) node[align=left,   above]{$l_{1}\quad$};
\fill[black] (1.2,-0.5) circle (.0001cm) node[align=left,   above]{$l_{2}\quad$};
\end{tikzpicture}
\hspace{3cm}
\begin{tikzpicture}[scale=0.45]
\draw[] (0,2) -- (3,2);
\draw[->] (3,5) -- (-2,5);
\draw[] (3,2) -- (3,5);
\draw[] (0,2) -- (3,5);
\draw[->] (0,2) -- (-2,2);
\draw[->] (0,2) -- (0,0);
\draw[->] (3,2) -- (3,0);
\draw[->] (3,2) -- (5,4);
\draw[->] (3,5) -- (5,7);
\draw[dashed] (1.5,3) circle (3cm);
\fill[black] (3,2) circle (.0001cm) node[align=left,   above]{$l_{2}\quad$};
\fill[black] (3,5) circle (.0001cm) node[align=left,   above]{$l_{1}\quad$};
\end{tikzpicture}
\caption{All cases where $T$ is adjacent to two or three four edged non uniform faces along with the corresponding rearrangement $\mathcal{L}'$.}
\label{fig;four_edged}
\end{figure}
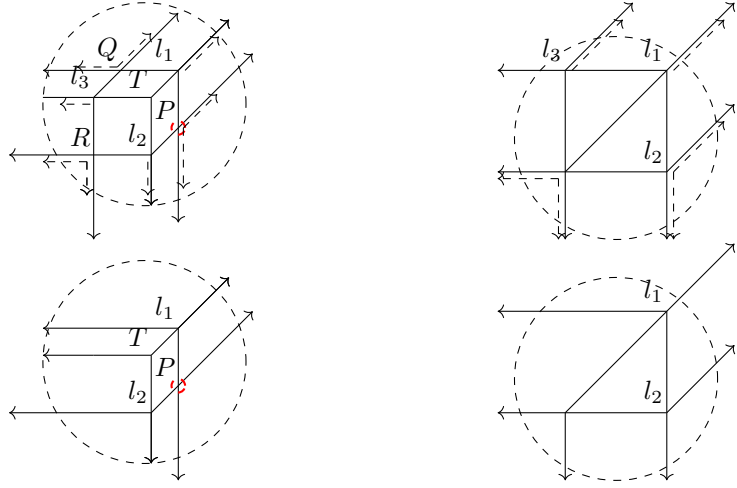

Firstly, by the definition of determined faces from the Figure \ref{fig:determined}, we realize that $T$ determines uniquely at least one non-adjacent semiuniform face dual to a stable intersection of first kind $P$ in $\mathcal{N}$. We notice that in such this case, we cannot obtain $\mathcal{N}'$ by the movement of just $l'$ and its coaxial lines since it results in an increase in the number of triangular faces. However, we observe that with a local rearrangement of $l'$ along with its neighbouring lines which are coaxial to $l'$, we can obtain $\mathcal{N}'$. When $T$ is adjacent to three or two such four edged faces, the local rearrangement is illustrated in Figure \ref{fig;four_edged}. In the first case we see that no lines can be present inside the hexagon $Pl_{1}Ql_{3}Rl_{2}$, where we abuse the notation to denote $l_{i}$ as the vertex of the line $l_{i}, i \in \{ 1,2,3 \}$, because that would contradict the adjacency of the faces dual to vertices of $l_{1},l_{2}$,$l_{3}$ and $T$. Also other lines coaxial to any of the $l_{i}$'s (if present) are depicted by dotted arrows in the figure. Essentially, one can think of this rearrangement as moving the line $l_{3}$ and $l_{2}$ along with the coaxial lines (if present) on the half rays not shared with $l'$, such that the vertices of $l_{2}$ and $l_{3}$ lie on the segments $Ql_{1}$ and $Pl_{1}$ respectively and one of the rays from each $l_{3}$ and $l_{2}$ meets the vertex of $l'$. In this way we obtain a subdivision $\mathcal{N}'$ with one less triangular face. Once again if during the rearrangement, any rays of the lines coaxial to $l'$ meet the vertex of another line, which might result in the reduction in the total number of triangular faces, we can consider a local perturbation of the vertex of such a line, along the half ray, and in this way we can prevent such a situation. A similar argument works for the remaining case in Figure \ref{fig;four_edged}. Also, we realize that the face dual to $P$ ceases to exist as we go from $\mathcal{N}$ to $\mathcal{N'}$, and this is illustrated in the Figure \ref{fig;four_edged}.

Hence, we see that in all cases for any subdivision $\mathcal{N}$ we can perform a rearrangement of $\mathcal{L}$ to $\mathcal{L}'$, to obtain a subdivision $\mathcal{N}'$ with exactly $m-1$ triangular faces. Also, we notice that as we change from $\mathcal{N}$ to $\mathcal{N'}$, there always exists a determined semiuniform face, dual to a stable intersection of first kind ($P$), which either ceases to exist in $\mathcal{N'}$ (Case (1) and (3)) or does not remain a determined semiuniform face in $\mathcal{N'}$ (Case (2)). Hence, there exists a determined semiuniform face in $\mathcal{N}$, which can never contribute to the total count of determined semiuniform faces in $\mathcal{N'}$. We now invoke the induction hypothesis for $\mathcal{N'}$ with $m-1$ triangular faces and we obtain,

\[ k \geq n(T_{1} \cup T_{2} \hdots \cup T_{m-1}) \geq m - 1 \]

Since, the face dual to $P$ cannot contribute to the $m-1$ faces determined by triangular faces present in $\mathcal{N}'$. Hence for $\mathcal{N}$, we have 

\[  n(T_{1} \cup T_{2} \hdots \cup T_{m}) \geq  n(T_{1} \cup T_{2} \hdots \cup T_{m-1}) + 1 \geq m-1 + 1 \geq m \]
     
Therefore, we realize that for all cases, given a subdivision $\mathcal{N}$ with $m$ triangular faces, 

\[   k \geq n(T_{1} \cup T_{2} \hdots \cup T_{m}) \geq m  \]

Hence, the proof.

\end{proof}

\begin{theorem}\label{thm:bound}
If $\mathcal{L}$ is a tropical line arrangement of $n$ tropical lines, then it determines at least $n - 3$ stable intersections.
\end{theorem}
\begin{proof}

We try to look at all possible places where triangular faces occur in a Newton subdivision. If $v$ is the number of triangular faces present in a subdivision, then we can write $v$ as 


\[ v = p + q \]

where $p$ be the number of triangular faces present in the interior of the Newton polygon, i.e., triangular faces which are adjacent to at least two or more faces in the subdivision and $q$ be the number of triangular faces present at the corners of the Newton polygon, i.e., triangular faces which are adjacent to exactly one other face in the subdivision. It is easy to see that $q \in \{0,1,2,3\}$. Then the lower bound on the number of semiuniform faces, which are determined by these triangular faces, is $p$ (by Theorem \ref{thm:induction}). Therefore if $k$ is the total number of faces corresponding to stable intersections of first kind, then    

\[ k \geq p \]

Also, the number of stable intersections of second kind $h$ = $n - v$ (since triangular faces and faces corresponding to stable intersections of second kind are contributed by vertices of lines, hence their sum is equal to $n$).

Therefore, the total number of stable intersections $b$ is given as
\[ b = n - v + k \geq n - p - q + p  = n - q \geq n - 3 \]

Hence, $b \geq n-3$.
\end{proof}

\begin{theorem}\label{thm:sharp}
Let $\mathcal{L}$ be a tropical line arrangement of $n$ tropical lines and let $\mathcal{N}$ be its dual Newton subdivision. If $\mathcal{L}$ determines $n-3$ stable intersections, then there are three triangular faces present at the corners of the Newton polygon and $\mathcal{N}$ can not have any triangular faces in the relative interior of the Newton polygon.
\end{theorem}

\begin{proof}
If $\mathcal{L}$ determines $n-3$ stable intersections, it is the case when the bound from Theorem \ref{thm:bound} is sharp, which happens when the following equalities hold true 

\begin{equation}\label{eq:sharp1}
q = 3    
\end{equation}

and 

\begin{equation}\label{eq:sharp}
k = n(T_{1} \cup T_{2} \hdots \cup T_{m}) = m     
\end{equation}

The first equality implies that there must be three triangular faces present at the corners of the Newton polygon.

We now assume to the contrary, that there does exist triangular faces in the relative interior. We now consider one such triangular face $T$ in the relative interior of the Newton polygon. We  consider all possible cases for $T$, where it can share faces with other triangular faces in $\mathcal{N}$,

\begin{enumerate}
    \item If $T$ does not share any semiuniform faces with any other triangular face in $\mathcal{N}$.
\end{enumerate}
By theorem \ref{thm:basecase} we know that any triangular face in the relative interior determines at least three semiuniform face. Then 

\[ k = n((T_{1} \cup T_{2} \hdots \cup T_{m-1}) \cup T) = n(T_{1} \cup T_{2} \hdots \cup T_{m-1}) + n(T) \geq m-1 + 3  = m + 2 \]

which gives a contradiction to the equation \ref{eq:sharp}.

\begin{enumerate}[resume]
    \item If $T$ shares a semiuniform face with exactly one other triangular face $T_{\alpha}$ in $\mathcal{N}$.
\end{enumerate}
 
all possible cases for $T$, upto symmetry, are listed in Figure \ref{fig:Case2}.
 
\begin{figure}[H]
\begin{tikzpicture}[scale=1.2]
\draw[][] (0,0) -- (0,-0.5);
\draw[][] (0,-0.5) -- (0.5,-0.5);
\draw[][] (0.5,-0.5) -- (0,0);
\draw[] (0,0) -- (-0.5,0);
\draw[] (-0.5,0) -- (-0.5,0.5);
\draw[] (-0.5,0.5) -- (0,0);
\draw[] (-0.5,0) -- (-0.5,-0.5);
\draw[] (-0.5,-0.5) -- (0,-0.5);
\draw[] (3,0) -- (3,-0.5);
\draw[] (3,-0.5) -- (3.5,-0.5);
\draw[] (3.5,-0.5) -- (3,0);
\draw[] (3.5,-0.5) -- (3.5,0);
\draw[] (3.5,-0.5) -- (4,-0.5);
\draw[] (4,-0.5) -- (3.5,0);
\draw[] (3.5,0) -- (3,0.5);
\draw[] (3,0) -- (3,0.5);
\draw[] (6,0) -- (6,-0.5);
\draw[] (6,-0.5) -- (6.5,-0.5);
\draw[] (6.5,-0.5) -- (6,0);
\draw[] (6,0) -- (6,0.5);
\draw[] (6,0) -- (6.5,0);
\draw[] (6,0.5) -- (6.5,0);
\draw[] (6.5,0) -- (7,-0.5);
\draw[] (6.5,-0.5) -- (7,-0.5);
\fill[black] (-0.2,-0.05) circle (.0001cm) node[align=right, above]{$T\quad$};
\fill[black] (0.35,-0.55) circle (.0001cm) node[align=right, above]{$T_{\alpha}\quad$};
\fill[black] (3.25,-0.55) circle (.0001cm) node[align=right, above]{$T\quad$};
\fill[black] (3.85,-0.55) circle (.0001cm) node[align=right, above]{$T_{\alpha}\quad$};
\fill[black] (6.25,-0.55) circle (.0001cm) node[align=right, above]{$T\quad$};
\fill[black] (6.35,-0.05) circle (.0001cm) node[align=right, above]{$T_{\alpha}\quad$};
\end{tikzpicture}
\qquad
\begin{tikzpicture}[scale=1.2]
\draw[][] (0,0) -- (0,-0.5);
\draw[][] (0,-0.5) -- (0.5,-0.5);
\draw[][] (0.5,-0.5) -- (0,0);
\draw[] (0,-0.5) -- (0,-1);
\draw[] (0.5,-0.5) -- (1,-1);
\draw[] (0,-1) -- (0.5,-1.5);
\draw[] (0.5,-1.5) -- (1,-1.5);
\draw[] (1,-1) -- (1,-1.5);
\draw[] (1,-1) -- (1.5,-1.5);
\draw[] (1,-1.5) -- (1.5,-1.5);
\fill[black] (0.3,-0.55) circle (.0001cm) node[align=right, above]{$T\quad$};
\fill[black] (1.35,-1.55) circle (.0001cm) node[align=right, above]{$T_{\alpha}\quad$};
\end{tikzpicture}
\caption{All possibilities for $T$, when it shares a semiuniform face with another triangular face}
\label{fig:Case2}
\end{figure}
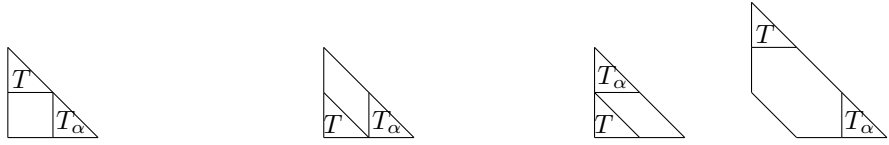

We realize that in all such cases, when we consider all possible adjacent faces to $T$, for all of them $n(T) = 4$, and none of the $m-2$ triangular faces apart from $T$ and $T_{\alpha}$, can determine the four faces determined by $T$, because that would contradict the fact that $T$ can share faces with exactly one other triangular face. Also, by Theorem \ref{thm:induction}, for the $m-2$ triangular faces apart from $T$ and $T_{\alpha}$, 

\[ n(T_{1} \cup T_{2} \hdots \cup T_{m-2}) \geq m-2   \]
 
therefore, 

\[ k = n(T_{1} \cup T_{2} \hdots \cup T_{m}) \geq n(T_{1} \cup T_{2} \hdots \cup T_{m-2}) + n(T) \geq m-2 + 4 = m+2  \]

which again gives a contradiction to the equation \ref{eq:sharp}.

We also remark that for this case and all subsequent cases, semiuniform faces which are parallelograms, and are determined by two different triangular faces, can not have edge lengths greater than one, since they share one edge, per pair of parallel edges, with a triangular face, whose edges always have unit lattice length. Hence, for all cases, the parallelogram faces are of unit lattice length. However, for hexagonal faces, edges not adjacent with triangular faces, can be of lattice length greater than one, although this does not change the count of determined faces for $T$, i.e., $n(T)$, rather it only enlarges the lengths of the edges adjacent to the hexagonal face. Hence, in our considerations, we would consider all hexagonal faces having unit lattice length.

\begin{enumerate}[resume]
    \item If $T$ shares a semiuniform face with exactly two other triangular faces $T_{\alpha}$ and $T_{\beta}$ in $\mathcal{N}$.
\end{enumerate}

all possible cases for $T$, upto symmetry, are listed in Figure \ref{fig:Case3} and Figure \ref{fig:hexagon_case}.

\begin{figure}[H]
\centering
\begin{tikzpicture}[scale=0.6]
\draw[][] (0,0) -- (0,-1);
\draw[][] (0,0) -- (1,-1);
\draw[][] (1,-1) -- (1,-2);
\draw[][] (0,-1) -- (1,-2);
\draw[][] (0,-1) -- (0,-2);
\draw[][] (0,-2) -- (1,-2);
\draw[][] (0,-2) -- (0,-3);
\draw[][] (0,-3) -- (1,-3);
\draw[][] (1,-3) -- (1,-2);
\draw[][] (1,-1) -- (2,-2);
\draw[][] (2,-2) -- (2,-3);
\draw[][] (1,-2) -- (2,-3);
\draw[][] (1,-3) -- (2,-3);
\draw[][] (2,-2) -- (3,-3);
\draw[][] (2,-3) -- (3,-3);
\fill[black] (1.5,-2.9) circle (.0001cm) node[align=right, above]{$T\quad$};
\fill[black] (0.6,-1.95) circle (.0001cm) node[align=right, above]{$T_{\alpha}\quad$};
\fill[black] (2.6,-3.1) circle (.0001cm) node[align=right, above]{$T_{\beta}\quad$};
\draw[][] (4,0) -- (4,-1);
\draw[][] (4,0) -- (5,-1);
\draw[][] (6,-2) -- (5,-2);
\draw[][] (4,-1) -- (5,-1);
\draw[][] (4,-1) -- (5,-2);
\draw[][] (4,-1) -- (4,-2);
\draw[][] (4,-2) -- (5,-2);
\draw[][] (4,-2) -- (4,-3);
\draw[][] (4,-3) -- (5,-3);
\draw[][] (5,-3) -- (5,-2);
\draw[][] (5,-1) -- (6,-2);
\draw[][] (6,-2) -- (5,-2);
\draw[][] (5,-2) -- (6,-3);
\draw[][] (5,-3) -- (6,-3);
\draw[][] (6,-2) -- (7,-3);
\draw[][] (6,-3) -- (7,-3);
\fill[black] (5.6,-3) circle (.0001cm) node[align=right, above]{$T_{\alpha}\quad$};
\fill[black] (4.5,-1.9) circle (.0001cm) node[align=right, above]{$T\quad$};
\fill[black] (4.6,-1.1) circle (.0001cm) node[align=right, above]{$T_{\beta}\quad$};
\draw[][] (8,0) -- (8,-1);
\draw[][] (8,0) -- (9,-1);
\draw[][] (10,-2) -- (9,-2);
\draw[][] (8,-1) -- (9,-1);
\draw[][] (9,-1) -- (9,-2);
\draw[][] (8,-1) -- (8,-2);
\draw[][] (8,-2) -- (9,-2);
\draw[][] (8,-2) -- (8,-3);
\draw[][] (8,-3) -- (9,-3);
\draw[][] (9,-3) -- (9,-2);
\draw[][] (9,-1) -- (10,-2);
\draw[][] (10,-2) -- (9,-2);
\draw[][] (10,-2) -- (10,-3);
\draw[][] (9,-3) -- (10,-3);
\draw[][] (10,-2) -- (11,-3);
\draw[][] (10,-3) -- (11,-3);
\fill[black] (10.6,-3) circle (.0001cm) node[align=right, above]{$T_{\alpha}\quad$};
\fill[black] (9.5,-1.9) circle (.0001cm) node[align=right, above]{$T\quad$};
\fill[black] (8.6,-1.1) circle (.0001cm) node[align=right, above]{$T_{\beta}\quad$};
\draw[][] (12,0) -- (12,-1);
\draw[][] (12,0) -- (13,-1);
\draw[][] (14,-2) -- (13,-2);
\draw[][] (12,-1) -- (13,-2);
\draw[][] (13,-1) -- (13,-2);
\draw[][] (12,-1) -- (12,-2);
\draw[][] (12,-2) -- (13,-2);
\draw[][] (12,-2) -- (12,-3);
\draw[][] (12,-3) -- (13,-3);
\draw[][] (13,-3) -- (13,-2);
\draw[][] (13,-1) -- (14,-2);
\draw[][] (14,-2) -- (13,-2);
\draw[][] (13,-2) -- (14,-3);
\draw[][] (13,-3) -- (14,-3);
\draw[][] (14,-2) -- (15,-3);
\draw[][] (14,-3) -- (15,-3);
\fill[black] (13.6,-3) circle (.0001cm) node[align=right, above]{$T_{\alpha}\quad$};
\fill[black] (13.5,-1.9) circle (.0001cm) node[align=right, above]{$T\quad$};
\fill[black] (12.6,-2.1) circle (.0001cm) node[align=right, above]{$T_{\beta}\quad$};
\draw[][] (-2,-4) -- (-2,-5);
\draw[][] (-2,-4) -- (-1,-5);
\draw[][] (-1,-5) -- (-2,-5);
\draw[][] (-2,-5) -- (-1,-6);
\draw[][] (-2,-5) -- (-2,-6);
\draw[][] (-2,-6) -- (-1,-6);
\draw[][] (-2,-6) -- (-2,-7);
\draw[][] (-2,-7) -- (-1,-7);
\draw[][] (0,-7) -- (-1,-6);
\draw[][] (-1,-5) -- (0,-6);
\draw[][] (0,-6) -- (-1,-6);
\draw[][] (-2,-6) -- (-1,-7);
\draw[][] (-2,-7) -- (0,-7);
\draw[][] (0,-6) -- (1,-7);
\draw[][] (0,-7) -- (1,-7);
\fill[black] (-1.35,-7) circle (.0001cm) node[align=right, above]{$T_{\alpha}\quad$};
\fill[black] (-1.5,-5.9) circle (.0001cm) node[align=right, above]{$T\quad$};
\fill[black] (-1.35,-5.1) circle (.0001cm) node[align=right, above]{$T_{\beta}\quad$};
\draw[][] (2,-4) -- (2,-5);
\draw[][] (2,-4) -- (3,-5);
\draw[][] (4,-6) -- (3,-6);
\draw[][] (3,-6) -- (3,-5);
\draw[][] (2,-5) -- (3,-6);
\draw[][] (2,-5) -- (2,-6);
\draw[][] (2,-6) -- (3,-6);
\draw[][] (2,-6) -- (2,-7);
\draw[][] (2,-7) -- (3,-7);
\draw[][] (3,-7) -- (2,-6);
\draw[][] (3,-5) -- (4,-6);
\draw[][] (4,-6) -- (3,-6);
\draw[][] (3,-6) -- (4,-7);
\draw[][] (4,-7) -- (4,-7);
\draw[][] (4,-6) -- (5,-7);
\draw[][] (4,-7) -- (5,-7);
\draw[][] (4,-7) -- (3,-7);
\fill[black] (2.65,-7.05) circle (.0001cm) node[align=right, above]{$T_{\alpha}\quad$};
\fill[black] (2.5,-5.9) circle (.0001cm) node[align=right, above]{$T\quad$};
\fill[black] (3.65,-6.1) circle (.0001cm) node[align=right, above]{$T_{\beta}\quad$};
\draw[][] (6,-4) -- (6,-5);
\draw[][] (6,-4) -- (7,-5);
\draw[][] (6,-5) -- (7,-6);
\draw[][] (7,-5) -- (7,-6);
\draw[][] (6,-5) -- (6,-6);
\draw[][] (6,-6) -- (7,-7);
\draw[][] (6,-6) -- (6,-7);
\draw[][] (6,-7) -- (7,-7);
\draw[][] (7,-7) -- (7,-6);
\draw[][] (7,-5) -- (8,-6);
\draw[][] (8,-7) -- (7,-6);
\draw[][] (8,-6) -- (8,-7);
\draw[][] (7,-7) -- (8,-7);
\draw[][] (8,-6) -- (9,-7);
\draw[][] (8,-7) -- (9,-7);
\fill[black] (8.6,-7) circle (.0001cm) node[align=right, above]{$T_{\alpha}\quad$};
\fill[black] (7.5,-6.9) circle (.0001cm) node[align=right, above]{$T\quad$};
\fill[black] (6.6,-7.05) circle (.0001cm) node[align=right, above]{$T_{\beta}\quad$};
\draw[][] (10,-4) -- (10,-5);
\draw[][] (10,-4) -- (11,-5);
\draw[][] (12,-6) -- (11,-6);
\draw[][] (10,-5) -- (11,-6);
\draw[][] (11,-5) -- (11,-6);
\draw[][] (10,-5) -- (10,-6);
\draw[][] (10,-6) -- (11,-7);
\draw[][] (10,-6) -- (10,-7);
\draw[][] (10,-7) -- (11,-7);
\draw[][] (11,-7) -- (11,-6);
\draw[][] (11,-5) -- (12,-6);
\draw[][] (12,-6) -- (11,-6);
\draw[][] (11,-6) -- (12,-7);
\draw[][] (11,-7) -- (12,-7);
\draw[][] (12,-6) -- (13,-7);
\draw[][] (12,-7) -- (13,-7);
\fill[black] (11.5,-6.9) circle (.0001cm) node[align=right, above]{$T\quad$};
\fill[black] (11.6,-6) circle (.0001cm) node[align=right, above]{$T_{\alpha}\quad$};
\fill[black] (10.6,-7.05) circle (.0001cm) node[align=right, above]{$T_{\beta}\quad$};
\draw[][] (14,-4) -- (14,-5);
\draw[][] (14,-4) -- (15,-5);
\draw[][] (15,-6) -- (15,-6);
\draw[][] (14,-5) -- (15,-6);
\draw[][] (14,-5) -- (14,-6);
\draw[][] (14,-6) -- (15,-6);
\draw[][] (14,-6) -- (14,-7);
\draw[][] (14,-7) -- (15,-7);
\draw[][] (16,-7) -- (16,-6);
\draw[][] (15,-5) -- (16,-6);
\draw[][] (16,-6) -- (15,-6);
\draw[][] (15,-6) -- (15,-5);
\draw[][] (15,-7) -- (16,-7);
\draw[][] (16,-6) -- (17,-7);
\draw[][] (16,-7) -- (17,-7);
\draw[][] (15,-6) -- (15,-7);
\fill[black] (16.6,-7) circle (.0001cm) node[align=right, above]{$T_{\alpha}\quad$};
\fill[black] (15.5,-5.9) circle (.0001cm) node[align=right, above]{$T\quad$};
\fill[black] (14.6,-6.1) circle (.0001cm) node[align=right, above]{$T_{\beta}\quad$};
\end{tikzpicture}
\caption{Possibilities for $T$, when it shares semiuniform faces with exactly two other triangular faces}
\label{fig:Case3}
\end{figure}
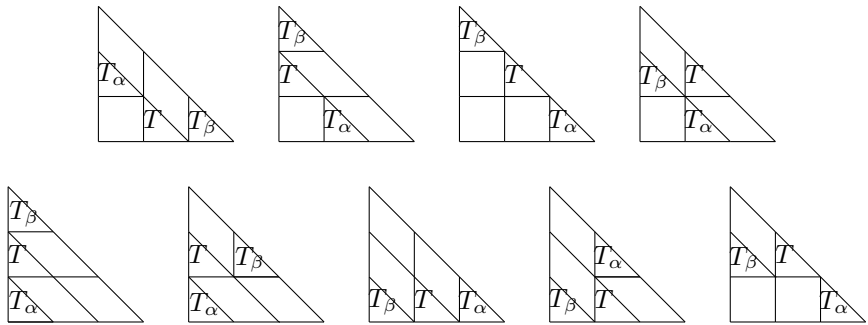

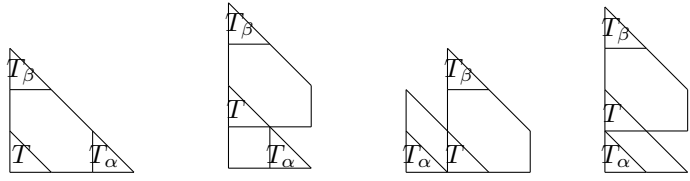
\begin{figure}[H]
\centering
\begin{tikzpicture}[scale=1.1]
\draw[][] (0,0) -- (0,-0.5);
\draw[][] (0,-0.5) -- (0.5,-0.5);
\draw[][] (0.5,-0.5) -- (0,0);
\draw[] (0,-0.5) -- (0,-1);
\draw[] (0.5,-0.5) -- (1,-1);
\draw[] (0,-1) -- (0.5,-1.5);
\draw[] (0.5,-1.5) -- (1,-1.5);
\draw[] (1,-1) -- (1,-1.5);
\draw[] (1,-1) -- (1.5,-1.5);
\draw[] (1,-1.5) -- (1.5,-1.5);
\draw[] (0,-1) -- (0,-1.5);
\draw[] (0,-1.5) -- (0.5,-1.5);
\fill[black] (0.3,-0.55) circle (.0001cm) node[align=right, above]{$T_{\beta}\quad$};
\fill[black] (0.3,-1.5) circle (.0001cm) node[align=right, above]{$T\quad$};
\fill[black] (1.3,-1.55) circle (.0001cm) node[align=right, above]{$T_{\alpha}\quad$};
\end{tikzpicture}
\qquad
\begin{tikzpicture}[scale=1.1]
\draw[][] (0,0) -- (0,-0.5);
\draw[][] (0,-0.5) -- (0.5,-0.5);
\draw[][] (0.5,-0.5) -- (0,0);
\draw[] (0,-0.5) -- (0,-1);
\draw[] (0.5,-0.5) -- (1,-1);
\draw[] (0,-1) -- (0.5,-1.5);
\draw[] (0.5,-1.5) -- (1,-1.5);
\draw[] (1,-1) -- (1,-1.5);
\draw[] (0,-1) -- (0,-1.5);
\draw[] (0,-1.5) -- (0.5,-1.5);
\draw[] (0,-1.5) -- (0,-2);
\draw[] (0,-2) -- (0.5,-2);
\draw[] (0.5,-2) -- (0.5,-1.5);
\draw[] (0.5,-1.5) -- (1,-2);
\draw[] (0.5,-2) -- (1,-2);
\fill[black] (0.3,-0.55) circle (.0001cm) node[align=right, above]{$T_{\beta}\quad$};
\fill[black] (0.25,-1.5) circle (.0001cm) node[align=right, above]{$T\quad$};
\fill[black] (0.8,-2.1) circle (.0001cm) node[align=right, above]{$T_{\alpha}\quad$};
\end{tikzpicture}
\qquad
\begin{tikzpicture}[scale=1.1]
\draw[][] (0,0) -- (0,-0.5);
\draw[][] (0,-0.5) -- (0.5,-0.5);
\draw[][] (0.5,-0.5) -- (0,0);
\draw[] (0,-0.5) -- (0,-1);
\draw[] (0.5,-0.5) -- (1,-1);
\draw[] (0,-1) -- (0.5,-1.5);
\draw[] (0.5,-1.5) -- (1,-1.5);
\draw[] (1,-1) -- (1,-1.5);
\draw[] (0,-1) -- (0,-1.5);
\draw[] (0,-1.5) -- (0.5,-1.5);
\draw[] (0,-1.5) -- (-0.5,-1.5);
\draw[] (-0.5,-1.5) -- (-0.5,-1);
\draw[] (-0.5,-1) -- (0,-1.5);
\draw[] (0,-1) -- (-0.5,-0.5);
\draw[] (-0.5,-1) -- (-0.5,-0.5);
\fill[black] (0.3,-0.55) circle (.0001cm) node[align=right, above]{$T_{\beta}\quad$};
\fill[black] (0.25,-1.5) circle (.0001cm) node[align=right, above]{$T\quad$};
\fill[black] (-0.2,-1.55) circle (.0001cm) node[align=right, above]{$T_{\alpha}\quad$};
\end{tikzpicture}
\qquad
\begin{tikzpicture}[scale=1.1]
\draw[][] (0,0) -- (0,-0.5);
\draw[][] (0,-0.5) -- (0.5,-0.5);
\draw[][] (0.5,-0.5) -- (0,0);
\draw[] (0,-0.5) -- (0,-1);
\draw[] (0.5,-0.5) -- (1,-1);
\draw[] (0,-1) -- (0.5,-1.5);
\draw[] (0.5,-1.5) -- (1,-1.5);
\draw[] (1,-1) -- (1,-1.5);
\draw[] (0,-1) -- (0,-1.5);
\draw[] (0,-1.5) -- (0.5,-1.5);
\draw[] (0,-1.5) -- (0,-2);
\draw[] (0,-2) -- (0.5,-2);
\draw[] (0.5,-2) -- (0,-1.5);
\draw[] (0.5,-2) -- (1,-2);
\draw[] (1,-2) -- (0.5,-1.5);
\fill[black] (0.3,-0.55) circle (.0001cm) node[align=right, above]{$T_{\beta}\quad$};
\fill[black] (0.25,-1.5) circle (.0001cm) node[align=right, above]{$T\quad$};
\fill[black] (0.3,-2.05) circle (.0001cm) node[align=right, above]{$T_{\alpha}\quad$};
\end{tikzpicture}
\caption{The case when $T$ shares semiuniform with two other triangular faces, with one of the determined faces being a hexagon}
\label{fig:hexagon_case}
\end{figure}

We realize that in all cases in Figure \ref{fig:Case3}, when we consider all possible adjacent faces for $T$, $n(T) = 5$, and for the first case in Figure \ref{fig:hexagon_case}, $ n(T) = 4$, while for all others in Figure \ref{fig:hexagon_case}, $n(T) = 5$. Also none of the $m-3$ triangular faces apart from $T$, $T_{\alpha}$ and $T_{\beta}$, can determine the faces determined by $T$, because that would contradict the fact that $T$ can share faces with only two other triangular faces. By Theorem \ref{thm:induction}, for the $m-3$ triangular faces apart from $T$, $T_{\alpha}$ and $T_{\beta}$, we have 

\[ n(T_{1} \cup T_{2} \hdots \cup T_{m-3}) \geq m-3   \]

therefore, 

\[ k = n(T_{1} \cup T_{2} \hdots \cup T_{m}) \geq n(T_{1} \cup T_{2} \hdots \cup T_{m-3}) + n(T) \geq m-3 + 4 = m+1  \]

which again gives a contradiction to the equation \ref{eq:sharp}.

\begin{enumerate}[resume]
    \item If $T$ shares a semiuniform face with exactly three other triangular faces $T_{\alpha}$, $T_{\beta}$ and $T_{\gamma}$ in $\mathcal{N}$.
\end{enumerate}

all possible cases for $T$, upto symmetry, are listed in Figure \ref{fig:Case4_1}, Figure \ref{fig:Case4} and Figure \ref{fig:hexagon_case1}.

\begin{center}
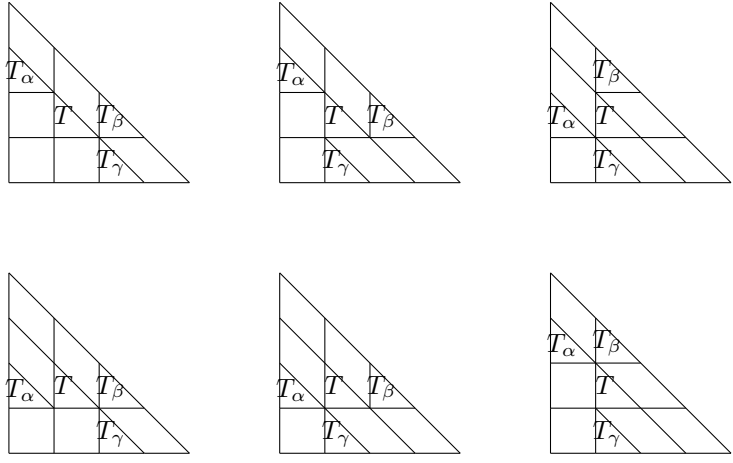
\begin{figure}
\begin{tikzpicture}[scale=0.6]
\draw[][] (0,0) -- (0,-1);
\draw[][] (0,0) -- (1,-1);
\draw[][] (1,-1) -- (1,-2);
\draw[][] (0,-1) -- (1,-2);
\draw[][] (0,-1) -- (0,-2);
\draw[][] (0,-2) -- (1,-2);
\draw[][] (0,-2) -- (0,-3);
\draw[][] (0,-3) -- (1,-3);
\draw[][] (1,-3) -- (1,-2);
\draw[][] (0,-3) -- (0,-4);
\draw[][] (0,-4) -- (1,-4);
\draw[][] (1,-4) -- (1,-3);
\draw[][] (1,-1) -- (2,-2);
\draw[][] (2,-2) -- (2,-3);
\draw[][] (1,-2) -- (2,-3);
\draw[][] (1,-3) -- (2,-3);
\draw[][] (1,-4) -- (2,-4);
\draw[][] (2,-3) -- (2,-4);
\draw[][] (2,-2) -- (3,-3);
\draw[][] (2,-3) -- (3,-3);
\draw[][] (2,-3) -- (3,-4);
\draw[][] (2,-4) -- (3,-4);
\draw[][] (3,-3) -- (4,-4);
\draw[][] (3,-4) -- (4,-4);
\fill[black] (1.5,-2.9) circle (.0001cm) node[align=right, above]{$T\quad$};
\fill[black] (0.55,-2) circle (.0001cm) node[align=right, above]{$T_{\alpha}\quad$};
\fill[black] (2.55,-3.05) circle (.0001cm) node[align=right, above]{$T_{\beta}\quad$};
\fill[black] (2.55,-4.05) circle (.0001cm) node[align=right, above]{$T_{\gamma}\quad$};
\draw[][] (6,0) -- (6,-1);
\draw[][] (6,0) -- (7,-1);
\draw[][] (7,-1) -- (7,-2);
\draw[][] (6,-1) -- (7,-2);
\draw[][] (6,-1) -- (6,-2);
\draw[][] (6,-2) -- (7,-2);
\draw[][] (6,-2) -- (6,-3);
\draw[][] (6,-3) -- (7,-3);
\draw[][] (7,-3) -- (7,-2);
\draw[][] (6,-3) -- (6,-4);
\draw[][] (6,-4) -- (7,-4);
\draw[][] (7,-4) -- (7,-3);
\draw[][] (7,-3) -- (8,-4);
\draw[][] (7,-1) -- (8,-2);
\draw[][] (8,-2) -- (8,-3);
\draw[][] (7,-2) -- (8,-3);
\draw[][] (7,-3) -- (8,-3);
\draw[][] (7,-4) -- (8,-4);
\draw[][] (8,-2) -- (9,-3);
\draw[][] (8,-3) -- (9,-3);
\draw[][] (8,-3) -- (9,-4);
\draw[][] (8,-4) -- (9,-4);
\draw[][] (9,-3) -- (10,-4);
\draw[][] (9,-4) -- (10,-4);
\fill[black] (7.5,-2.9) circle (.0001cm) node[align=right, above]{$T\quad$};
\fill[black] (6.55,-2.05) circle (.0001cm) node[align=right, above]{$T_{\alpha}\quad$};
\fill[black] (8.55,-3.05) circle (.0001cm) node[align=right, above]{$T_{\beta}\quad$};
\fill[black] (7.55,-4.05) circle (.0001cm) node[align=right, above]{$T_{\gamma}\quad$};
\draw[][] (12,0) -- (12,-1);
\draw[][] (12,0) -- (13,-1);
\draw[][] (13,-1) -- (13,-2);
\draw[][] (12,-1) -- (13,-2);
\draw[][] (12,-1) -- (12,-2);
\draw[][] (12,-2) -- (13,-3);
\draw[][] (12,-2) -- (12,-3);
\draw[][] (12,-3) -- (13,-3);
\draw[][] (13,-3) -- (13,-2);
\draw[][] (12,-3) -- (12,-4);
\draw[][] (12,-4) -- (13,-4);
\draw[][] (13,-4) -- (13,-3);
\draw[][] (13,-3) -- (14,-4);
\draw[][] (13,-1) -- (14,-2);
\draw[][] (14,-2) -- (13,-2);
\draw[][] (13,-2) -- (14,-3);
\draw[][] (13,-3) -- (14,-3);
\draw[][] (13,-4) -- (14,-4);
\draw[][] (14,-2) -- (15,-3);
\draw[][] (14,-3) -- (15,-3);
\draw[][] (14,-3) -- (15,-4);
\draw[][] (14,-4) -- (15,-4);
\draw[][] (15,-3) -- (16,-4);
\draw[][] (15,-4) -- (16,-4);
\fill[black] (13.5,-2.9) circle (.0001cm) node[align=right, above]{$T\quad$};
\fill[black] (12.55,-3) circle (.0001cm) node[align=right, above]{$T_{\alpha}\quad$};
\fill[black] (13.55,-2.05) circle (.0001cm) node[align=right, above]{$T_{\beta}\quad$};
\fill[black] (13.55,-4.05) circle (.0001cm) node[align=right, above]{$T_{\gamma}\quad$};
\draw[][] (0,-6) -- (0,-7);
\draw[][] (0,-6) -- (1,-7);
\draw[][] (1,-7) -- (1,-8);
\draw[][] (0,-7) -- (1,-8);
\draw[][] (0,-7) -- (0,-8);
\draw[][] (0,-8) -- (1,-9);
\draw[][] (0,-8) -- (0,-9);
\draw[][] (0,-9) -- (1,-9);
\draw[][] (1,-8) -- (1,-9);
\draw[][] (0,-9) -- (0,-10);
\draw[][] (0,-10) -- (1,-10);
\draw[][] (1,-10) -- (1,-9);
\draw[][] (1,-7) -- (2,-8);
\draw[][] (2,-8) -- (2,-9);
\draw[][] (1,-8) -- (2,-9);
\draw[][] (1,-9) -- (2,-9);
\draw[][] (1,-10) -- (2,-10);
\draw[][] (2,-9) -- (2,-10);
\draw[][] (2,-8) -- (3,-9);
\draw[][] (2,-9) -- (3,-9);
\draw[][] (2,-9) -- (3,-10);
\draw[][] (2,-10) -- (3,-10);
\draw[][] (3,-9) -- (4,-10);
\draw[][] (3,-10) -- (4,-10);
\fill[black] (1.5,-8.9) circle (.0001cm) node[align=right, above]{$T\quad$};
\fill[black] (0.55,-9.05) circle (.0001cm) node[align=right, above]{$T_{\alpha}\quad$};
\fill[black] (2.55,-9.05) circle (.0001cm) node[align=right, above]{$T_{\beta}\quad$};
\fill[black] (2.55,-10.05) circle (.0001cm) node[align=right, above]{$T_{\gamma}\quad$};
\draw[][] (6,-6) -- (6,-7);
\draw[][] (6,-6) -- (7,-7);
\draw[][] (7,-7) -- (7,-8);
\draw[][] (6,-7) -- (7,-8);
\draw[][] (6,-7) -- (6,-8);
\draw[][] (6,-8) -- (7,-9);
\draw[][] (6,-8) -- (6,-9);
\draw[][] (6,-9) -- (7,-9);
\draw[][] (7,-9) -- (7,-8);
\draw[][] (6,-9) -- (6,-10);
\draw[][] (6,-10) -- (7,-10);
\draw[][] (7,-10) -- (7,-9);
\draw[][] (7,-9) -- (8,-10);
\draw[][] (7,-7) -- (8,-8);
\draw[][] (8,-8) -- (8,-9);
\draw[][] (7,-8) -- (8,-9);
\draw[][] (7,-9) -- (8,-9);
\draw[][] (7,-10) -- (8,-10);
\draw[][] (8,-8) -- (9,-9);
\draw[][] (8,-9) -- (9,-9);
\draw[][] (8,-9) -- (9,-10);
\draw[][] (8,-10) -- (9,-10);
\draw[][] (9,-9) -- (10,-10);
\draw[][] (9,-10) -- (10,-10);
\fill[black] (7.5,-8.9) circle (.0001cm) node[align=right, above]{$T\quad$};
\fill[black] (6.55,-9.05) circle (.0001cm) node[align=right, above]{$T_{\alpha}\quad$};
\fill[black] (8.55,-9.05) circle (.0001cm) node[align=right, above]{$T_{\beta}\quad$};
\fill[black] (7.55,-10.05) circle (.0001cm) node[align=right, above]{$T_{\gamma}\quad$};
\draw[][] (12,-6) -- (12,-7);
\draw[][] (12,-6) -- (13,-7);
\draw[][] (13,-7) -- (13,-8);
\draw[][] (12,-7) -- (13,-8);
\draw[][] (12,-7) -- (12,-8);
\draw[][] (12,-8) -- (13,-8);
\draw[][] (12,-8) -- (12,-9);
\draw[][] (12,-9) -- (13,-9);
\draw[][] (13,-9) -- (13,-8);
\draw[][] (12,-9) -- (12,-10);
\draw[][] (12,-10) -- (13,-10);
\draw[][] (13,-10) -- (13,-9);
\draw[][] (13,-9) -- (14,-10);
\draw[][] (13,-7) -- (14,-8);
\draw[][] (14,-8) -- (13,-8);
\draw[][] (13,-8) -- (14,-9);
\draw[][] (13,-9) -- (14,-9);
\draw[][] (13,-10) -- (14,-10);
\draw[][] (14,-8) -- (15,-9);
\draw[][] (14,-9) -- (15,-9);
\draw[][] (14,-9) -- (15,-10);
\draw[][] (14,-10) -- (15,-10);
\draw[][] (15,-9) -- (16,-10);
\draw[][] (15,-10) -- (16,-10);
\fill[black] (13.5,-8.9) circle (.0001cm) node[align=right, above]{$T\quad$};
\fill[black] (12.55,-8.05) circle (.0001cm) node[align=right, above]{$T_{\alpha}\quad$};
\fill[black] (13.55,-8.05) circle (.0001cm) node[align=right, above]{$T_{\beta}\quad$};
\fill[black] (13.55,-10.05) circle (.0001cm) node[align=right, above]{$T_{\gamma}\quad$};
\end{tikzpicture}
\caption{Possibilities for $T$, when it shares semiuniform faces with three other triangular faces}
\label{fig:Case4_1}
\end{figure}
\end{center}

\begin{center}
\begin{figure}
\begin{tikzpicture}[scale=0.6]
\draw[][] (0,0) -- (0,-1);
\draw[][] (0,0) -- (1,-1);
\draw[][] (1,-1) -- (1,-2);
\draw[][] (0,-1) -- (1,-2);
\draw[][] (0,-1) -- (0,-2);
\draw[][] (0,-2) -- (1,-2);
\draw[][] (0,-2) -- (0,-3);
\draw[][] (0,-3) -- (1,-3);
\draw[][] (1,-3) -- (1,-2);
\draw[][] (1,-1) -- (2,-2);
\draw[][] (2,-2) -- (2,-3);
\draw[][] (1,-2) -- (2,-3);
\draw[][] (1,-3) -- (2,-3);
\draw[][] (2,-2) -- (3,-3);
\draw[][] (2,-3) -- (3,-3);
\draw[][] (1,-3) -- (1,-4);
\draw[][] (1,-4) -- (2,-5);
\draw[][] (2,-5) -- (3,-5);
\draw[][] (3,-5) -- (3,-4);
\draw[][] (3,-4) -- (2,-3);
\draw[][] (3,-5) -- (4,-5);
\draw[][] (4,-5) -- (3,-4);
\fill[black] (1.5,-2.9) circle (.0001cm) node[align=right, above]{$T\quad$};
\fill[black] (0.55,-2) circle (.0001cm) node[align=right, above]{$T_{\alpha}\quad$};
\fill[black] (2.55,-3) circle (.0001cm) node[align=right, above]{$T_{\beta}\quad$};
\fill[black] (3.55,-5.05) circle (.0001cm) node[align=right, above]{$T_{\gamma}\quad$};
\end{tikzpicture}
\qquad
\begin{tikzpicture}[scale=0.6]
\draw[][] (4,0) -- (4,-1);
\draw[][] (4,0) -- (5,-1);
\draw[][] (6,-2) -- (5,-2);
\draw[][] (4,-1) -- (5,-1);
\draw[][] (4,-1) -- (5,-2);
\draw[][] (4,-1) -- (4,-2);
\draw[][] (4,-2) -- (5,-2);
\draw[][] (4,-2) -- (4,-3);
\draw[][] (4,-3) -- (5,-3);
\draw[][] (5,-3) -- (5,-2);
\draw[][] (5,-1) -- (6,-2);
\draw[][] (6,-2) -- (5,-2);
\draw[][] (5,-2) -- (6,-3);
\draw[][] (5,-3) -- (6,-3);
\draw[][] (6,-2) -- (7,-3);
\draw[][] (6,-3) -- (7,-3);
\draw[][] (4,-1) -- (3,0);
\draw[][] (3,0) -- (2,0);
\draw[][] (2,0) -- (2,-1);
\draw[][] (2,-1) -- (3,-2);
\draw[][] (3,-2) -- (4,-2);
\draw[][] (2,-1) -- (2,-2);
\draw[][] (2,-2) -- (3,-2);
\fill[black] (5.55,-3) circle (.0001cm) node[align=right, above]{$T_{\alpha}\quad$};
\fill[black] (4.5,-1.9) circle (.0001cm) node[align=right, above]{$T\quad$};
\fill[black] (4.55,-1) circle (.0001cm) node[align=right, above]{$T_{\beta}\quad$};
\fill[black] (2.55,-2.05) circle (.0001cm) node[align=right, above]{$T_{\gamma}\quad$};
\end{tikzpicture}
\qquad
\begin{tikzpicture}[scale=0.6]
\draw[][] (8,0) -- (8,-1);
\draw[][] (8,0) -- (9,-1);
\draw[][] (10,-2) -- (9,-2);
\draw[][] (8,-1) -- (9,-1);
\draw[][] (9,-1) -- (9,-2);
\draw[][] (8,-1) -- (8,-2);
\draw[][] (8,-2) -- (9,-2);
\draw[][] (8,-2) -- (8,-3);
\draw[][] (8,-3) -- (9,-3);
\draw[][] (9,-3) -- (9,-2);
\draw[][] (9,-1) -- (10,-2);
\draw[][] (10,-2) -- (9,-2);
\draw[][] (10,-2) -- (10,-3);
\draw[][] (9,-3) -- (10,-3);
\draw[][] (10,-2) -- (11,-3);
\draw[][] (10,-3) -- (11,-3);
\draw[][] (9,-1) -- (9,0);
\draw[][] (9,0) -- (10,0);
\draw[][] (10,0) -- (11,-1);
\draw[][] (11,-1) -- (11,-2);
\draw[][] (11,-2) -- (10,-2);
\draw[][] (11,-1) -- (12,-2);
\draw[][] (11,-2) -- (12,-2);
\fill[black] (10.55,-3) circle (.0001cm) node[align=right, above]{$T_{\alpha}\quad$};
\fill[black] (9.5,-1.9) circle (.0001cm) node[align=right, above]{$T\quad$};
\fill[black] (8.55,-1.05) circle (.0001cm) node[align=right, above]{$T_{\beta}\quad$};
\fill[black] (11.55,-2.05) circle (.0001cm) node[align=right, above]{$T_{\gamma}\quad$};
\end{tikzpicture}
\qquad
\begin{tikzpicture}[scale=0.6]
\draw[][] (12,0) -- (12,-1);
\draw[][] (12,0) -- (13,-1);
\draw[][] (14,-2) -- (13,-2);
\draw[][] (12,-1) -- (13,-2);
\draw[][] (13,-1) -- (13,-2);
\draw[][] (12,-1) -- (12,-2);
\draw[][] (12,-2) -- (13,-2);
\draw[][] (12,-2) -- (12,-3);
\draw[][] (12,-3) -- (13,-3);
\draw[][] (13,-3) -- (13,-2);
\draw[][] (13,-1) -- (14,-2);
\draw[][] (14,-2) -- (13,-2);
\draw[][] (13,-2) -- (14,-3);
\draw[][] (13,-3) -- (14,-3);
\draw[][] (14,-2) -- (15,-3);
\draw[][] (14,-3) -- (15,-3);
\draw[][] (13,-1) -- (13,0);
\draw[][] (13,0) -- (14,0);
\draw[][] (14,0) -- (15,-1);
\draw[][] (15,-1) -- (15,-2);
\draw[][] (14,-2) -- (15,-2);
\draw[][] (15,-2) -- (16,-2);
\draw[][] (15,-1) -- (16,-2);
\fill[black] (13.55,-3.05) circle (.0001cm) node[align=right, above]{$T_{\alpha}\quad$};
\fill[black] (13.5,-1.9) circle (.0001cm) node[align=right, above]{$T\quad$};
\fill[black] (12.55,-2.05) circle (.0001cm) node[align=right, above]{$T_{\beta}\quad$};
\fill[black] (15.55,-2.05) circle (.0001cm) node[align=right, above]{$T_{\gamma}\quad$};
\end{tikzpicture}
\qquad
\begin{tikzpicture}[scale=0.6]
\draw[][] (0,-4) -- (0,-5);
\draw[][] (0,-4) -- (1,-5);
\draw[][] (1,-5) -- (0,-5);
\draw[][] (0,-5) -- (1,-6);
\draw[][] (0,-5) -- (0,-6);
\draw[][] (0,-6) -- (1,-6);
\draw[][] (0,-6) -- (0,-7);
\draw[][] (0,-7) -- (1,-7);
\draw[][] (2,-7) -- (1,-6);
\draw[][] (1,-5) -- (2,-6);
\draw[][] (2,-6) -- (1,-6);
\draw[][] (0,-6) -- (1,-7);
\draw[][] (1,-7) -- (2,-7);
\draw[][] (2,-6) -- (3,-7);
\draw[][] (2,-7) -- (3,-7);
\draw[][] (0,-5) -- (-1,-4);
\draw[][] (-1,-4) -- (-2,-4);
\draw[][] (-2,-4) -- (-2,-5);
\draw[][] (-2,-5) -- (-1,-6);
\draw[][] (-2,-5) -- (-2,-6);
\draw[][] (-2,-6) -- (-1,-6);
\draw[][] (-1,-6) -- (0,-6);
\fill[black] (0.55,-7.05) circle (.0001cm) node[align=right, above]{$T_{\alpha}\quad$};
\fill[black] (0.5,-5.9) circle (.0001cm) node[align=right, above]{$T\quad$};
\fill[black] (0.55,-5.05) circle (.0001cm) node[align=right, above]{$T_{\beta}\quad$};
\fill[black] (-1.45,-6.05) circle (.0001cm) node[align=right, above]{$T_{\gamma}\quad$};
\end{tikzpicture}
\qquad
\begin{tikzpicture}[scale=0.6]
\draw[][] (4,-4) -- (4,-5);
\draw[][] (4,-4) -- (5,-5);
\draw[][] (6,-6) -- (5,-6);
\draw[][] (5,-6) -- (5,-5);
\draw[][] (4,-5) -- (5,-6);
\draw[][] (4,-5) -- (4,-6);
\draw[][] (4,-6) -- (5,-6);
\draw[][] (4,-6) -- (4,-7);
\draw[][] (4,-7) -- (5,-7);
\draw[][] (5,-7) -- (4,-6);
\draw[][] (5,-5) -- (6,-6);
\draw[][] (6,-6) -- (5,-6);
\draw[][] (5,-6) -- (6,-7);
\draw[][] (5,-7) -- (6,-7);
\draw[][] (6,-6) -- (7,-7);
\draw[][] (6,-7) -- (7,-7);
\draw[][] (4,-5) -- (3,-4);
\draw[][] (3,-4) -- (2,-4);
\draw[][] (2,-4) -- (2,-5);
\draw[][] (2,-5) -- (3,-6);
\draw[][] (3,-6) -- (4,-6);
\draw[][] (2,-5) -- (2,-6);
\draw[][] (2,-6) -- (3,-6);
\fill[black] (4.55,-7.05) circle (.0001cm) node[align=right, above]{$T_{\alpha}\quad$};
\fill[black] (4.5,-5.9) circle (.0001cm) node[align=right, above]{$T\quad$};
\fill[black] (5.55,-6.05) circle (.0001cm) node[align=right, above]{$T_{\beta}\quad$};
\fill[black] (2.55,-6) circle (.0001cm) node[align=right, above]{$T_{\gamma}\quad$};
\end{tikzpicture}
\qquad
\begin{tikzpicture}[scale=0.6]
\draw[][] (8,-4) -- (8,-5);
\draw[][] (8,-4) -- (9,-5);
\draw[][] (8,-5) -- (9,-6);
\draw[][] (9,-5) -- (9,-6);
\draw[][] (8,-5) -- (8,-6);
\draw[][] (8,-6) -- (9,-7);
\draw[][] (8,-6) -- (8,-7);
\draw[][] (8,-7) -- (9,-7);
\draw[][] (9,-7) -- (9,-6);
\draw[][] (9,-5) -- (10,-6);
\draw[][] (10,-7) -- (9,-6);
\draw[][] (10,-6) -- (10,-7);
\draw[][] (9,-7) -- (10,-7);
\draw[][] (10,-6) -- (11,-7);
\draw[][] (10,-7) -- (11,-7);
\draw[][] (9,-7) -- (9,-8);
\draw[][] (9,-8) -- (10,-9);
\draw[][] (10,-9) -- (11,-9);
\draw[][] (11,-9) -- (11,-8);
\draw[][] (11,-8) -- (10,-7);
\draw[][] (9,-8) -- (9,-9);
\draw[][] (9,-9) -- (10,-9);
\fill[black] (10.55,-7.05) circle (.0001cm) node[align=right, above]{$T_{\alpha}\quad$};
\fill[black] (9.5,-6.9) circle (.0001cm) node[align=right, above]{$T\quad$};
\fill[black] (8.55,-7.05) circle (.0001cm) node[align=right, above]{$T_{\beta}\quad$};
\fill[black] (9.55,-9.05) circle (.0001cm) node[align=right, above]{$T_{\gamma}\quad$};
\end{tikzpicture}
\vspace{1cm}
\begin{tikzpicture}[scale=0.6]
\draw[][] (12,-4) -- (12,-5);
\draw[][] (12,-4) -- (13,-5);
\draw[][] (14,-6) -- (13,-6);
\draw[][] (12,-5) -- (13,-6);
\draw[][] (13,-5) -- (13,-6);
\draw[][] (12,-5) -- (12,-6);
\draw[][] (12,-6) -- (13,-7);
\draw[][] (12,-6) -- (12,-7);
\draw[][] (12,-7) -- (13,-7);
\draw[][] (13,-7) -- (13,-6);
\draw[][] (13,-5) -- (14,-6);
\draw[][] (14,-6) -- (13,-6);
\draw[][] (13,-6) -- (14,-7);
\draw[][] (13,-7) -- (14,-7);
\draw[][] (14,-6) -- (15,-7);
\draw[][] (14,-7) -- (15,-7);
\draw[][] (13,-7) -- (13,-8);
\draw[][] (13,-8) -- (14,-9);
\draw[][] (14,-9) -- (15,-9);
\draw[][] (15,-9) -- (15,-8);
\draw[][] (15,-8) -- (14,-7);
\draw[][] (13,-8) -- (13,-9);
\draw[][] (13,-9) -- (14,-9);
\fill[black] (13.5,-6.9) circle (.0001cm) node[align=right, above]{$T\quad$};
\fill[black] (13.55,-6.05) circle (.0001cm) node[align=right, above]{$T_{\alpha}\quad$};
\fill[black] (12.55,-7.05) circle (.0001cm) node[align=right, above]{$T_{\beta}\quad$};
\fill[black] (13.55,-9.05) circle (.0001cm) node[align=right, above]{$T_{\gamma}\quad$};
\end{tikzpicture}
\qquad
\begin{tikzpicture}[scale=0.6]
\draw[][] (4,-10) -- (4,-11);
\draw[][] (4,-10) -- (5,-11);
\draw[][] (6,-12) -- (5,-12);
\draw[][] (5,-12) -- (5,-11);
\draw[][] (4,-11) -- (5,-12);
\draw[][] (4,-11) -- (4,-12);
\draw[][] (4,-12) -- (5,-12);
\draw[][] (4,-12) -- (4,-13);
\draw[][] (4,-13) -- (5,-13);
\draw[][] (5,-13) -- (5,-12);
\draw[][] (5,-11) -- (6,-12);
\draw[][] (6,-12) -- (5,-12);
\draw[][] (6,-12) -- (6,-13);
\draw[][] (5,-13) -- (6,-13);
\draw[][] (6,-12) -- (7,-13);
\draw[][] (6,-13) -- (7,-13);
\draw[][] (5,-11) -- (5,-10);
\draw[][] (5,-10) -- (6,-10);
\draw[][] (6,-10) -- (7,-11);
\draw[][] (7,-11) -- (7,-12);
\draw[][] (7,-12) -- (6,-12);
\draw[][] (7,-12) -- (8,-12);
\draw[][] (8,-12) -- (7,-11);
\fill[black] (6.55,-13.05) circle (.0001cm) node[align=right, above]{$T_{\alpha}\quad$};
\fill[black] (4.55,-12.05) circle (.0001cm) node[align=right, above]{$T_{\beta}\quad$};
\fill[black] (5.5,-11.9) circle (.0001cm) node[align=right, above]{$T\quad$};
\fill[black] (7.55,-12.05) circle (.0001cm) node[align=right, above]{$T_{\gamma}\quad$};
\end{tikzpicture}
\caption{Possibilities for $T$, when it shares semiuniform faces with three another triangular faces, involving a hexagonal face which $T$ shares with one other triangular face}
\label{fig:Case4}
\end{figure}
\end{center}

\begin{figure}
\centering
\begin{tikzpicture}[scale=1.1]
\draw[][] (0,0) -- (0,-0.5);
\draw[][] (0,-0.5) -- (0.5,-0.5);
\draw[][] (0.5,-0.5) -- (0,0);
\draw[] (0,-0.5) -- (0,-1);
\draw[] (0.5,-0.5) -- (1,-1);
\draw[] (0,-1) -- (0.5,-1.5);
\draw[] (0.5,-1.5) -- (1,-1.5);
\draw[] (1,-1) -- (1,-1.5);
\draw[] (1,-1) -- (1.5,-1.5);
\draw[] (1,-1.5) -- (1.5,-1.5);
\draw[] (0,-1) -- (0,-1.5);
\draw[] (0,-1.5) -- (0.5,-1.5);
\draw[] (0,-1.5) -- (0,-2);
\draw[] (0,-2) -- (0.5,-2.5);
\draw[] (0.5,-2.5) -- (1,-2.5);
\draw[] (1,-2.5) -- (1,-2);
\draw[] (1,-2) -- (0.5,-1.5);
\draw[] (0,-2) -- (0,-2.5);
\draw[] (0,-2.5) -- (0.5,-2.5);
\fill[black] (0.3,-0.55) circle (.0001cm) node[align=right, above]{$T_{\beta}\quad$};
\fill[black] (0.25,-1.5) circle (.0001cm) node[align=right, above]{$T\quad$};
\fill[black] (0.3,-2.55) circle (.0001cm) node[align=right, above]{$T_{\gamma}\quad$};
\fill[black] (1.3,-1.55) circle (.0001cm) node[align=right, above]{$T_{\alpha}\quad$};
\end{tikzpicture}
\qquad
\begin{tikzpicture}[scale=1.1]
\draw[][] (0,0) -- (0,-0.5);
\draw[][] (0,-0.5) -- (0.5,-0.5);
\draw[][] (0.5,-0.5) -- (0,0);
\draw[] (0,-0.5) -- (0,-1);
\draw[] (0.5,-0.5) -- (1,-1);
\draw[] (0,-1) -- (0.5,-1.5);
\draw[] (0.5,-1.5) -- (1,-1.5);
\draw[] (1,-1) -- (1,-1.5);
\draw[] (1,-1) -- (1.5,-1.5);
\draw[] (1,-1.5) -- (1.5,-1.5);
\draw[] (0,-1) -- (0,-1.5);
\draw[] (0,-1.5) -- (0.5,-1.5);
\draw[] (0,-1.5) -- (0,-2);
\draw[] (0,-2) -- (0.5,-2);
\draw[] (0.5,-2) -- (0,-1.5);
\draw[] (0.5,-1.5) -- (1,-2);
\draw[] (0.5,-2) -- (1,-2);
\fill[black] (0.3,-0.55) circle (.0001cm) node[align=right, above]{$T_{\beta}\quad$};
\fill[black] (0.25,-1.5) circle (.0001cm) node[align=right, above]{$T\quad$};
\fill[black] (0.3,-2.05) circle (.0001cm) node[align=right, above]{$T_{\gamma}\quad$};
\fill[black] (1.3,-1.55) circle (.0001cm) node[align=right, above]{$T_{\alpha}\quad$};
\end{tikzpicture}
\qquad
\begin{tikzpicture}[scale=1.1]
\draw[][] (0,0) -- (0,-0.5);
\draw[][] (0,-0.5) -- (0.5,-0.5);
\draw[][] (0.5,-0.5) -- (0,0);
\draw[] (0,-0.5) -- (0,-1);
\draw[] (0.5,-0.5) -- (1,-1);
\draw[] (0,-1) -- (0.5,-1.5);
\draw[] (0.5,-1.5) -- (1,-1.5);
\draw[] (1,-1) -- (1,-1.5);
\draw[] (1,-1) -- (1.5,-1.5);
\draw[] (1,-1.5) -- (1.5,-1.5);
\draw[] (0,-1) -- (0,-1.5);
\draw[] (0,-1.5) -- (0.5,-1.5);
\draw[] (0,-1.5) -- (0,-2);
\draw[] (0,-2) -- (0.5,-2);
\draw[] (0.5,-2) -- (0.5,-1.5);
\draw[] (0.5,-1.5) -- (1,-2);
\draw[] (0.5,-2) -- (1,-2);
\fill[black] (0.3,-0.55) circle (.0001cm) node[align=right, above]{$T_{\beta}\quad$};
\fill[black] (0.25,-1.5) circle (.0001cm) node[align=right, above]{$T\quad$};
\fill[black] (0.8,-2.05) circle (.0001cm) node[align=right, above]{$T_{\gamma}\quad$};
\fill[black] (1.3,-1.55) circle (.0001cm) node[align=right, above]{$T_{\alpha}\quad$};
\end{tikzpicture}
\qquad
\begin{tikzpicture}[scale=1.1]
\draw[][] (0,0) -- (0,-0.5);
\draw[][] (0,-0.5) -- (0.5,-0.5);
\draw[][] (0.5,-0.5) -- (0,0);
\draw[] (0,-0.5) -- (0,-1);
\draw[] (0.5,-0.5) -- (1,-1);
\draw[] (0,-1) -- (0.5,-1.5);
\draw[] (0.5,-1.5) -- (1,-1.5);
\draw[] (1,-1) -- (1,-1.5);
\draw[] (1,-1) -- (1.5,-1.5);
\draw[] (1,-1.5) -- (1.5,-1.5);
\draw[] (0,-1) -- (0,-1.5);
\draw[] (0,-1.5) -- (0.5,-1.5);
\draw[] (0,-1.5) -- (-0.5,-1.5);
\draw[] (-0.5,-1.5) -- (-0.5,-1);
\draw[] (-0.5,-1) -- (0,-1.5);
\draw[] (-0.5,-1) -- (-0.5,-0.5);
\draw[] (-0.5,-0.5) -- (0,-1);
\fill[black] (0.3,-0.55) circle (.0001cm) node[align=right, above]{$T_{\beta}\quad$};
\fill[black] (0.25,-1.5) circle (.0001cm) node[align=right, above]{$T\quad$};
\fill[black] (-0.2,-1.55) circle (.0001cm) node[align=right, above]{$T_{\gamma}\quad$};
\fill[black] (1.3,-1.55) circle (.0001cm) node[align=right, above]{$T_{\alpha}\quad$};
\end{tikzpicture}
\caption{Possibilities for $T$, when it shares semiuniform faces with three other triangular faces, involving a hexagonal face which $T$ shares with two other triangular face}
\label{fig:hexagon_case1}
\end{figure}
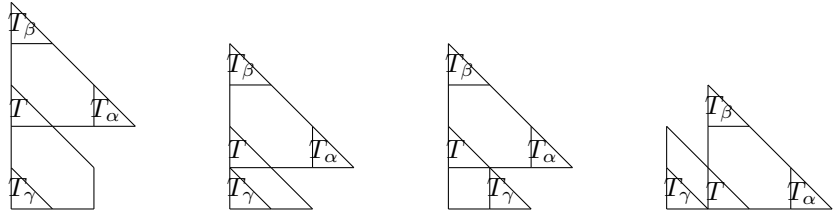

We realize that in all cases in Figure \ref{fig:Case4_1} and \ref{fig:Case4}, $n(T) = 6$, and for all the cases in Figure \ref{fig:hexagon_case1}, $n(T) = 5$. Also none of the $m-4$ triangular faces apart from $T$, $T_{\alpha}$, $T_{\beta}$ and $T_{\gamma}$, can determine the faces determined by $T$, because that would contradict the fact that $T$ can share faces with only three other triangular faces. By Theorem \ref{thm:induction}, for the $m-4$ triangular faces apart from $T$, $T_{\alpha}$, $T_{\beta}$ and $T_{\gamma}$, 

\[ n(T_{1} \cup T_{2} \hdots \cup T_{m-4}) \geq m-4   \]

therefore, 

\[ k = n(T_{1} \cup T_{2} \hdots \cup T_{m}) \geq n(T_{1} \cup T_{2} \hdots \cup T_{m-4}) + n(T) \geq m-4 + 5 = m+1  \]

which again gives a contradiction to the equation \ref{eq:sharp}.

\begin{enumerate}[resume]
    \item If $T$ shares a semiuniform face with exactly four other triangular faces $T_{\alpha}$, $T_{\beta}$, $T_{\gamma}$ and $T_{\phi}$ in $\mathcal{N}$.
\end{enumerate}

all possible cases for $T$, upto symmetry, are listed in Figure \ref{fig:Case5_1} and Figure \ref{fig:Case5}.

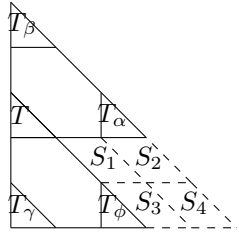
\begin{figure}
\centering
\begin{tikzpicture}[scale=0.6]
\draw[][] (0,0) -- (0,-1);
\draw[][] (0,-1) -- (1,-1);
\draw[][] (0,0) -- (1,-1);
\draw[][] (0,0) -- (0,1);
\draw[][] (0,1) -- (1,1);
\draw[][] (1,1) -- (2,0);
\draw[][] (2,0) -- (2,-1);
\draw[][] (2,-1) -- (1,-1);
\draw[][] (1,-1) -- (0,0);
\draw[][] (0,-1) -- (0,-2);
\draw[][] (0,-2) -- (1,-3);
\draw[][] (1,-3) -- (2,-3);
\draw[][] (2,-3) -- (2,-2);
\draw[][] (2,-2) -- (1,-1);
\draw[][] (1,-1) -- (0,-1);
\draw[][] (2,0) -- (3,-1);
\draw[][] (2,-1) -- (3,-1);
\draw[][] (0,-2) -- (0,-3);
\draw[][] (0,-3) -- (1,-3);
\draw[][] (2,-2) -- (3,-3);
\draw[][] (2,-3) -- (3,-3);
\draw[][] (0,1) -- (0,2);
\draw[][] (0,2) -- (1,1);
\draw[dashed][] (2,-2) -- (3,-2);
\draw[dashed][] (2,-1) -- (3,-2);
\draw[dashed][] (3,-2) -- (4,-2);
\draw[dashed][] (4,-2) -- (3,-1);
\draw[dashed][] (3,-3) -- (4,-3);
\draw[dashed][] (3,-2) -- (4,-3);
\draw[dashed][] (4,-3) -- (5,-3);
\draw[dashed][] (4,-2) -- (5,-3);
\fill[black] (2.35,-1.9) circle (.0001cm) node[align=right, above]{$S_{1}\quad$};
\fill[black] (3.35,-1.9) circle (.0001cm) node[align=right, above]{$S_{2}\quad$};
\fill[black] (3.35,-2.9) circle (.0001cm) node[align=right, above]{$S_{3}\quad$};
\fill[black] (4.35,-2.9) circle (.0001cm) node[align=right, above]{$S_{4}\quad$};
\fill[black] (0.5,-0.9) circle (.0001cm) node[align=right, above]{$T\quad$};
\fill[black] (2.55,-0.95) circle (.0001cm) node[align=right, above]{$T_{\alpha}\quad$};
\fill[black] (0.55,1) circle (.0001cm) node[align=right, above]{$T_{\beta}\quad$};
\fill[black] (0.55,-3.05) circle (.0001cm) node[align=right, above]{$T_{\gamma}\quad$};
\fill[black] (2.55,-3.05) circle (.0001cm) node[align=right, above]{$T_{\phi}\quad$};
\end{tikzpicture}
\caption{The case for $T$, when it shares two hexagonal faces with four other triangular faces}
\label{fig:Case5_1}
\end{figure}

\begin{figure}
\centering
\begin{tikzpicture}[scale=0.6]
\draw[][] (0,0) -- (0,-1);
\draw[][] (0,-1) -- (1,-1);
\draw[][] (0,0) -- (1,-1);
\draw[][] (0,0) -- (-1,1);
\draw[][] (-1,1) -- (-2,1);
\draw[][] (-2,1) -- (-2,0);
\draw[][] (-2,0) -- (-1,-1);
\draw[][] (-1,-1) -- (0,-1);
\draw[][] (0,0) -- (0,1);
\draw[][] (0,1) -- (1,1);
\draw[][] (1,1) -- (2,0);
\draw[][] (2,0) -- (2,-1);
\draw[][] (2,-1) -- (1,-1);
\draw[][] (1,-1) -- (0,0);
\draw[][] (0,-1) -- (0,-2);
\draw[][] (0,-2) -- (1,-3);
\draw[][] (1,-3) -- (2,-3);
\draw[][] (2,-3) -- (2,-2);
\draw[][] (2,-2) -- (1,-1);
\draw[][] (1,-1) -- (0,-1);
\draw[][] (2,0) -- (3,-1);
\draw[][] (2,-1) -- (3,-1);
\draw[][] (-2,1) -- (-2,0);
\draw[][] (-2,-1) -- (-1,-1);
\draw[][] (-2,-1) -- (-2,0);
\draw[][] (0,-2) -- (0,-3);
\draw[][] (0,-3) -- (1,-3);
\draw[][] (0,1) -- (0,2);
\draw[][] (0,2) -- (1,1);
\fill[black] (0.5,-0.9) circle (.0001cm) node[align=right, above]{$T\quad$};
\fill[black] (-1.45,-1) circle (.0001cm) node[align=right, above]{$T_{\phi}\quad$};
\fill[black] (2.55,-1) circle (.0001cm) node[align=right, above]{$T_{\alpha}\quad$};
\fill[black] (0.55,0.95) circle (.0001cm) node[align=right, above]{$T_{\beta}\quad$};
\fill[black] (0.55,-3) circle (.0001cm) node[align=right, above]{$T_{\gamma}\quad$};
\end{tikzpicture}
\qquad
\begin{tikzpicture}[scale=0.6]
\draw[][] (0,0) -- (0,-1);
\draw[][] (0,-1) -- (1,-1);
\draw[][] (0,0) -- (1,-1);
\draw[][] (0,0) -- (0,1);
\draw[][] (0,-1) -- (-1,-1);
\draw[][] (-1,-1) -- (-1,0);
\draw[][] (-1,0) -- (0,-1);
\draw[][] (-1,0) -- (-1,1);
\draw[][] (-1,1) -- (0,0);
\draw[][] (0,1) -- (1,1);
\draw[][] (1,1) -- (2,0);
\draw[][] (2,0) -- (2,-1);
\draw[][] (2,-1) -- (1,-1);
\draw[][] (1,-1) -- (0,0);
\draw[][] (0,-1) -- (0,-2);
\draw[][] (0,-2) -- (1,-3);
\draw[][] (1,-3) -- (2,-3);
\draw[][] (2,-3) -- (2,-2);
\draw[][] (2,-2) -- (1,-1);
\draw[][] (1,-1) -- (0,-1);
\draw[][] (2,0) -- (3,-1);
\draw[][] (2,-1) -- (3,-1);
\draw[][] (0,-2) -- (0,-3);
\draw[][] (0,-3) -- (1,-3);
\draw[][] (0,1) -- (0,2);
\draw[][] (0,2) -- (1,1);
\fill[black] (0.5,-0.9) circle (.0001cm) node[align=right, above]{$T\quad$};
\fill[black] (-0.45,-0.95) circle (.0001cm) node[align=right, above]{$T_{\phi}\quad$};
\fill[black] (2.55,-0.95) circle (.0001cm) node[align=right, above]{$T_{\alpha}\quad$};
\fill[black] (0.55,0.95) circle (.0001cm) node[align=right, above]{$T_{\beta}\quad$};
\fill[black] (0.55,-3.05) circle (.0001cm) node[align=right, above]{$T_{\gamma}\quad$};
\end{tikzpicture}
\qquad
\begin{tikzpicture}[scale=0.6]
\draw[][] (0,0) -- (0,-1);
\draw[][] (0,0) -- (1,-1);
\draw[][] (1,-1) -- (1,-2);
\draw[][] (0,-1) -- (1,-2);
\draw[][] (0,-1) -- (0,-2);
\draw[][] (0,-2) -- (1,-2);
\draw[][] (0,-2) -- (0,-3);
\draw[][] (0,-3) -- (1,-3);
\draw[][] (1,-3) -- (1,-2);
\draw[][] (1,-1) -- (2,-2);
\draw[][] (2,-2) -- (2,-3);
\draw[][] (1,-2) -- (2,-3);
\draw[][] (1,-3) -- (2,-3);
\draw[][] (2,-2) -- (3,-3);
\draw[][] (2,-3) -- (3,-3);
\draw[][] (1,-3) -- (1,-4);
\draw[][] (1,-4) -- (2,-5);
\draw[][] (2,-5) -- (3,-5);
\draw[][] (3,-5) -- (3,-4);
\draw[][] (3,-4) -- (2,-3);
\draw[][] (3,-5) -- (4,-5);
\draw[][] (4,-5) -- (3,-4);
\draw[][] (1,-4) -- (1,-5);
\draw[][] (1,-5) -- (2,-5);
\fill[black] (1.5,-2.9) circle (.0001cm) node[align=right, above]{$T\quad$};
\fill[black] (0.55,-2.05) circle (.0001cm) node[align=right, above]{$T_{\alpha}\quad$};
\fill[black] (2.55,-3.05) circle (.0001cm) node[align=right, above]{$T_{\beta}\quad$};
\fill[black] (3.55,-5.05) circle (.0001cm) node[align=right, above]{$T_{\gamma}\quad$};
\fill[black] (1.55,-5.05) circle (.0001cm) node[align=right, above]{$T_{\phi}\quad$};
\end{tikzpicture}
\qquad
\begin{tikzpicture}[scale=0.6]
\draw[][] (4,0) -- (4,-1);
\draw[][] (4,0) -- (5,-1);
\draw[][] (6,-2) -- (5,-2);
\draw[][] (4,-1) -- (5,-1);
\draw[][] (4,-1) -- (5,-2);
\draw[][] (4,-1) -- (4,-2);
\draw[][] (4,-2) -- (5,-2);
\draw[][] (4,-2) -- (4,-3);
\draw[][] (4,-3) -- (5,-3);
\draw[][] (5,-3) -- (5,-2);
\draw[][] (5,-1) -- (6,-2);
\draw[][] (6,-2) -- (5,-2);
\draw[][] (5,-2) -- (6,-3);
\draw[][] (5,-3) -- (6,-3);
\draw[][] (6,-2) -- (7,-3);
\draw[][] (6,-3) -- (7,-3);
\draw[][] (4,-1) -- (3,0);
\draw[][] (3,0) -- (2,0);
\draw[][] (2,0) -- (2,-1);
\draw[][] (2,-1) -- (3,-2);
\draw[][] (3,-2) -- (4,-2);
\draw[][] (2,-1) -- (2,-2);
\draw[][] (2,-2) -- (3,-2);
\draw[][] (2,0) -- (2,1);
\draw[][] (2,1) -- (3,0);
\fill[black] (5.55,-3.05) circle (.0001cm) node[align=right, above]{$T_{\alpha}\quad$};
\fill[black] (4.5,-1.9) circle (.0001cm) node[align=right, above]{$T\quad$};
\fill[black] (4.55,-1.05) circle (.0001cm) node[align=right, above]{$T_{\beta}\quad$};
\fill[black] (2.55,-2.05) circle (.0001cm) node[align=right, above]{$T_{\gamma}\quad$};
\fill[black] (2.55,-0.05) circle (.0001cm) node[align=right, above]{$T_{\phi}\quad$};
\end{tikzpicture}
\qquad
\begin{tikzpicture}[scale=0.6]
\draw[][] (8,0) -- (8,-1);
\draw[][] (8,0) -- (9,-1);
\draw[][] (10,-2) -- (9,-2);
\draw[][] (8,-1) -- (9,-1);
\draw[][] (9,-1) -- (9,-2);
\draw[][] (8,-1) -- (8,-2);
\draw[][] (8,-2) -- (9,-2);
\draw[][] (8,-2) -- (8,-3);
\draw[][] (8,-3) -- (9,-3);
\draw[][] (9,-3) -- (9,-2);
\draw[][] (9,-1) -- (10,-2);
\draw[][] (10,-2) -- (9,-2);
\draw[][] (10,-2) -- (10,-3);
\draw[][] (9,-3) -- (10,-3);
\draw[][] (10,-2) -- (11,-3);
\draw[][] (10,-3) -- (11,-3);
\draw[][] (9,-1) -- (9,0);
\draw[][] (9,0) -- (10,0);
\draw[][] (10,0) -- (11,-1);
\draw[][] (11,-1) -- (11,-2);
\draw[][] (11,-2) -- (10,-2);
\draw[][] (11,-1) -- (12,-2);
\draw[][] (11,-2) -- (12,-2);
\draw[][] (9,0) -- (9,1);
\draw[][] (9,1) -- (10,0);
\fill[black] (10.55,-3) circle (.0001cm) node[align=right, above]{$T_{\alpha}\quad$};
\fill[black] (9.5,-1.9) circle (.0001cm) node[align=right, above]{$T\quad$};
\fill[black] (8.55,-1) circle (.0001cm) node[align=right, above]{$T_{\beta}\quad$};
\fill[black] (11.55,-2.05) circle (.0001cm) node[align=right, above]{$T_{\gamma}\quad$};
\fill[black] (9.55,-0.05) circle (.0001cm) node[align=right, above]{$T_{\phi}\quad$};
\end{tikzpicture}
\qquad
\begin{tikzpicture}[scale=0.6]
\draw[][] (12,0) -- (12,-1);
\draw[][] (12,0) -- (13,-1);
\draw[][] (14,-2) -- (13,-2);
\draw[][] (12,-1) -- (13,-2);
\draw[][] (13,-1) -- (13,-2);
\draw[][] (12,-1) -- (12,-2);
\draw[][] (12,-2) -- (13,-2);
\draw[][] (12,-2) -- (12,-3);
\draw[][] (12,-3) -- (13,-3);
\draw[][] (13,-3) -- (13,-2);
\draw[][] (13,-1) -- (14,-2);
\draw[][] (14,-2) -- (13,-2);
\draw[][] (13,-2) -- (14,-3);
\draw[][] (13,-3) -- (14,-3);
\draw[][] (14,-2) -- (15,-3);
\draw[][] (14,-3) -- (15,-3);
\draw[][] (13,-1) -- (13,0);
\draw[][] (13,0) -- (14,0);
\draw[][] (14,0) -- (15,-1);
\draw[][] (15,-1) -- (15,-2);
\draw[][] (14,-2) -- (15,-2);
\draw[][] (15,-2) -- (16,-2);
\draw[][] (15,-1) -- (16,-2);
\draw[][] (13,0) -- (13,1);
\draw[][] (13,1) -- (14,0);
\fill[black] (13.55,-3) circle (.0001cm) node[align=right, above]{$T_{\alpha}\quad$};
\fill[black] (13.5,-1.9) circle (.0001cm) node[align=right, above]{$T\quad$};
\fill[black] (12.55,-2) circle (.0001cm) node[align=right, above]{$T_{\beta}\quad$};
\fill[black] (15.55,-2.05) circle (.0001cm) node[align=right, above]{$T_{\gamma}\quad$};
\fill[black] (13.55,-0.05) circle (.0001cm) node[align=right, above]{$T_{\phi}\quad$};
\end{tikzpicture}
\qquad
\begin{tikzpicture}[scale=0.6]
\draw[][] (0,-4) -- (0,-5);
\draw[][] (0,-4) -- (1,-5);
\draw[][] (1,-5) -- (0,-5);
\draw[][] (0,-5) -- (1,-6);
\draw[][] (0,-5) -- (0,-6);
\draw[][] (0,-6) -- (1,-6);
\draw[][] (0,-6) -- (0,-7);
\draw[][] (0,-7) -- (1,-7);
\draw[][] (2,-7) -- (1,-6);
\draw[][] (1,-5) -- (2,-6);
\draw[][] (2,-6) -- (1,-6);
\draw[][] (0,-6) -- (1,-7);
\draw[][] (1,-7) -- (2,-7);
\draw[][] (2,-6) -- (3,-7);
\draw[][] (2,-7) -- (3,-7);
\draw[][] (0,-5) -- (-1,-4);
\draw[][] (-1,-4) -- (-2,-4);
\draw[][] (-2,-4) -- (-2,-5);
\draw[][] (-2,-5) -- (-1,-6);
\draw[][] (-2,-5) -- (-2,-6);
\draw[][] (-2,-6) -- (-1,-6);
\draw[][] (-1,-6) -- (0,-6);
\draw[][] (-2,-4) -- (-2,-3);
\draw[][] (-2,-3) -- (-1,-4);
\fill[black] (0.55,-7) circle (.0001cm) node[align=right, above]{$T_{\alpha}\quad$};
\fill[black] (0.5,-5.9) circle (.0001cm) node[align=right, above]{$T\quad$};
\fill[black] (0.55,-5) circle (.0001cm) node[align=right, above]{$T_{\beta}\quad$};
\fill[black] (-1.45,-6.05) circle (.0001cm) node[align=right, above]{$T_{\gamma}\quad$};
\fill[black] (-1.45,-4.05) circle (.0001cm) node[align=right, above]{$T_{\phi}\quad$};
\end{tikzpicture}
\qquad
\begin{tikzpicture}[scale=0.6]
\draw[][] (4,-4) -- (4,-5);
\draw[][] (4,-4) -- (5,-5);
\draw[][] (6,-6) -- (5,-6);
\draw[][] (5,-6) -- (5,-5);
\draw[][] (4,-5) -- (5,-6);
\draw[][] (4,-5) -- (4,-6);
\draw[][] (4,-6) -- (5,-6);
\draw[][] (4,-6) -- (4,-7);
\draw[][] (4,-7) -- (5,-7);
\draw[][] (5,-7) -- (4,-6);
\draw[][] (5,-5) -- (6,-6);
\draw[][] (6,-6) -- (5,-6);
\draw[][] (5,-6) -- (6,-7);
\draw[][] (5,-7) -- (6,-7);
\draw[][] (6,-6) -- (7,-7);
\draw[][] (6,-7) -- (7,-7);
\draw[][] (4,-5) -- (3,-4);
\draw[][] (3,-4) -- (2,-4);
\draw[][] (2,-4) -- (2,-5);
\draw[][] (2,-5) -- (3,-6);
\draw[][] (3,-6) -- (4,-6);
\draw[][] (2,-5) -- (2,-6);
\draw[][] (2,-6) -- (3,-6);
\draw[][] (2,-4) -- (2,-3);
\draw[][] (2,-3) -- (3,-4);
\fill[black] (4.55,-7) circle (.0001cm) node[align=right, above]{$T_{\alpha}\quad$};
\fill[black] (4.5,-5.9) circle (.0001cm) node[align=right, above]{$T\quad$};
\fill[black] (5.55,-6) circle (.0001cm) node[align=right, above]{$T_{\beta}\quad$};
\fill[black] (2.55,-6.05) circle (.0001cm) node[align=right, above]{$T_{\gamma}\quad$};
\fill[black] (2.55,-4.05) circle (.0001cm) node[align=right, above]{$T_{\phi}\quad$};
\end{tikzpicture}
\qquad
\begin{tikzpicture}[scale=0.6]
\draw[][] (8,-4) -- (8,-5);
\draw[][] (8,-4) -- (9,-5);
\draw[][] (8,-5) -- (9,-6);
\draw[][] (9,-5) -- (9,-6);
\draw[][] (8,-5) -- (8,-6);
\draw[][] (8,-6) -- (9,-7);
\draw[][] (8,-6) -- (8,-7);
\draw[][] (8,-7) -- (9,-7);
\draw[][] (9,-7) -- (9,-6);
\draw[][] (9,-5) -- (10,-6);
\draw[][] (10,-7) -- (9,-6);
\draw[][] (10,-6) -- (10,-7);
\draw[][] (9,-7) -- (10,-7);
\draw[][] (10,-6) -- (11,-7);
\draw[][] (10,-7) -- (11,-7);
\draw[][] (9,-7) -- (9,-8);
\draw[][] (9,-8) -- (10,-9);
\draw[][] (10,-9) -- (11,-9);
\draw[][] (11,-9) -- (11,-8);
\draw[][] (11,-8) -- (10,-7);
\draw[][] (9,-8) -- (9,-9);
\draw[][] (9,-9) -- (10,-9);
\draw[][] (11,-8) -- (12,-9);
\draw[][] (11,-9) -- (12,-9);
\fill[black] (10.55,-7) circle (.0001cm) node[align=right, above]{$T_{\alpha}\quad$};
\fill[black] (9.5,-6.9) circle (.0001cm) node[align=right, above]{$T\quad$};
\fill[black] (8.55,-7) circle (.0001cm) node[align=right, above]{$T_{\beta}\quad$};
\fill[black] (9.55,-9.05) circle (.0001cm) node[align=right, above]{$T_{\gamma}\quad$};
\fill[black] (11.55,-9.05) circle (.0001cm) node[align=right, above]{$T_{\phi}\quad$};
\end{tikzpicture}
\qquad
\begin{tikzpicture}[scale=0.6]
\draw[][] (12,-4) -- (12,-5);
\draw[][] (12,-4) -- (13,-5);
\draw[][] (14,-6) -- (13,-6);
\draw[][] (12,-5) -- (13,-6);
\draw[][] (13,-5) -- (13,-6);
\draw[][] (12,-5) -- (12,-6);
\draw[][] (12,-6) -- (13,-7);
\draw[][] (12,-6) -- (12,-7);
\draw[][] (12,-7) -- (13,-7);
\draw[][] (13,-7) -- (13,-6);
\draw[][] (13,-5) -- (14,-6);
\draw[][] (14,-6) -- (13,-6);
\draw[][] (13,-6) -- (14,-7);
\draw[][] (13,-7) -- (14,-7);
\draw[][] (14,-6) -- (15,-7);
\draw[][] (14,-7) -- (15,-7);
\draw[][] (13,-7) -- (13,-8);
\draw[][] (13,-8) -- (14,-9);
\draw[][] (14,-9) -- (15,-9);
\draw[][] (15,-9) -- (15,-8);
\draw[][] (15,-8) -- (14,-7);
\draw[][] (13,-8) -- (13,-9);
\draw[][] (13,-9) -- (14,-9);
\draw[][] (15,-8) -- (16,-9);
\draw[][] (15,-9) -- (16,-9);
\fill[black] (13.5,-6.9) circle (.0001cm) node[align=right, above]{$T\quad$};
\fill[black] (13.55,-6) circle (.0001cm) node[align=right, above]{$T_{\alpha}\quad$};
\fill[black] (12.55,-7) circle (.0001cm) node[align=right, above]{$T_{\beta}\quad$};
\fill[black] (13.55,-9.05) circle (.0001cm) node[align=right, above]{$T_{\gamma}\quad$};
\fill[black] (15.55,-9.05) circle (.0001cm) node[align=right, above]{$T_{\phi}\quad$};
\end{tikzpicture}
\qquad
\begin{tikzpicture}[scale=0.6]
\draw[][] (4,-10) -- (4,-11);
\draw[][] (4,-10) -- (5,-11);
\draw[][] (6,-12) -- (5,-12);
\draw[][] (5,-12) -- (5,-11);
\draw[][] (4,-11) -- (5,-12);
\draw[][] (4,-11) -- (4,-12);
\draw[][] (4,-12) -- (5,-12);
\draw[][] (4,-12) -- (4,-13);
\draw[][] (4,-13) -- (5,-13);
\draw[][] (5,-13) -- (5,-12);
\draw[][] (5,-11) -- (6,-12);
\draw[][] (6,-12) -- (5,-12);
\draw[][] (6,-12) -- (6,-13);
\draw[][] (5,-13) -- (6,-13);
\draw[][] (6,-12) -- (7,-13);
\draw[][] (6,-13) -- (7,-13);
\draw[][] (5,-11) -- (5,-10);
\draw[][] (5,-10) -- (6,-10);
\draw[][] (6,-10) -- (7,-11);
\draw[][] (7,-11) -- (7,-12);
\draw[][] (7,-12) -- (6,-12);
\draw[][] (7,-12) -- (8,-12);
\draw[][] (8,-12) -- (7,-11);
\draw[][] (5,-10) -- (5,-9);
\draw[][] (5,-9) -- (6,-10);
\fill[black] (6.55,-13) circle (.0001cm) node[align=right, above]{$T_{\alpha}\quad$};
\fill[black] (4.55,-12) circle (.0001cm) node[align=right, above]{$T_{\beta}\quad$};
\fill[black] (5.5,-11.9) circle (.0001cm) node[align=right, above]{$T\quad$};
\fill[black] (7.55,-12.05) circle (.0001cm) node[align=right, above]{$T_{\gamma}\quad$};
\fill[black] (5.55,-9.95) circle (.0001cm) node[align=right, above]{$T_{\phi}\quad$};
\end{tikzpicture}
\caption{Possibilities for $T$, when it shares semiuniform faces with four other triangular faces}
\label{fig:Case5}
\end{figure}
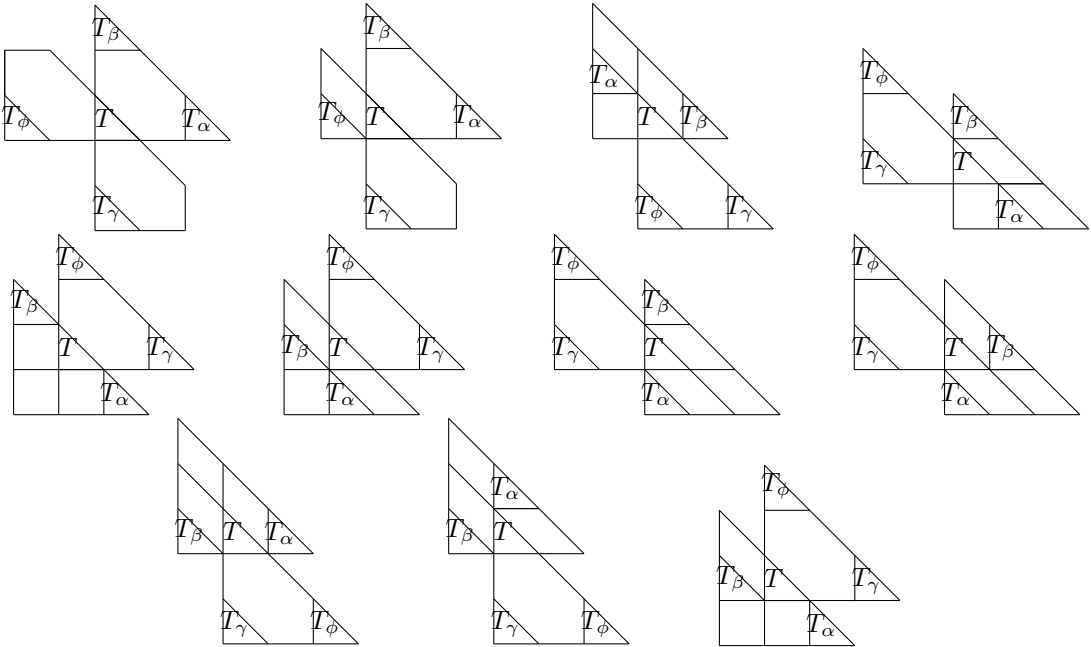

We realize that for the case in Figure \ref{fig:Case5_1}, $n(T) = 5$. However, due to the arrangements of the faces, some of the faces are fixed and are bound to appear in the subdivision, which we show as $S_{1}, S_{2}, S_{3}$ and $S_{4}$ in the Figure \ref{fig:Case5_1}. Amongst these faces, $S_{4}$ is a face which can not be determined by $T, T_{\alpha}, T_{\beta}, T_{\phi}$ and  $T_{\gamma}$. Additionally, we observe that it can also not be determined by any of the remaining $m-5$ triangular faces in $\mathcal{N}$ since it has no free edges which could be adjacent to a triangular face. This implies that the dual point to $S_{4}$ contributes to the count of stable intersections of first kind $k$, although it is not determined by any triangular face in $\mathcal{N}$. This gives a contradiction to the following equality 

\[ k = n(T_{1} \cup T_{2} \hdots \cup T_{m})  \]

in \ref{eq:sharp}.

For the other cases in Figure \ref{fig:Case5}, $n(T) = 6$. None of the $m-5$ triangular faces apart from $T$, $T_{\alpha}$, $T_{\beta}, T_{\gamma}$ and $T_{\phi}$, can determine the faces determined by $T$, because that would contradict the fact that $T$ can share faces with only four other triangular faces. By Theorem \ref{thm:induction}, for the $m-5$ triangular faces apart from $T$, $T_{\alpha}$, $T_{\beta}$ and $T_{\gamma}$, 

\[ n(T_{1} \cup T_{2} \hdots \cup T_{m-5}) \geq m-5   \]

therefore, 

\[ k = n(T_{1} \cup T_{2} \hdots \cup T_{m}) \geq n(T_{1} \cup T_{2} \hdots \cup T_{m-5}) + n(T) \geq m-5 + 6 = m+1  \]

which again gives a contradiction to the equation \ref{eq:sharp}.

The remaining three cases illustrated in Figure \ref{fig:Case6} can also be eliminated by a similar argument, since in all these cases we obtain a semiuniform face $S'$, which cannot be determined by a triangular face, which gives a contradiction to the equation \ref{eq:sharp}.

\begin{figure}[H]
\centering
\begin{tikzpicture}[scale=0.6]
\draw[][] (0,0) -- (0,-1);
\draw[][] (0,-1) -- (1,-1);
\draw[][] (0,0) -- (1,-1);
\draw[][] (0,0) -- (0,1);
\draw[][] (0,1) -- (1,1);
\draw[][] (1,1) -- (2,0);
\draw[][] (2,0) -- (2,-1);
\draw[][] (2,-1) -- (1,-1);
\draw[][] (1,-1) -- (0,0);
\draw[][] (0,-1) -- (0,-2);
\draw[][] (0,-2) -- (1,-3);
\draw[][] (1,-3) -- (2,-3);
\draw[][] (2,-3) -- (2,-2);
\draw[][] (2,-2) -- (1,-1);
\draw[][] (1,-1) -- (0,-1);
\draw[][] (2,0) -- (3,-1);
\draw[][] (2,-1) -- (3,-1);
\draw[][] (0,-2) -- (0,-3);
\draw[][] (0,-3) -- (1,-3);
\draw[][] (2,-2) -- (3,-3);
\draw[][] (2,-3) -- (3,-3);
\draw[][] (0,1) -- (0,2);
\draw[][] (0,2) -- (1,1);
\draw[][] (0,-1) -- (-1,-1);
\draw[][] (-1,-1) -- (-1,0);
\draw[][] (-1,0) -- (0,-1);
\draw[][] (-1,0) -- (-1,1);
\draw[][] (-1,1) -- (0,0);
\draw[dashed][] (2,-2) -- (3,-2);
\draw[dashed][] (2,-1) -- (3,-2);
\draw[dashed][] (3,-2) -- (4,-2);
\draw[dashed][] (4,-2) -- (3,-1);
\draw[dashed][] (3,-3) -- (4,-3);
\draw[dashed][] (3,-2) -- (4,-3);
\draw[dashed][] (4,-3) -- (5,-3);
\draw[dashed][] (4,-2) -- (5,-3);
\fill[black] (4.35,-2.9) circle (.0001cm) node[align=right, above]{$S_{'}\quad$};
\fill[black] (0.5,-0.9) circle (.0001cm) node[align=right, above]{$T\quad$};
\fill[black] (-0.45,-1.05) circle (.0001cm) node[align=right, above]{$T_{\beta}\quad$};
\fill[black] (2.55,-1) circle (.0001cm) node[align=right, above]{$T_{\gamma}\quad$};
\fill[black] (0.55,1.05) circle (.0001cm) node[align=right, above]{$T_{\alpha}\quad$};
\fill[black] (0.55,-2.95) circle (.0001cm) node[align=right, above]{$T_{\theta}\quad$};
\fill[black] (2.55,-2.95) circle (.0001cm) node[align=right, above]{$T_{\lambda}\quad$};
\end{tikzpicture}
\qquad
\begin{tikzpicture}[scale=0.6]
\draw[][] (0,0) -- (0,-1);
\draw[][] (0,-1) -- (1,-1);
\draw[][] (0,0) -- (1,-1);
\draw[][] (0,0) -- (-1,1);
\draw[][] (-1,1) -- (-2,1);
\draw[][] (-2,1) -- (-2,0);
\draw[][] (-2,0) -- (-1,-1);
\draw[][] (-1,-1) -- (0,-1);
\draw[][] (0,0) -- (0,1);
\draw[][] (0,1) -- (1,1);
\draw[][] (1,1) -- (2,0);
\draw[][] (2,0) -- (2,-1);
\draw[][] (2,-1) -- (1,-1);
\draw[][] (1,-1) -- (0,0);
\draw[][] (0,-1) -- (0,-2);
\draw[][] (0,-2) -- (1,-3);
\draw[][] (1,-3) -- (2,-3);
\draw[][] (2,-3) -- (2,-2);
\draw[][] (2,-2) -- (1,-1);
\draw[][] (1,-1) -- (0,-1);
\draw[][] (2,0) -- (3,-1);
\draw[][] (2,-1) -- (3,-1);
\draw[][] (-2,1) -- (-2,0);
\draw[][] (-2,-1) -- (-1,-1);
\draw[][] (-2,-1) -- (-2,0);
\draw[][] (0,-2) -- (0,-3);
\draw[][] (0,-3) -- (1,-3);
\draw[][] (2,-2) -- (3,-3);
\draw[][] (2,-3) -- (3,-3);
\draw[][] (0,1) -- (0,2);
\draw[][] (0,2) -- (1,1);
\draw[dashed][] (2,-2) -- (3,-2);
\draw[dashed][] (2,-1) -- (3,-2);
\draw[dashed][] (3,-2) -- (4,-2);
\draw[dashed][] (4,-2) -- (3,-1);
\draw[dashed][] (3,-3) -- (4,-3);
\draw[dashed][] (3,-2) -- (4,-3);
\draw[dashed][] (4,-3) -- (5,-3);
\draw[dashed][] (4,-2) -- (5,-3);
\fill[black] (4.35,-2.9) circle (.0001cm) node[align=right, above]{$S_{'}\quad$};
\fill[black] (0.5,-0.9) circle (.0001cm) node[align=right, above]{$T\quad$};
\fill[black] (-1.45,-1) circle (.0001cm) node[align=right, above]{$T_{\beta}\quad$};
\fill[black] (2.55,-1) circle (.0001cm) node[align=right, above]{$T_{\gamma}\quad$};
\fill[black] (0.55,1.05) circle (.0001cm) node[align=right, above]{$T_{\alpha}\quad$};
\fill[black] (0.55,-2.95) circle (.0001cm) node[align=right, above]{$T_{\theta}\quad$};
\fill[black] (2.55,-2.95) circle (.0001cm) node[align=right, above]{$T_{\lambda}\quad$};
\end{tikzpicture}
\vspace{0.5cm}
\begin{tikzpicture}[scale=0.6]
\draw[][] (0,0) -- (0,-1);
\draw[][] (0,-1) -- (1,-1);
\draw[][] (0,0) -- (1,-1);
\draw[][] (0,0) -- (-1,1);
\draw[][] (-1,1) -- (-2,1);
\draw[][] (-2,1) -- (-2,0);
\draw[][] (-2,0) -- (-1,-1);
\draw[][] (-1,-1) -- (0,-1);
\draw[][] (0,0) -- (0,1);
\draw[][] (0,1) -- (1,1);
\draw[][] (1,1) -- (2,0);
\draw[][] (2,0) -- (2,-1);
\draw[][] (2,-1) -- (1,-1);
\draw[][] (1,-1) -- (0,0);
\draw[][] (0,-1) -- (0,-2);
\draw[][] (0,-2) -- (1,-3);
\draw[][] (1,-3) -- (2,-3);
\draw[][] (2,-3) -- (2,-2);
\draw[][] (2,-2) -- (1,-1);
\draw[][] (1,-1) -- (0,-1);
\draw[][] (2,0) -- (3,-1);
\draw[][] (2,-1) -- (3,-1);
\draw[][] (-2,1) -- (-2,0);
\draw[][] (-2,-1) -- (-1,-1);
\draw[][] (-2,-1) -- (-2,0);
\draw[][] (0,-2) -- (0,-3);
\draw[][] (0,-3) -- (1,-3);
\draw[][] (2,-2) -- (3,-3);
\draw[][] (2,-3) -- (3,-3);
\draw[][] (-2,2) -- (-1,1);
\draw[][] (-2,2) -- (-2,1);
\draw[][] (0,1) -- (0,2);
\draw[][] (0,2) -- (1,1);
\draw[dashed][] (2,-2) -- (3,-2);
\draw[dashed][] (2,-1) -- (3,-2);
\draw[dashed][] (3,-2) -- (4,-2);
\draw[dashed][] (4,-2) -- (3,-1);
\draw[dashed][] (3,-3) -- (4,-3);
\draw[dashed][] (3,-2) -- (4,-3);
\draw[dashed][] (4,-3) -- (5,-3);
\draw[dashed][] (4,-2) -- (5,-3);
\fill[black] (4.35,-2.9) circle (.0001cm) node[align=right, above]{$S_{'}\quad$};
\fill[black] (0.5,-0.9) circle (.0001cm) node[align=right, above]{$T\quad$};
\fill[black] (-1.45,-1) circle (.0001cm) node[align=right, above]{$T_{\beta}\quad$};
\fill[black] (2.55,-1) circle (.0001cm) node[align=right, above]{$T_{\gamma}\quad$};
\fill[black] (0.55,1.05) circle (.0001cm) node[align=right, above]{$T_{\alpha}\quad$};
\fill[black] (-1.45,1) circle (.0001cm) node[align=right, above]{$T_{\phi}\quad$};
\fill[black] (0.55,-2.95) circle (.0001cm) node[align=right, above]{$T_{\theta}\quad$};
\fill[black] (2.55,-2.95) circle (.0001cm) node[align=right, above]{$T_{\lambda}\quad$};
\end{tikzpicture}
\caption{The cases where $T$ shares faces with five or six other triangular faces}
\label{fig:Case6}
\end{figure}
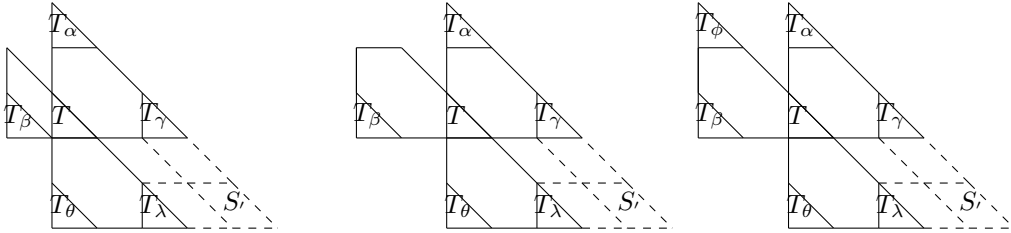

Hence, we completed all cases and we infer that the presence of a triangular face in the relative interior contradicts the sharpness of the bound. Hence, the proof.

\end{proof} 

\begin{remark}
We note that the converse of Theorem \ref{thm:sharp} does not hold true, meaning that if $\mathcal{L}$ is a tropical near-pencil arrangement, then it does not imply that the number of stable intersections equals $n-3$, an example of which is illustrated in Figure \ref{fig:tropical_near_pencil_arrangement}. 
\end{remark}

Now we have established the required setup to state the tropical versions of the de-Bruijn Erd\H os Theorem,

\begin{theorem}[Dual Tropical de Bruijn-Erd\H os Theorem]\label{thm:main_thm}
Let $\mathcal{L}$ be a tropical line arrangement of $n$ $(n \geq 4)$ tropical lines in the plane. Let $b$ denote the number of stable intersections determined by $\mathcal{L}$. Then,

\begin{enumerate}
    \item $b \geq n - 3$
    \item if $b = n - 3$, then $\mathcal{L}$ is a \textbf{tropical near-pencil arrangement}. 
\end{enumerate}
\end{theorem}

With the duality elaborated in ~\ref{thm: duality}, we can now state the main theorem,

\begin{theorem}[Tropical de Bruijn-Erd\H os Theorem]
Let $\mathcal{S}$ denote a set of points in the tropical plane. Let $v$ $(v \geq 4)$ denote the number of points in $\mathcal{S}$, and let $b$ denote the number of stable tropical lines determined by these points. Then,

\begin{enumerate}
    \item $b \geq v - 3$
    \item if $b = v - 3$, then $\mathcal{S}$ forms a \textbf{tropical near-pencil}.
\end{enumerate}
\end{theorem}

\section{Further Perspectives}

There is a rich theory of hyperplane arrangements which has been studied tropically and tropical line arrangements turn out to be the base case. Therefore, a very logical question to ask after studying the results for line arrangements is to look for analogues for general hyperplane arrangements in higher dimensions. Also, the regular subdivisions of Newton polytopes via the Cayley trick, tie up well with mixed subdivisions and tropical line arrangements. Using this, we have verified the bound in Theorem \ref{thm:main_thm} computationally for the case of $n = 4$, by considering all mixed subdivisions of 4 $\cdot \Delta_{2}$, via computing its secondary fan using polymake. However, such computations become non-viable as $n$ increases beyond 5. We also provide scripts to compute the dual Newton subdivision and the number of stable intersections for a given tropical line arrangement. The relevant scripts and code for these computations can be found at the following link

\begin{center}
 \url{https://sites.google.com/view/ayushkumartewari/home}  
\end{center}

In \cite{AF09} a \textbf{type} of a point is defined as follows,

\begin{definition}
A $(n,d)$ \textit{type} is a $n$ tuple $A = (A_{1}, \hdots , A_{n})$ of nonempty subsets of $[d]:= \{ 1,2, \hdots ,d \}$. The $A_{i}$'s are called \textit{coordinates} of $A$, $1, \hdots, n$ are called the \textit{positions} and  $1, \hdots, d$ are called the \textit{directions}.
\end{definition}

which assigns a tuple to each point in the plane based on its location with respect to a collection of hyperplanes, which in our case are lines and so $d = 3$ in this case. It might be interesting to look into the derivation of our results in terms of these types. Figure \ref{fig:oriented_types}, depicts the types corresponding to all the various faces that are present in a linear Newton subdivision. The $'*'$ in the tuples represents a singleton, while the coordinates which have multiple elements may not occur consecutively, but they can be made consecutive, by rearranging the way we count the lines in the arrangement. We can obtain the type for a face $P$ with edge lengths greater than one, by assigning copies of the directions $12,13$ or $23$ depending on the direction of coaxiality of other lines with the vertex of the line dual to $P$. Such an analysis could help in trying to look for generalizations of our results in higher dimensions.

Also there has been a lot of interest in the study of tropical lines present in tropical cubic surfaces owing to the existence of classical results such as the famous 27 lines on a cubic surface, which has provided detailed analysis about lines embedded in surfaces and is explored widely in \cite{PV19} and \cite{JPS19}, and one can try to generalize our results to higher dimensions using techniques from their work.

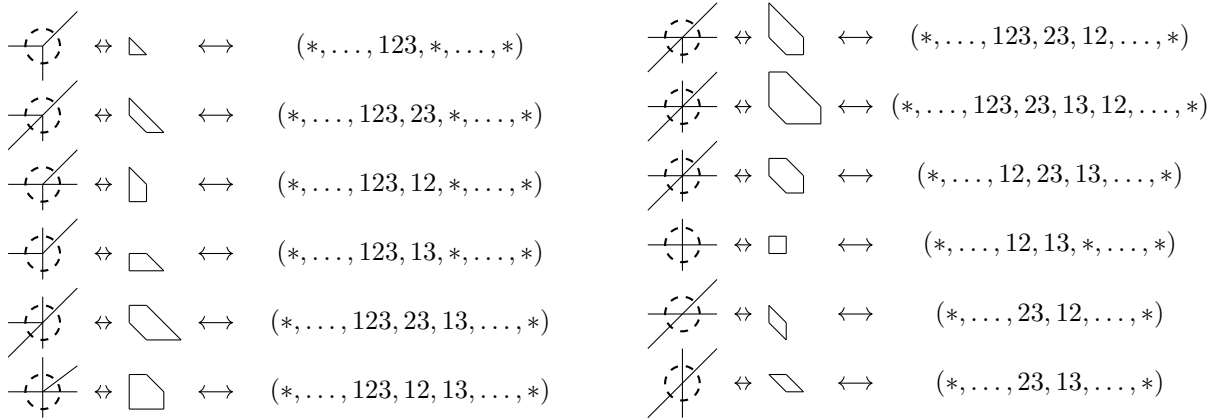
\begin{figure}
\begin{center}
\begin{tikzpicture}[scale=0.23]
\draw[] (0,0) -- (-2,0);
\draw[] (0,0) -- (0,-2);
\draw[] (0,0) -- (2,2);
\draw[thick,dashed] (0,0) circle (1cm);
\draw[<->] (3,0) -- (4,0);
\draw[] (5,0.5) -- (5,-0.5);
\draw[] (5,-0.5) -- (6,-0.5);
\draw[] (6,-0.5) -- (5,0.5);
\draw[<->] (9,0) -- (11,0);
\fill[black] (22,0) circle (.0001cm) node[]{$(*, \hdots , 123, *, \hdots , *)\quad$};
\draw[] (0,-4) -- (-2,-4);
\draw[] (0,-4) -- (0,-5.5);
\draw[] (0,-4) -- (2,-2);
\draw[] (0,-4) -- (-2,-6);
\draw[<->] (3,-4) -- (4,-4);
\draw[] (5,-4) -- (5,-3);
\draw[] (5,-4) -- (6,-5);
\draw[] (6,-5) -- (7,-5);
\draw[] (5,-3) -- (7,-5);
\draw[thick,dashed] (0,-4) circle (1cm);
\draw[<->] (9,-4) -- (11,-4);
\fill[black] (22,-4) circle (.0001cm) node[]{$(*, \hdots , 123, 23, *, \hdots , *)\quad$};
\draw[] (0,-8) -- (-2,-8);
\draw[] (0,-8) -- (0,-9.5);
\draw[] (0,-8) -- (2,-6);
\draw[] (0,-8) -- (2,-8);
\draw[<->] (3,-8) -- (4,-8);
\draw[] (5,-9) -- (5,-7);
\draw[] (6,-8) -- (5,-7);
\draw[] (6,-8) -- (6,-9);
\draw[] (5,-9) -- (6,-9);
\draw[thick,dashed] (0,-8) circle (1cm);
\draw[<->] (9,-8) -- (11,-8);
\fill[black] (22,-8) circle (.0001cm) node[]{$(*, \hdots , 123, 12, *, \hdots , *)\quad$};
\draw[] (0,-12) -- (-2,-12);
\draw[] (0,-12) -- (0,-13.5);
\draw[] (0,-12) -- (2,-10);
\draw[] (0,-12) -- (0,-10.5);
\draw[<->] (3,-12) -- (4,-12);
\draw[] (5,-12) -- (5,-13);
\draw[] (5,-13) -- (7,-13);
\draw[] (5,-12) -- (6,-12);
\draw[] (6,-12) -- (7,-13);
\draw[thick,dashed] (0,-12) circle (1cm);
\draw[<->] (9,-12) -- (11,-12);
\fill[black] (22,-12) circle (.0001cm) node[]{$(*, \hdots , 123, 13, *, \hdots , *)\quad$};
\draw[] (0,-16) -- (-2,-16);
\draw[] (0,-16) -- (0,-17.5);
\draw[] (0,-16) -- (2,-14);
\draw[] (0,-16) -- (0,-14.5);
\draw[] (0,-16) -- (-2,-18);
\draw[<->] (3,-16) -- (4,-16);
\draw[thick,dashed] (0,-16) circle (1cm);
\draw[] (5,-16) -- (5,-15);
\draw[] (5,-15) -- (6,-15);
\draw[] (6,-15) -- (8,-17);
\draw[] (8,-17) -- (6,-17);
\draw[] (6,-17) -- (5,-16);
\draw[<->] (9,-16) -- (11,-16);
\fill[black] (22,-16) circle (.0001cm) node[]{$(*, \hdots , 123, 23, 13, \hdots , *)\quad$};
\draw[] (0,-20) -- (-2,-20);
\draw[] (0,-20) -- (0,-21.5);
\draw[] (0,-20) -- (2,-18.5);
\draw[] (0,-20) -- (0,-18);
\draw[] (0,-20) -- (2,-20);
\draw[<->] (3,-20) -- (4,-20);
\draw[thick,dashed] (0,-20) circle (1cm);
\draw[] (5,-19) -- (5,-21);
\draw[] (5,-21) -- (7,-21);
\draw[] (7,-21) -- (7,-20);
\draw[] (7,-20) -- (6,-19);
\draw[] (6,-19) -- (5,-19);
\draw[<->] (9,-20) -- (11,-20);
\fill[black] (22,-20) circle (.0001cm) node[]{$(*, \hdots , 123, 12, 13, \hdots , *)\quad$};
\end{tikzpicture}
\qquad
\begin{tikzpicture}[scale=0.23]
\draw[] (0,-24) -- (-2,-24);
\draw[] (0,-24) -- (0,-25.5);
\draw[] (0,-24) -- (2,-22);
\draw[] (0,-24) -- (2,-24);
\draw[] (0,-24) -- (-2,-26);
\draw[<->] (3,-24) -- (4,-24);
\draw[thick,dashed] (0,-24) circle (1cm);
\draw[] (5,-24) -- (5,-22);
\draw[] (5,-24) -- (6,-25);
\draw[] (6,-25) -- (7,-25);
\draw[] (7,-24) -- (7,-25);
\draw[] (7,-24) -- (5,-22);
\draw[<->] (9,-24) -- (11,-24);
\fill[black] (22,-24) circle (.0001cm) node[]{$(*, \hdots , 123, 23, 12, \hdots , *)\quad$};
\draw[] (0,-28) -- (-2,-28);
\draw[] (0,-28) -- (0,-29.5);
\draw[] (0,-28) -- (2,-26);
\draw[] (0,-28) -- (0,-26.5);
\draw[] (0,-28) -- (2,-28);
\draw[] (0,-28) -- (-2,-30);
\draw[<->] (3,-28) -- (4,-28);
\draw[thick,dashed] (0,-28) circle (1cm);
\draw[] (5,-28) -- (5,-26);
\draw[] (5,-26) -- (6,-26);
\draw[] (6,-26) -- (8,-28);
\draw[] (8,-28) -- (8,-29);
\draw[] (8,-29) -- (6,-29);
\draw[] (6,-29) -- (5,-28);
\draw[<->] (9,-28) -- (11,-28);
\fill[black] (22,-28) circle (.0001cm) node[]{$(*, \hdots , 123, 23, 13, 12, \hdots , *)\quad$};
\draw[] (-2,-32) -- (2,-32);
\draw[] (2,-30) -- (-2,-34);
\draw[] (0,-30.5) -- (0,-33.5);
\draw[<->] (3,-32) -- (4,-32);
\draw[thick,dashed] (0,-32) circle (1cm);
\draw[] (5,-32) -- (6,-33);
\draw[] (6,-33) -- (7,-33);
\draw[] (7,-33) -- (7,-32);
\draw[] (7,-32) -- (6,-31);
\draw[] (5,-31) -- (6,-31);
\draw[] (5,-32) -- (5,-31);
\draw[<->] (9,-32) -- (11,-32);
\fill[black] (22,-32) circle (.0001cm) node[]{$(*, \hdots , 12, 23, 13, \hdots , *)\quad$};
\draw[] (-2,-36) -- (2,-36);
\draw[] (0,-34.5) -- (0,-37.5);
\draw[<->] (3,-36) -- (4,-36);
\draw[thick,dashed] (0,-36) circle (1cm);
\draw[] (5,-36.5) -- (5,-35.5);
\draw[] (5,-35.5) -- (6,-35.5);
\draw[] (6,-35.5) -- (6,-36.5);
\draw[] (5,-36.5) -- (6,-36.5);
\draw[<->] (9,-36) -- (11,-36);
\fill[black] (22,-36) circle (.0001cm) node[]{$(*, \hdots , 12, 13, *, \hdots , *)\quad$};
\draw[] (-2,-40) -- (2,-40);
\draw[] (-2,-42) -- (2,-38);
\draw[<->] (3,-40) -- (4,-40);
\draw[thick,dashed] (0,-40) circle (1cm);
\draw[] (5,-39.5) -- (5,-40.5);
\draw[] (5,-40.5) -- (6,-41.5);
\draw[] (6,-41.5) -- (6,-40.5);
\draw[] (6,-40.5) -- (5,-39.5);
\draw[<->] (9,-40) -- (11,-40);
\fill[black] (22,-40) circle (.0001cm) node[]{$(*, \hdots , 23, 12, \hdots , *)\quad$};
\draw[] (0,-42) -- (0,-46);
\draw[] (-2,-46) -- (2,-42);
\draw[<->] (3,-44) -- (4,-44);
\draw[thick,dashed] (0,-44) circle (1cm);
\draw[] (5,-43.5) -- (6,-43.5);
\draw[] (5,-43.5) -- (6,-44.5);
\draw[] (6,-43.5) -- (7,-44.5);
\draw[] (6,-44.5) -- (7,-44.5);
\draw[<->] (9,-44) -- (11,-44);
\fill[black] (22,-44) circle (.0001cm) node[]{$(*, \hdots , 23, 13, \hdots , *)\quad$};
\end{tikzpicture}
\caption{All possible shapes of faces present in the Newton subdivision of a tropical line arrangement; with the corresponding type in the tropical oriented matroid on the right}
\label{fig:oriented_types}
\end{center}
\end{figure}

\bibliographystyle{siam}

\bibliography{biblio.bib}

\end{document}